\RequirePackage{fix-cm}
\documentclass[reqno]{amsart}
\usepackage[margin=1.5in,bottom=1.25in]{geometry}	

\usepackage{amsmath}	
\usepackage{amssymb}	
\usepackage{amsfonts}	
\usepackage{amsthm}	
\usepackage[foot]{amsaddr}	

\usepackage{mathtools}	
\mathtoolsset{%
centercolon=true,
}

\usepackage[
cal=cm,
]
{mathalfa}

\usepackage{dsfont}	

\usepackage[proportional,tabular,lining,sf,mono=false]{libertine}

\usepackage[light,scaled=.95]{AlegreyaSans}

\usepackage[T1]{fontenc}	

\usepackage{acronym}	
\renewcommand*{\aclabelfont}[1]{\acsfont{#1}}	

\usepackage[labelfont={bf,small},labelsep=colon,font=small]{caption}	

\usepackage{subcaption}	
\usepackage[svgnames]{xcolor}	

\definecolor{CardinalRed}{HTML}{C41E3A}
\definecolor{Dartmouth}{HTML}{00693E}

\colorlet{MyRed}{CardinalRed}
\colorlet{MyGreen}{Dartmouth}
\colorlet{MyBlue}{DodgerBlue}
\colorlet{MyViolet}{DarkMagenta}

\colorlet{MyLightRed}{MyRed!25}
\colorlet{MyLightGreen}{MyGreen!25}
\colorlet{MyLightBlue}{MyBlue!25}

\colorlet{PrimalColor}{MyBlue}
\colorlet{PrimalFill}{PrimalColor!25}
\colorlet{DualColor}{MyRed}

\colorlet{AlertColor}{MyRed}	
\colorlet{BadColor}{MyRed}	
\colorlet{GoodColor}{MyGreen}	
\colorlet{LinkColor}{MediumBlue}	
\colorlet{RevColor}{black}	
\colorlet{MacroColor}{black}

\usepackage[]{titlesec}	
\titleformat{name=\section}{\medskip}{\thetitle.}{0.8em}{\centering\scshape}

\titleformat{name=\subsection}[runin]{}{\bfseries\thetitle.}{0.5em}{\bfseries}[.]
\titleformat{name=\subsubsection}[runin]{}{\thetitle.}{0.5em}{\bfseries}[.]

\titleformat{name=\paragraph,numberless}[runin]{}{}{0em}{\bfseries}[]
\titlespacing{\paragraph}{0em}{\medskipamount}{1em}

\titleformat{name=\subparagraph,numberless}[runin]{}{}{0em}{}[.]
\titlespacing{\subparagraph}{0em}{0em}{0.5em}

\usepackage{latexsym}	
\usepackage{fontawesome}	
\usepackage{pifont}	

\usepackage{tikz}	
\usetikzlibrary{calc,patterns}	
\usepackage{pgfplots}

\usepackage{array}	
\usepackage{booktabs}	
\usepackage[inline,shortlabels]{enumitem}	
\setlist[1]{topsep=\smallskipamount,itemsep=\smallskipamount,left=\parindent}
\setlist[2]{left=0pt}

\usepackage[kerning=true]{microtype}	

\usepackage{tabto}	
\usepackage{xspace}	

\usepackage[authoryear,sort&compress]{natbib}	

\bibpunct[, ]{[}{]}{,}{}{,}{,}
\setcitestyle{numbers,square}

\usepackage{hyperref}
\hypersetup{
final,
colorlinks=true,
linktocpage=true,
pdfstartview=FitH,
breaklinks=true,
pdfpagemode=UseNone,
pageanchor=true,
pdfpagemode=UseOutlines,
plainpages=false,
bookmarksnumbered,
bookmarksopen=false,
bookmarksopenlevel=1,
hypertexnames=false,
pdfhighlight=/O,
urlcolor=LinkColor,linkcolor=LinkColor,citecolor=LinkColor,	
pdftitle={},
pdfauthor={},
pdfsubject={},
pdfkeywords={},
pdfcreator={pdfLaTeX},
pdfproducer={LaTeX with hyperref}
}

\usepackage{thmtools}	
\usepackage{thm-restate}	

\usepackage[sort&compress,capitalize,nameinlink]{cleveref}	
\crefname{algo}{Algorithm}{Algorithms}
\crefname{assumption}{Assumption}{Assumptions}
\crefname{case}{Case}{Cases}

\makeatletter
\AtBeginDocument
{
	\def\ltx@label#1{\cref@label{#1}}	
	\def\label@in@display@noarg#1{\cref@old@label@in@display{#1}}	
}%
\makeatother

\usepackage{algorithm}	
\usepackage{algpseudocode}	

\usepackage{thm-restate}	

\theoremstyle{plain}
\newtheorem{theorem}{Theorem}	
\newtheorem{corollary}{Corollary}	
\newtheorem{lemma}{Lemma}	
\newtheorem{proposition}{Proposition}	

\newtheorem*{theorem*}{Theorem}	
\newtheorem*{corollary*}{Corollary}	

\theoremstyle{definition}
\newtheorem{definition}{Definition}	
\newtheorem{assumption}{Assumption}	

\newtheorem*{definition*}{Definition}	
\newtheorem*{assumption*}{Assumptions}	
\newtheorem*{example*}{Example}	

\theoremstyle{remark}
\newtheorem{remark}{Remark}	

\newtheorem*{remark*}{Remark}	
\newtheorem*{notation*}{Notation}	

\newcounter{proofpart}

\usepackage{color-edits}	
\usepackage[normalem]{ulem}	

\newcommand{\draft}[1]{{\color{MacroColor}#1}}	

\newcommand{\revise}[1]{#1}	

\newcommand{\newmacro}[2]{\newcommand{#1}{\draft{#2}}}	
\newcommand{\newop}[2]{\DeclareMathOperator{#1}{\draft{#2}}}	

\DeclarePairedDelimiter{\abs}{\lvert}{\rvert}	

\DeclarePairedDelimiterX{\setdef}[2]{\{}{\}}{#1:#2}	
\DeclarePairedDelimiterXPP{\exclude}[1]{\mathopen{}\setminus}{\{}{\}}{}{#1}

\newcommand{\R}{\mathbb{R}}	

\DeclareMathOperator*{\argmax}{arg\,max}	

\DeclarePairedDelimiterXPP{\bigof}[1]{\mathcal{O}}{(}{)}{}{#1}	
\DeclareMathOperator{\crit}{crit}	
\DeclareMathOperator{\dist}{dist}	
\DeclareMathOperator{\one}{\mathds{1}}	
\DeclareMathOperator{\relint}{ri}	

\newcommand{\eg}{e.g.,\xspace}	
\newcommand{\ie}{i.e.,\xspace}	

\newcommand{\alt}[1]{#1'}	
\newcommand{\altalt}[1]{#1''}	

\newmacro{\dd}{\:d}	
\newcommand{\eps}{\varepsilon}	

\newmacro{\const}{c}	
\newmacro{\Const}{\rho}	
\newmacro{\coefalt}{\mu}	
\usepackage{xparse}
\NewDocumentCommand{\coef}{O{\lambda}}{\draft{#1}}
\newmacro{\param}{\theta}	
\newmacro{\params}{\Theta}	

\newmacro{\pexp}{p}	
\newmacro{\qexp}{q}	
\newmacro{\rexp}{r}	

\newmacro{\radius}{r}

\newmacro{\beforestart}{0}	
\newmacro{\start}{1}	
\newmacro{\afterstart}{2}	
\newmacro{\running}{\start,\afterstart,\dotsc}	
\newmacro{\halfrunning}{1,3/2,2\dotsc}	

\newmacro{\run}{t}	
\newmacro{\runalt}{s}	
\newmacro{\runaltalt}{\tau}	
\newmacro{\nRuns}{T}	
\newmacro{\runs}{\mathcal{\nRuns}}	

\newmacro{\state}{x}	
\newmacro{\statealt}{y}	
\newmacro{\statealtalt}{z}	

\newmacro{\tstart}{0}	
\newmacro{\timealt}{s}	
\newmacro{\horizon}{T}	

\newmacro{\pos}{\point}	
\newmacro{\posalt}{y}	

\newmacro{\flowmap}{\Theta}	
\DeclarePairedDelimiterXPP{\flowof}[2]{\flowmap_{#1}}{(}{)}{}{#2}	

\newop{\Nash}{NE}	
\newop{\CE}{CE}	
\newop{\CCE}{CCE}	
\newop{\NI}{NI}	

\newop{\brep}{br}	
\newop{\reg}{Reg}	
\newop{\preg}{\overline{Reg}}	
\newop{\val}{val}	

\newcommand{\eq}{\sol}	

\newmacro{\play}{i}	
\newmacro{\playalt}{j}	
\newmacro{\playaltlalt}{k}	
\newmacro{\nPlayers}{N}	
\newmacro{\players}{\mathcal{\nPlayers}}	

\newmacro{\pure}{\alpha}	
\newmacro{\purealt}{\beta}	
\newmacro{\purealtalt}{\gamma}	
\newmacro{\nPures}{A}	
\newmacro{\pures}{\mathcal{\nPures}}	

\newmacro{\loss}{\ell}	
\newmacro{\pay}{u}	
\newmacro{\payv}{v}	
\newmacro{\pot}{f}	

\newmacro{\game}{\mathcal{G}}	
\newmacro{\gamefull}{\game(\players,\points,\pay)}	

\newmacro{\fingame}{\Gamma}	
\newmacro{\fingamefull}{\Gamma(\players,\pures,\pay)}	

\newmacro{\gmat}{g}	
\newmacro{\gdist}{\dist_{\gmat}}
\newmacro{\mfld}{M}	
\newmacro{\form}{\omega}	

\newmacro{\tvec}{z}	
\newmacro{\uvec}{u}	

\newmacro{\ball}{\mathbb{B}}	
\newmacro{\sphere}{\mathbb{S}}	

\newmacro{\graph}{\mathcal{G}}
\newmacro{\vertices}{\mathcal{V}}
\newmacro{\edges}{\mathcal{E}}

\newmacro{\mat}{A}	
\newmacro{\matalt}{c}	
\newmacro{\hmat}{H}	

\newop{\row}{row}	
\newop{\col}{col}	

\newmacro{\ones}{\mathbf{1}}	
\newmacro{\eye}{I}	
\newmacro{\zer}{\mathbf{0}}	

\DeclarePairedDelimiter{\norm}{\lVert}{\rVert}	
\DeclarePairedDelimiterXPP{\dnorm}[1]{}{\lVert}{\rVert}{_{\ast}}{#1}	

\DeclarePairedDelimiterXPP{\onenorm}[1]{}{\lVert}{\rVert}{_{1}}{#1}	
\DeclarePairedDelimiterXPP{\twonorm}[1]{}{\lVert}{\rVert}{_{2}}{#1}	
\DeclarePairedDelimiterXPP{\supnorm}[1]{}{\lVert}{\rVert}{_{\infty}}{#1}	

\DeclarePairedDelimiterX{\braket}[2]{\langle}{\rangle}{#1,#2}	
\DeclarePairedDelimiterX{\inner}[2]{\langle}{\rangle}{#1,#2}	

\newmacro{\vecspace}{\mathcal{V}}	
\newmacro{\subspace}{\mathcal{W}}	

\newmacro{\coord}{i}	
\newmacro{\coordalt}{j}	
\newmacro{\coordaltalt}{k}	
\newmacro{\nCoords}{n}	
\newmacro{\dims}{\nCoords}	
\newmacro{\vdim}{\nCoords}	

\newmacro{\pvec}{z}	
\newmacro{\pvecalt}{r}	
\newmacro{\bvec}{e}	
\newmacro{\bvecs}{\mathcal{E}}	
\newmacro{\cvec}{b}     
\newmacro{\cvecalt}{d}     

\newmacro{\pspace}{\mathcal{V}}	
\newmacro{\dspace}{\pspace^{\ast}}	

\newmacro{\dvec}{\dpoint}	
\newmacro{\dbvec}{\eps}	

\newmacro{\dpoint}{y}	
\newmacro{\dpointalt}{\alt\dpoint}	
\newmacro{\dpointaltalt}{\altalt\dpoint}	
\newmacro{\dpoints}{\mathcal{Y}}	

\newmacro{\dstate}{Y}	
\newmacro{\dbase}{v}	

\newcommand{\from}{\colon}	

\newop{\Opt}{Opt}	
\newop{\Sol}{Sol}	
\newop{\gap}{Gap}	
\newop{\orcl}{Or}	

\newmacro{\tfun}{f}	
\newmacro{\obj}{F}	
\newmacro{\objalt}{g}	
\newmacro{\sobj}{F}	

\newmacro{\oper}{A}	
\newmacro{\vecfield}{g}	

\newcommand{\sol}[1][\point]{#1^{\ast}}	
\newmacro{\solvec}{\vecfield(\sol)}	
\newmacro{\solpay}{\eq[\payv]}	

\newmacro{\signal}{g}	
\newmacro{\step}{\gamma}	
\newmacro{\learn}{\eta}	

\newmacro{\vbound}{G}	
\newmacro{\lips}{L}	
\newmacro{\strong}{\mu}	
\newmacro{\smooth}{\beta}	

\newop{\cone}{cone}
\newop{\tspace}{T}	
\newop{\tcone}{TC}	
\newop{\dcone}{DC}	
\newop{\ncone}{NC}	
\newop{\pcone}{PC}	
\newop{\hull}{\Delta}	

\newmacro{\cvx}{\mathcal{C}}	
\newmacro{\subd}{\partial}	

\newmacro{\minmax}{\mathcal{L}}	

\newmacro{\minvar}{\point_{1}}	
\newmacro{\minvaralt}{\alt\minvar}	
\newmacro{\minvars}{\points_{1}}	
\newmacro{\minsol}{\sol[\minvar]}	

\newmacro{\maxvar}{\point_{2}}	
\newmacro{\maxvaaltr}{\alt\maxvar}	
\newmacro{\maxvars}{\points_{2}}	
\newmacro{\maxsol}{\sol[\maxvar]}	

\newop{\Eucl}{\Pi}	
\newop{\logit}{\Lambda}	
\newop{\dkl}{KL}	

\newmacro{\hreg}{h}	
\newmacro{\hconj}{\hreg^{\ast}}	
\newmacro{\breg}{D}	
\newmacro{\mprox}{P}	
\newmacro{\mirror}{Q}	
\newmacro{\fench}{F}	
\newmacro{\hstr}{\revise{\kappa}}	
\newmacro{\depth}{H}	
\newmacro{\proxdom}{\points_{\hreg}}	
\newmacro{\zone}{\mathcal{D}}	
\newmacro{\hker}{\theta} 

\DeclarePairedDelimiterXPP{\proxof}[2]{\mprox_{#1}}{(}{)}{}{#2}	

\newmacro{\point}{x}	
\newmacro{\pointalt}{\alt\point}	
\newmacro{\pointaltalt}{\altalt\point}	
\newmacro{\points}{\hilbert}	
\newmacro{\intpoints}{\relint\points}	

\newmacro{\base}{p}	
\newmacro{\basealt}{q}	
\newmacro{\basealtalt}{u}	

\newmacro{\open}{\mathcal{U}}	
\newmacro{\closed}{\mathcal{C}}	
\newmacro{\cpt}{\mathcal{K}}	
\newmacro{\nhd}{\mathcal{U}}	
\newmacro{\nhdalt}{\nhd}	

\newop{\ex}{\mathbb{E}}	
\newop{\prob}{\mathbb{P}}	
\newop{\Var}{Var}	
\newop{\simplex}{\hull}	

\DeclarePairedDelimiterXPP{\exof}[1]{\ex}{[}{]}{}{
 #1}

\DeclarePairedDelimiterXPP{\exwrt}[2]{\ex_{#1}}{[}{]}{}{
 #2}

\DeclarePairedDelimiterXPP{\probof}[1]{\prob}{(}{)}{}{
 #1}

\DeclarePairedDelimiterXPP{\oneof}[1]{\one}{\{}{\}}{}{
 #1}

\newmacro{\sample}{\omega}	
\newmacro{\samples}{\Omega}	

\newmacro{\filter}{\mathcal{F}}	
\newmacro{\probspace}{(\samples,\filter,\prob)}	

\newmacro{\event}{E}       
\newmacro{\eventalt}{H}       
\newmacro{\mean}{\mu}	
\newmacro{\sdev}{\sigma}	
\newmacro{\variance}{\sdev^{2}}	

\newmacro{\proper}{\tau}	

\newmacro{\error}{Z}	
\newmacro{\noise}{U}	
\newmacro{\bias}{b}	
\newmacro{\brown}{W}	

\newmacro{\serror}{\theta}	
\newmacro{\snoise}{\xi}	
\newmacro{\sbias}{\psi}	

\newmacro{\sbound}{M}	
\newmacro{\bbound}{B}	
\newmacro{\noisepar}{\sdev}	
\newmacro{\noisevar}{\variance}	

\addauthor[Vas]{VA}{Crimson}

\newmacro{\hilbert}{\mathcal{H}}
\newmacro{\Lb}{\mathcal{L}}

\newcommand{\Bl}[1]{\ball\mathopen{}\left(#1\right)}
\newcommand{\petito}[1]{o\mathopen{}\left(#1\right)}

\newcommand{\Cb}{\mathcal{C}}
\newcommand{\Fb}{\mathcal{F}}

\newcommand{\Hb}{\mathcal{H}}
\newcommand{\Ib}{\mathcal{I}}
\newcommand{\Jb}{\mathcal{J}}

\newcommand{\Rb}{\mathcal{R}}

\newcommand{\VM}[1]{}

\newcommand{\PM}[1]{}

\newmacro{\klexp}{q}
\newmacro{\auxconst}{Q}

\addauthor[Thodoris]{TGT}{ForestGreen}

\begin{document}


\title
[Non-convex heavy-ball dynamics with Hessian-driven damping]
{Heavy-Ball Dynamics with Hessian-Driven Damping for\\
Non-Convex Optimization under the Łojasiewicz Condition}	

\author
[V.~Apidopoulos]
{Vassilis Apidopoulos$^{c,\ast}$}
\address{$^{c}$\,%
Corresponding author.}
\address{$^{\ast}$\,%
Archimedes, Athena Research Center, 1 Artemidos street, Athens, 15125, Greece.}
\email{vassilis.apid@gmail.com}
\author
[V.~Mavrogeorgou]
{Vasiliki Mavrogeorgou$^{\diamond,\ast}$}
\address{$^{\diamond}$\,%
Department of Mathematics, National \& Kapodistrian University of Athens, Athens, Greece.}
\email{vasiliki.mavrogeorgou@gmail.com}
\author
[T.~G.~Tsironis]
{Theodoros G.~Tsironis$^{\sharp,\ast}$}
\address{$^{\sharp}$\,%
Department of Physics, National \& Kapodistrian University of Athens, Athens, Greece.}
\email{tgtsironis@phys.uoa.gr}



\newacro{LHS}{left-hand side}
\newacro{RHS}{right-hand side}
\newacro{iid}[i.i.d.]{independent and identically distributed}
\newacro{lsc}[l.s.c.]{lower semi-continuous}

\newacro{VI}{variational inequality}
\newacroplural{VI}[VIs]{variational inequalities}

\newacro{DIN}{dynamical inertial Newton-like}
\newacro{HBF}{heavy-ball with friction}

\maketitle
\begin{abstract}
%
%
In this paper, we examine the convergence properties of heavy-ball
\PM{``heavy-ball'' or ``heavy ball''?}
dynamics with Hessian-driven damping in smooth non-convex optimization problems satisfying a \L ojasiewicz condition.
In this general setting, we provide a series of tight, worst-case optimal convergence rate guarantees as a function of the dynamics' friction coefficients and the \L ojasiewicz exponent of the problem's objective function.
Importantly, the linear rates that we obtain improve on previous available rates and they suggest a different tuning of the dynamics' damping terms, even in the strongly convex regime.
We complement our analysis with a range of stability estimates in the presence of perturbation errors and inexact gradient input, as well as an avoidance result showing that the dynamics under study avoid strict saddle points from almost every initial condition.
\end{abstract}

\allowdisplaybreaks	
\acresetall	
\acused{iid}
\acused{LHS}
\acused{RHS}
\maketitle

\section{Introduction}
\label{sec:introduction}

Consider the  minimization problem
\begin{equation}
\label{eq:opt}
\tag{Opt}
\begin{aligned}
\operatorname*{minimize}_{x\in\Hb}
	&\quad
	F(x)
\end{aligned}
\end{equation}
where $\obj\from\Hb\to\R$ is a $C^{2}$-smooth (not necessarily convex) function on some Hilbert space $\Hb$.
Our aim in this paper is to study the asymptotic minimization and convergence properties of the dynamical inertial Newton-like \eqref{eq:DIN} system
\begin{equation}
\label{eq:DIN}
\tag{DIN}
    \ddot{x}(t) +\alpha\dot{x}(t) +\beta\nabla^{2}F(x(t))\dot{x}(t) +\nabla F(x(t)) =0 
\end{equation}
where
$\alpha,\beta\geq0$ are the system's \emph{friction coefficients} \textendash\ \emph{inertial} and \emph{Hessian} respectively. System \eqref{eq:DIN} was introduced and studied in \cite{alvarez2002second} and its origins can be traced back to the seminal work of Polyak \cite{polyak1964some} (see also \cite{antipin1994minimization,alvarez2000minimizing,attouch2000heavy}) for the case $\beta = 0$ \textendash\ in which case \eqref{eq:DIN} is typically referred to as the \acf{HBF} dynamics:
\begin{equation}\label{HBode}\tag{HBF}
	\ddot{x}(t)+\alpha\dot{x}(t)+\nabla F(x(t)) = 0
\end{equation}

The key feature of system \eqref{eq:DIN} is the incorporation of the Hessian-driven friction term $\nabla^{2}F(x(t))\dot{x}(t)$, to the classical heavy ball with friction dynamics \eqref{HBode} (\cite{polyak1964some,antipin1994minimization,alvarez2000minimizing,attouch2000heavy}).
The consideration of such Hessian damping term makes the generated trajectory more adaptive to the local curvature of the objective function $F$ and mitigates the oscillations that may appear in the behavior of the heavy ball with friction trajectory in \eqref{HBode}, which may have undesirable effects from an optimization objective.

One of the major interests of studying inertial systems like \eqref{eq:DIN} is their accelerated convergence properties (e.g. in terms of speed of convergence of $F(x(t))$ to a critical value $\bar{F}$), with respect to first-order in time systems like the standard \textit{gradient flow}: 
\begin{equation}\label{GF}\tag{GF}
	\dot{x}(t) + \nabla F(x(t))=0.
\end{equation}
by properly tuning the friction parameters $\alpha$ and $\beta$.




In general, the quantitative properties of systems like \eqref{eq:DIN} and more precisely the question concerning their rates of convergence  (and thus their acceleration properties) strongly depends on i) the precise optimization framework (i.e. the assumptions made on the geometry/regularity $F$) and ii) the adequate tuning of the friction parameters $\alpha$ and $\beta$, in order to achieve the best possible convergence results.  In fact, the question of acceleration for the system \eqref{eq:DIN} beyond the (strongly) convex realm, under milder assumptions such as the \emph{\L ojasiewicz condition} or \emph{Kurdyka-\L ojasiewicz condition} \cite{polyak1963gradient,Loja63,Kurdyka1998} still remains unclear and is the main subject of this work.



\paragraph{Contributions}

In this work, we aim to address the aforementioned issue for the system \eqref{eq:DIN}. In particular, we study the convergence properties of \eqref{eq:DIN} in a non-convex setting (when $F$ is not necessarily convex), under the \L ojasiewicz condition \cite{Loja63,polyak1964some} (see Assumption \ref{definition PL} in Section \ref{sec:prelims}).

More precisely, when $F$ satisfies the $\mu$-\L ojasiewicz condition of order $2$, we provide explicit exponential rates of convergence for $F(x(t))$ to a critical value $\bar{F}$ and for the decay of $\nabla F(x(t))$ to zero, as a function of the friction parameters $\alpha$ and $\beta$. As a by-product, we provide an explicit tuning for the friction parameters, leading to an optimal worst-case rate of order $\bigof{e^{-2\sqrt{\mu}t}}$, which is improved with respect to previous related works \cite{alvarez2002second,attouch2022first,castera2021inertial,maulen2024inertial,wang2025accelerated} for the system \eqref{eq:DIN}. Our theoretical improvements are also witnessed in some simple numerical illustrations.

We additionally provide stability estimates for the trajectory of \eqref{eq:DIN} in the presence of an external perturbation error, that depend both on the local landscape of the objective function $F$, around its critical points (i.e. the \L ojasiewicz exponent of $F$) and the tuning of the parameters $\alpha$, $\beta$. Such stability results present some interest on their own when it comes to the consideration of stochastic versions of \eqref{eq:DIN} where the computation of $\nabla F$ (or $\nabla^{2}F$) is not (or hardly) accessible.

Furthermore, we complement the analysis by establishing an almost-surely avoidance of strict saddle points result, for functions satisfying the \L ojasiewicz condition, further generalizing upon existing works \cite{castera2023inertial} (see also \cite{maulen2024inertial}) for functions with isolated critical points. 
Finally, we extend the above-mentioned results to the first-order in time coupled system \eqref{HBHDfirstorder}, which is related to \eqref{eq:DIN} (see below in paragraph \ref{subsection: DIN}), paving the way to the consideration of more generalized settings, where the objective function $F$ is not necessarily smooth or even differentiable.

\subsection{Related work}

\subsubsection{Heavy-ball with friction}
System \eqref{HBode} was firstly introduced and studied in the seminal work of Polyak \cite{polyak1964some}, showing its advantages with respect to the gradient flow dynamics \eqref{GF}. 
In particular, for twice differentiable and $\mu$-strongly convex functions the \eqref{HBode} system with $\alpha=2\sqrt{\mu}$ enjoys the worst-case optimal linear rate of convergence of order $\bigof{e^{-2\sqrt{\mu}t}}$, in terms of objective function values $F(x(t)) - \min F$, while the corresponding rate for \eqref{GF} is $\bigof{e^{-2\mu t}}$. For the majority of practical functions, it is most likely that the strong convexity parameter is very small (\ie $\mu \ll 1$), making \eqref{HBode} much faster than \eqref{GF}. In this viewpoint the order $\bigof{e^{-c\sqrt{\mu}t}}$, for some constant $c\leq 2$ is usually identified as an accelerated linear rate. Beyond the (strongly) convex setting, several studies have been devoted to the convergence properties of \eqref{HBode}, including i) almost sure convergence to local minima for (non convex) Morse functions \cite{attouch2000heavy,goudou2009gradient}  and ii) convergence rates for the heavy ball trajectories under various relaxations of strong convexity, such as quasi-strong convexity \cite{siegel2019accelerated,aujol2022convergencequasi,aujol2024heavy,aujol2023convergencePL} and  \L ojasiewicz conditioning (see Assumption \hyperref[PL]{$\Lb(2)$}) \cite{haraux1998convergence,begout2015damped,polyak2017lyapunov,apidopoulos2022convergence}. However, it is worth mentioning that, to the best of the authors knowledge, the worst-case optimal rate $\bigof{e^{-2\sqrt{\mu}t}}$, obtained by \eqref{HBode} for strongly convex functions, has not been established for any of the aforementioned relaxations (see also the related discussion in Remark \ref{remark comparison} in Section \ref{sec:results}).

\subsubsection{Dynamic Inertial Newton-like system}\label{subsection: DIN}
While the heavy ball system \eqref{HBode} enjoys some nice convergence properties, its trajectory often presents oscillatory behavior, which may cause numerical instabilities and affect the overall efficiency of the method.
In order to overcome this issue, in \cite{alvarez2002second} the authors proposed the system \eqref{eq:DIN}, which incorporates to the plain heavy ball system \eqref{HBode}, an additional geometric damping term that captures the local curvature of the function (Hessian driven damping).

A seemingly inconvenience of \eqref{eq:DIN} is its second order character in space (due to the presence of the Hessian term $\nabla^{2}F(x(t))$), which may seem restricting for functions that are not necessarily twice differentiable.
However, one of the key features concerning system \eqref{eq:DIN}, is that it is strongly related with a first-order coupled system (in time and space) with no occurrence of the Hessian, that reads as follows:
\begin{equation}
	\begin{cases}{}\label{HBHDfirstorder1}\tag{g-DIN}
		\dot{x}(t) +\beta\nabla F(x(t)) + \left(\alpha-\frac{1}{\beta}\right)x(t) + \frac{1}{\beta}y(t) = 0 \\
		\dot{y}(t) + \left(\alpha-\frac{1}{\beta}\right)x(t) + \frac{1}{\beta}y(t) = 0 \\
		x(0)=x_{0} ~ , \quad y(0)= (1-\alpha\beta)x_{0}  - \beta^{2}\nabla F(x_{0})
	\end{cases}
\end{equation}
In particular, as it was shown in \cite[Theorem $6.1$]{alvarez2002second}, systems \eqref{eq:DIN} and \eqref{HBHDfirstorder1} are equivalent (under an additional regularity assumption on $y(\cdot)$).
System \eqref{HBHDfirstorder1} does not employ any second order information and thus is well defined for a broader class of functions, without requiring second-order differentiability of $F$.
Going even further system \eqref{HBHDfirstorder1} can be generalized for non-differentiable convex functions involving the subdifferential instead of the gradient, or even to non-convex settings by considering the Clarke's generalized subdifferential (see \cite[Definition $10.3$]{clarke2013functional}), see \eg \cite{attouch2012second,castera2021inertial}.

In \cite{alvarez2002second} the authors established some basic results concerning systems \eqref{eq:DIN} and \eqref{HBHDfirstorder1} (and its generalized version involving a non-differential, convex objective function), including existence of solutions, (weak) convergence of trajectories to a critical point when $F$ is (either convex, or) analytic, as also the exact equivalence relation between \eqref{eq:DIN} and \eqref{HBHDfirstorder1}.

Furthermore, in the recent works \cite{castera2021inertial} (see also \cite{maulen2024inertial}) the authors study the generalized version of system \eqref{HBHDfirstorder1} for possibly non-convex, non-differentiable, Lipschitz and tame functions.
Under the $\mu$-\L ojasiewicz conditioning of order $2$, they provide an abstract rate of convergence for $F(x(t))$, to a critical value $\bar{F}$ of order $\bigof{e^{-c\mu t}}$, for some $c\leq 2$ and thus does not constitute an accelerated rate (see \cite[Theorem $13$]{castera2023inertial}). Under the additional assumption of strong convexity for the objective function $F$, the authors in \cite{attouch2022first} derive a rate of convergence of order $\bigof{e^{-\sqrt{\frac{\mu}{2}}t}}$ for a particular choice of $\alpha$ and $\beta$, which however is not tight (see also Remark \ref{rem:optimality rates} in Section \ref{sec:results}). In addition, in the same (strongly convex) setting, the aforementioned rates are proven to be robust with respect to exogenous additive errors to the gradient of $F$ in \cite{attouch2021effect}. In a  recent preprint \cite{wang2025accelerated}, the authors manage to improve the rate to a  $\bigof{e^{-2\sqrt{\mu}t}}$ in the strongly convex setting, which however is proven for a rescaled version of \eqref{eq:DIN}. 
Our findings further generalize and improve upon these works by i) considering a much broader setting without requiring (strong) convexity, but only the $\mu$-\L ojasiewicz condition of order $2$ and ii) providing explicit rate of convergence and stability estimates depending on the parameters $\alpha$ and $\beta$, which is also optimal for an adequate choice of these parameters.  

 Finally in \cite{castera2023inertial} (see also \cite{maulen2024inertial}) the authors provide a result showing that the trajectory generated by \eqref{HBHDfirstorder1} almost surely avoids the set of strict saddle point of Morse functions (see \cite[Theorem \& Corollary $3.1$]{castera2024continuous} and the associated discussion). We extend these results beyond the Morse setting to also account for the avoidance of connected components of saddle points, which are often a feature of relevant objective functions, stemming, for example, from the symmetries of the optimization task \cite{simsek2021geometry}.

\subsection{Organization}
The rest of the paper is organized as follows. In Section \ref{sec:prelims}, we set up the main assumptions and we recall some preliminary results regarding system \eqref{eq:DIN}. In Section \ref{sec:results}, we provide the main results of the paper. Section \ref{sec:analysis} contains the related convergence analysis and all the associated proofs.
In Section \ref{sec:numerics}, we present some numerical experiments, supporting our theoretical findings. Finally the proofs of all the auxiliary results can be found in Appendices \ref{app:aux}, \ref{app:rates} and \ref{app:avoidance}.

\section{Preliminaries}
\label{sec:prelims}

In this section we give some basic definitions and tools that will be used throughout this work. We recall that we are interested in the following optimization problem
\begin{equation}\label{minimization problem}
    \underset{x\in\Hb}{\min}F(x)
\end{equation}
where $F:\Hb\to\R$ denotes the objective function that we want to minimize and $\Hb$ is a Hilbert space endowed with a scalar product $\inner{\cdot}{\cdot}$ and its associated norm $\norm{\cdot}=\sqrt{\inner{\cdot}{\cdot}}$.

\subsection{Main assumptions}
\begin{assumption}
\label{asm:obj}
\label{assumption basic F}
We make the following assumptions on the function $F:\Hb\to\R$ :
\begin{enumerate}
\item The function $F:\Hb\to\R$ is twice continuously differentiable.
\item $F$ is bounded from below and $\crit{F}=\left\{x\in \Hb~:~\nabla F(x)=0\right\}\neq \emptyset$. 
\end{enumerate}
\end{assumption}

Note that  Assumption \ref{assumption basic F} guarantees the existence of a solution to the Cauchy initial value problem \eqref{eq:DIN} with some initial conditions for $x(0)$ and $\dot{x}(0)$, which may not be unique (by Peano's existence Theorem, see e.g. \cite{teschl2012ordinary}). By further assuming that $\nabla^{2}F$ is locally Lipschitz, one can assure the uniqueness of the solution. However, since we are only interested in the (asymptotic) convergence properties of a trajectory, we will not employ the Lipschitz character of the Hessian, unless otherwise stated.


Throughout the paper, we also assume that the function $F$ satisfies the \textit{\L ojasiewicz} condition, as defined below:
\begin{assumption}
\label{asm:KL}
\label{definition PL}
The function $F$ satisfies the \L ojasiewicz condition of order $q\in(1,2]$, i.e. for every $\bar{x}\in \crit{F}$, there exists $\Omega\subset \Hb$ and some positive constant $\mu>0$ such that:
\begin{equation}\label{PL}
(\forall x\in \Omega) \quad : \qquad \abs{F(x)-F(\bar{x})} \leq \frac{1}{2\mu}\norm{\nabla F(x)}^{q} \tag{$\Lb(q)$}
\end{equation}
We refer to \ref{PL} as \textit{global} if $\Omega =\Hb$ and as \textit{local} if there exist $\delta>0$ and $r>0$ such that $\Omega = \Bl{\bar{x},\delta}\cap[F < r]$.
\end{assumption}


The \L ojasiewicz condition \ref{PL} was first introduced in \cite{Loja63} (see also \cite{polyak1964some,Loja93}), as a property of analytic functions and was later extended in \cite{Kurdyka1998} to definable functions in o-minimal sets - also known as \emph{Kurdyka-\L ojasiewicz} condition - as also to non-smooth functions \cite{bolte2007lojasiewicz,bolte2007clarke,bolte2010characterizations}.
Initially, it was used as a key condition to guarantee convergence of the trajectory for classical dynamical systems (such as the gradient flow \cite{Loja63}) and it was later explored further in the optimization literature both in convex \cite{garrigos2017convergence,bolte2017error,bolte2007lojasiewicz,bolte2010characterizations,aujol2023convergencePL,aujol2023fast} and non-convex settings \cite{polyak1963gradient,polyak1964some,polyak2017lyapunov}, \cite{attouch2010proximal,attouch2013convergence,begout2015damped,Karimi2016,necoara2019linear,drusvyatskiy2018error,zhang2020new}. 

Condition \hyperref[PL]{$\Lb(2)$} is also frequently met in the related literature under the name \textit{Polyak-\L ojasiewicz} condition \cite{polyak1963gradient,polyak1964some,Karimi2016} (when $\Omega =\Hb$), as also \textit{gradient domination} (see e.g. \cite[paragraph $4.1.3$]{nesterov2013introductory}), or \textit{Kurdyka - \L ojasiewicz condition} \cite{Kurdyka1998,Bolte2006,bolte2007lojasiewicz,bolte2010characterizations}.  It is of particular interest, since it is closely related to the establishment of linear (geometric) convergence for continuous-in-time and discrete dynamics (algorithms) both in convex and non-convex settings \cite{polyak1963gradient,polyak1964some,Karimi2016,bolte2017error,garrigos2017convergence}.  In fact condition \hyperref[PL]{$\Lb(2)$} can be seen as a relaxation of the standard notion of the strong convexity property of a function. Note that unlike strong convexity, condition \hyperref[PL]{$\Lb(2)$} is a local condition and does not imply neither convexity nor uniqueness of minimizers. 
Over the last years a lot of effort has been made to deduce convergence results for functions that are not strongly convex, but rather satisfy a weaker geometric condition. Indeed, while strong convexity is a potent property that yields powerful convergence results for first-order methods (such as strong convergence and linear rates of convergence), in most of the situations, it is not met in practice. Among others, some of the relaxations of the strong convexity property, include the notion of \textit{quasi-strong convexity} (see e.g. \cite[Definition $1$]{necoara2019linear}), the \textit{restricted secant inequality} \cite[Definition $2$]{zhang2013gradient}, the \textit{error bound condition} \cite[Assumption A]{luo1993error}, the \textit{quadratic growth condition} \cite{penot1996conditioning,anitescu2000degenerate}. For a more detailed discussion of the above, we refer the reader to \cite{bolte2007lojasiewicz,bolte2010characterizations,bolte2017error}, \cite{necoara2019linear,Karimi2016,garrigos2017convergence} and \cite{zhang2020new,apidopoulos2022convergence} and references therein.



\subsection{Basic properties of the dynamical system}

To streamline the analysis to come, we recall here some basic properties of the system \eqref{eq:DIN} as originally established in \cite{alvarez2002second} (see Theorem \& Corollary $2.1$ and Theorem $4.1$). 

\PM{Minor rewriting needed here.}
\begin{proposition}[\cite{alvarez2002second}]\label{proposition basic}
	Let $\obj\from\Hb\to\R$ be a twice differentiable function satisfying \cref{assumption basic F} and let $x(t)$, $t\geq 0$, be a solution trajectory of \eqref{eq:DIN}. Then the following properties hold:
	\begin{enumerate}[label=\roman*)]
		\item If $\nabla^{2}F(x)$ is locally Lipschitz continuous, there exists a unique global solution $(x(t))_{t\geq 0}$ satisfying \eqref{eq:DIN}, with $\dot{x}(\cdot)$ and $\nabla F(x(\cdot))$ belonging to $L^{2}(0,+\infty)$ and $\lim\limits_{t\to\infty}F(x(t))\in \R$. 
		\item If $(x(t))_{t\geq 0}$ is bounded, then
		$\lim\limits_{t\to\infty}\nabla F(x(t)) = \lim\limits_{t\to\infty}\dot{x}(t)=\lim\limits_{t\to\infty}\ddot{x}(t) = 0$
		\item If in addition the function $F$ satisfies the \L ojasiewicz condition \ref{PL}, then the solution-trajectory converges towards a critical point of $F$.
	\end{enumerate} 
	
\end{proposition}


\section{Main results}
\label{sec:results}

In this section we present and discuss the main results of this work, concerning the convergence behavior of the solution-trajectory of the system \eqref{eq:DIN} and its first-order variant \eqref{HBHDfirstorder1}. 

\VM{Slightly restating: In this Section we present and discuss the main results of this work, regarding the convergence behavior of the solution-trajectory of the system \eqref{eq:DIN} and its first-order variant \eqref{HBHDfirstorder1}. More precisely, Theorem \ref{basicteorates} provides asymptotic linear rates of convergence for the objective function $F(x(t))$ to a limiting critical value $F(\bar{x})$ and guarantees the decay of $\nabla F(x(t))$ to zero. Lastly, provides an explicit formula of the geometric convergence rate  $\Rb(\alpha,\beta)$, for the gap $F(x(t))-F(\bar{x})$, as also for $\norm{\nabla F(x(t))}$, depending on the friction parameters $\alpha$ and $\beta$. Note that we are always under the assumption that $F$ satisfies condition \ref{PL} with $q=2$.}


Theorem \ref{basicteorates} provides asymptotic linear rates of convergence for the objective function $F(x(t))$ to a limiting critical value $F(\bar{x})$ and the decay of $\nabla F(x(t))$ to zero, when the function $F$ satisfies condition \hyperref[PL]{$\mathcal{L}(2)$}.

\PM{TODO: slight rewrite.}
\begin{theorem}\label{basicteorates}
Suppose $\obj$ satisfies \cref{asm:obj,asm:KL} with $q=2$.
Let also $\left(x(t)\right)_{t\geq 0}$ be the solution-trajectory of \eqref{eq:DIN}, with $\lim\limits_{t\to\infty}x(t)=\bar{x}\in \crit{F}$.
Let also $t_{0}>0$ such that $\left(x(t)\right)_{t\geq t_{0}} \in \Omega$, where $\Omega$ is a neighborhood of $\bar{x}$, that $F$ satisfies \hyperref[PL]{$\mathcal{L}(2)$} with $\mu>0$. Then for all $t\geq t_{0}$ the following hold:
\begin{equation}
	F(x(t))-F(\bar{x}) \leq \left(F(x(t_{0}))-F(\bar{x}) + \frac{1}{2C(\alpha,\beta)}\norm{\beta\nabla F(x(t_{0}))+\dot{x}(t_{0})}^2\right)e^{-\Rb(\alpha,\beta)(t-t_{0})}
\end{equation}
where 
\begin{equation}
\label{rate factor on alphabeta}
\Rb(\alpha,\beta)
    = \begin{dcases}
        2\alpha
            & \text{ if } ~   \left\{\alpha < \sqrt{\mu}\right\}\&  \left\{\beta\in\left(\frac{\alpha}{\mu},\frac{1}{\alpha}\right)\right\}
            \\
        \frac{1+\alpha\beta - \auxconst(\alpha,\beta)}{2\beta}
            \qquad   & \text{ if } ~   \left\{  \alpha > 0\right\} \& \left\{\beta\in \left(0, \frac{\alpha}{\mu}\right]\cup \left[\frac{1}{\alpha},+\infty\right) \right\} 
    \end{dcases}
\end{equation}
and
\begin{equation}
\label{constant on alphabeta}
C(\alpha,\beta) =  \begin{dcases}
			1-\alpha\beta   & \text{ if } ~   \left\{\alpha < \sqrt{\mu}\right\}\&  \left\{\beta\in\left(\frac{\alpha}{\mu},\frac{1}{\alpha}\right)\right\} 
            \\
			\frac{\left(1+\alpha\beta + \auxconst(\alpha,\beta)\right)}{2}   \qquad   & \text{ if } ~   \left\{  \alpha > 0\right\} \& \left\{\beta\in \left(0, \frac{\alpha}{\mu}\right]\cup \left[\frac{1}{\alpha},+\infty\right) \right\} 
		\end{dcases}
	\end{equation}
and $\auxconst(\alpha,\beta)$ is given by \begin{equation}\label{eq: q expression}
    \auxconst(\alpha,\beta):= \sqrt{\left((\alpha-\mu\beta)\beta+1\right)^{2}-4(1-\alpha\beta)\mu\beta^{2}} -\mu\beta^{2}
\end{equation}   
If in addition, $F(x(t)) \geq F(\bar{x})$, for all $t\geq t_{0}$, then for any $\varepsilon>0$
\PM{TODO: remove the integral, remove $\tilde\Rb$ and restate accordingly, being careful with open vs. closed intervals of validity.
Transfer stuff to remarks.}
	\begin{equation}\label{estimate grad basic teo}
		\int_{t_{0}}^{+\infty}e^{(\Rb(\alpha,\beta)-\varepsilon)t}\norm{\nabla F(x(t))}^2dt <+\infty
	\end{equation} 
	In particular, $\norm{\nabla F(x(t))}=\petito{e^{-\frac{\Rb(\alpha,\beta)-\varepsilon}{2}t}}$, as $t\to+\infty$. 
\end{theorem}


\begin{remark}
	Theorem \ref{basicteorates}, provides an explicit formula of the geometric convergence rate  $\Rb(\alpha,\beta)$, for the gap $F(x(t))-F(\bar{x})$, as also for $\nabla F(x(t))$, depending on the friction parameters $\alpha$ and $\beta$. 
    In Figure \ref{Figure2}, we provide an illustration of  the convergence exponent $\Rb(\alpha,\beta)$ expressed in \eqref{rate factor on alphabeta} as a function of $\alpha$ and $\beta$, which gives a better grasp of its behavior.
\end{remark}


\begin{figure}[t]
\begin{center}
\scalebox{0.41}[0.34]{\includegraphics{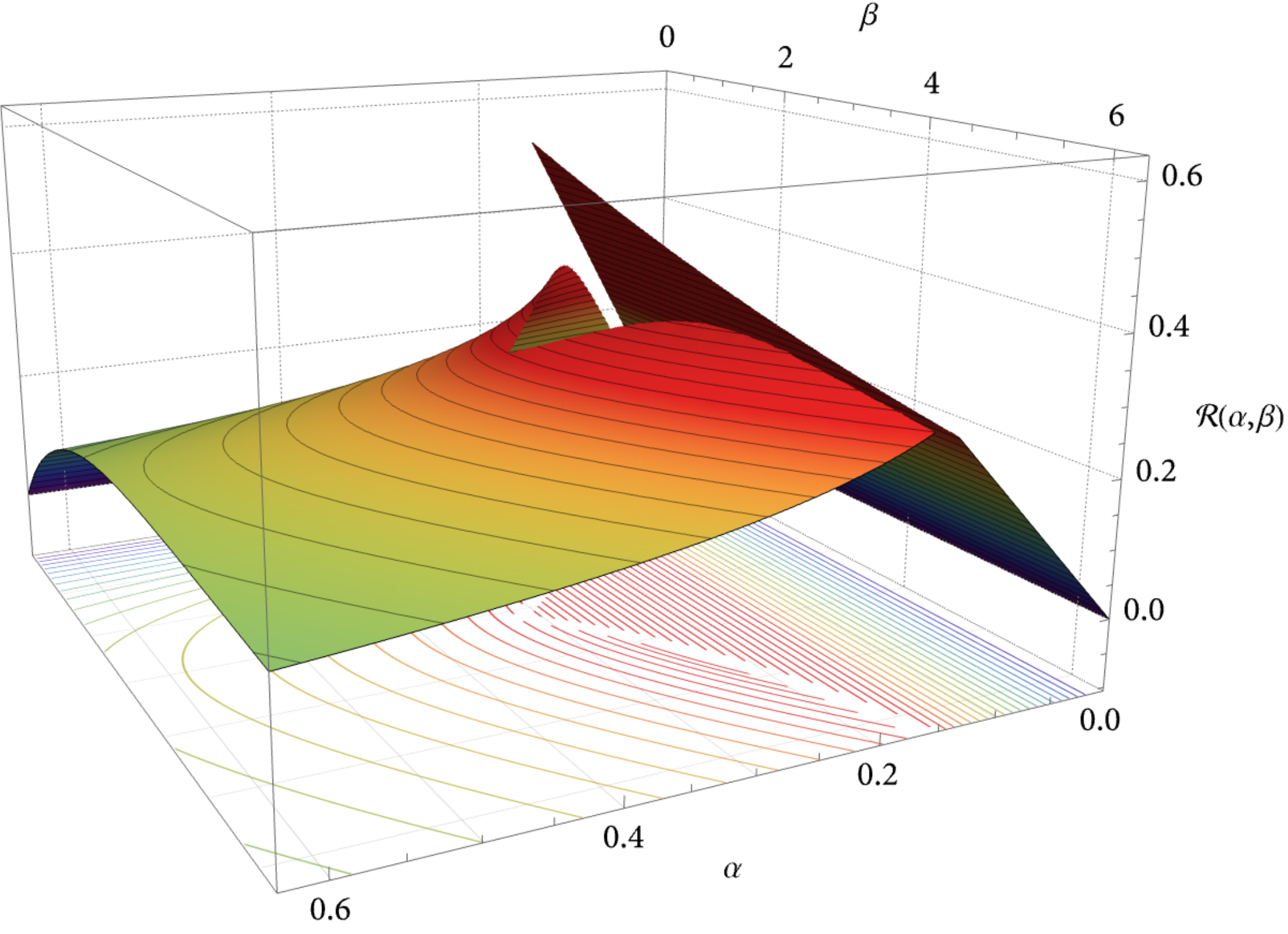}}
\caption{Illustration of the convergence exponent $\Rb(\alpha,\beta)$ expressed in \eqref{rate factor on alphabeta} as a function of $\alpha$ and $\beta$. 
}\label{Figure2}
\end{center}
\end{figure}

\begin{remark}[Localization]
	Note that the convergence rates as stated in Theorem \ref{basicteorates} hold true locally, in a neighborhood $\Omega$ of $\bar{x}\in \crit{F}$, where $F$ satisfies the \L ojasiewicz condition $\Lb(2)$ with constant $\mu>0$. However, since the function $F$ satisfies $\Lb(2)$ the convergence of $x(t)$ to $\bar{x}$ is assured (see point iii) in Proposition \ref{proposition basic}), there exists a time $t_{0}$, such that $x(t)$ will eventually enter (and remain for all $t\geq t_{0}$) in $\Omega$, where the rates of Theorem \ref{basicteorates} are valid. 
\end{remark}
\VM{I would say this seems more like a comment to clear things up. Possibly a footnote, if allowed? Sth along the lines of "F satisfies $\Lb(2)$, so convergence of $x(t)$ to $\bar{x}$ is assured by Proposition \ref{proposition basic}. Trajectory $x(t)$ enters and remains in a neighborhood $\Omega$ of $\bar{x}\in \crit{F}$, where theorem's rates are valid."}

The next Corollary provides the optimal tuning for the friction parameters $\alpha$ and $\beta$, that yield an optimal (maximal) worst case convergence factor $\Rb(\alpha,\beta)$, according to Theorem \ref{basicteorates}.
\PM{TODO: Check if can be streamlined.}
\begin{corollary}[Optimal rate]\label{cor optimal rates}
	Let $F$ be a function satisfying Assumption \ref{assumption basic F} and \ref{definition PL} with $q=2$. Let also $\left(x(t)\right)_{t\geq 0}$ be the solution-trajectory of \eqref{eq:DIN}, with $\lim\limits_{t\to\infty}x(t)=\bar{x}\in \crit{F}$.
Assume also that $\left(x(t)\right)_{t\geq t_{0}} \in \Omega$, where $\Omega$ is a neighborhood of $\bar{x}$, that $F$ satisfies \hyperref[PL]{$\mathcal{L}(2)$} with $\mu>0$. Then for all $0<\varepsilon<2\sqrt{\mu}$, if  $\alpha=\sqrt{\mu}-\frac{\varepsilon}{2}$ and $\beta\in \left(\frac{2\sqrt{\mu}-\varepsilon}{2\mu}, \frac{2}{2\sqrt{\mu}-\varepsilon}\right)$, it holds:
	\begin{equation}\label{rate optimal 2sqrtm}
		F(x(t))-F(\bar{x}) \leq \left(F(x(t_{0}))-F(\bar{x})+ \frac{1}{2\sqrt{\mu}\varepsilon}\norm{\nabla F(x(t_{0}))+\sqrt{\mu}\dot{x}(t_{0})}^2\right)e^{-(2\sqrt{\mu}-\varepsilon)(t-t_{0})}
	\end{equation}
	and 
	\begin{equation}
		\norm{\nabla F(x(t))} = \petito{e^{-(\sqrt{\mu}-\varepsilon)t}} ~ , ~ \text{ as } t\to+\infty
	\end{equation}
\end{corollary}

Below we provide some remarks commenting on the results of Theorem \ref{basicteorates} and Corollary \ref{cor optimal rates}. 
\begin{remark}[Optimality of rates]\label{rem:optimality rates}
	If $\hilbert=\R$ and $F(x)=\frac{\mu}{2}x^{2}$ (one dimensional quadratic function), which satisfies the \L ojasiewicz condition \hyperref[PL]{$\Lb(2)$} (it is strongly convex), the system \eqref{eq:DIN} simplifies to the linear ODE $\ddot{x}(t) +(\alpha+\mu\beta)\dot{x}(t) + \mu x(t)=0$, for which the solution trajectory can be expressed explicitly. In this case, a straightforward computation yields
	\begin{equation}
		F(x(t)) - F_{\ast} \leq C_{\varepsilon} e ^{-(r(\alpha,\beta)-\varepsilon)} \quad \text{ with : } ~ r(\alpha,\beta)=\alpha+\mu\beta -\sqrt{\max\{0,(\alpha+\mu\beta)^{2}-4\mu\}}
	\end{equation}
	and $\sup\{r(\alpha,\beta)~:~\alpha>0 ,~ \beta >0\} = 2\sqrt{\mu}$. This example shows that the convergence rate $\bigof{e^{-(2\sqrt{\mu}-\varepsilon)t}}$ found in \eqref{rate optimal 2sqrtm} of Corollary \ref{cor optimal rates} is worst case optimal for the class of functions satisfying the \L ojasiewicz condition \hyperref[PL]{$\Lb(2)$}.
\end{remark}

\begin{remark}\label{remark comparison}[Comparison with previous works]
To the best of our knowledge the rates found in Theorem \ref{basicteorates} and Corollary \ref{cor optimal rates} for the system \eqref{eq:DIN} are improved with respect to the ones found in the related literature. In particular the rate $\bigof{e^{-2\sqrt{\mu}t}}$ is faster than the one found in \cite[Theorem $7$]{attouch2022first} for strongly convex functions, which is $\bigof{t^{-\frac{\sqrt{\mu}}{2}t}}$. 
Since the strong convexity is a stronger notion than condition \hyperref[PL]{$\Lb(2)$}, the difference between the two rates is due to the convergence analysis and in particular the choice of the Lyapunov energy (indeed our analysis can be trivially applied for strongly convex functions). It is also worth mentioning that the rate $\bigof{e^{-2\sqrt{\mu}t}}$ was also established very recently in \cite{wang2025accelerated} in the strongly convex setting, which however holds for a rescaled version of system \eqref{eq:DIN} (see \cite[Corollary $2.2$]{wang2025accelerated}). Notice that the optimal rate $\bigof{e^{-2\sqrt{\mu}t}}$, as expressed in \eqref{rate optimal 2sqrtm} of Corollary \ref{cor optimal rates} is obtained by choosing $\alpha = \sqrt{\mu}$ and $\beta = \frac{1}{\sqrt{\mu}}$. This choice contrasts the one made in \cite[Theorem $7$]{attouch2022first} with $\alpha=2\sqrt{\mu}$, $\beta\leq \frac{1}{2\sqrt{\mu}}$ and indicates why the former (i.e. $\alpha = \sqrt{\mu}$, $\beta = \frac{1}{\sqrt{\mu}}$) is a better choice, leading to a faster worst-case rate. Indeed, this choice suggests to apply a more aggressive policy concerning the isotropic damping term (i.e. taking $\alpha$ smaller) and enforce a stronger geometric damping (i.e. taking $\beta$ larger). The difference on the performance based on these two policies is also witnessed numerically in the example of minimization of quadratic functions, as illustrated in Section \ref{sec:numerics} (see Figure \ref{Figure3}).

Finally, concerning the non-convex setting, under \hyperref[PL]{$\Lb(2)$}, by inspecting the proof of \cite[Theorem $13$]{castera2021inertial} (see also \cite[Theorem $3.4$]{maulen2024inertial}) , one can derive a linear rate of convergence result for $F(x(t))-F(\bar{x})$, which is not explicit and thus a straightforward comparison is difficult to be made. However, by following the lines of the aforementioned proof, our computations reveal a linear rate of the form $\bigof{e^{c_1\mu t}}$, for some abstract constant $c_{1}\leq2$, which is worse than the one in \cref{cor optimal rates}, for small values of $\mu$.
\end{remark}

\begin{remark}\label{remark comparison hb}[Comparison with the heavy ball with friction dynamics]
	It is worth mentioning that the optimal worst-case rate $\bigof{e^{-2\sqrt{\mu}t}}$ found in Corollary \ref{cor optimal rates}, holds true in a largely more general setting than the strongly convex one.  In contrast, for the plain heavy ball system \eqref{HBode}, the optimal rate $\bigof{e^{-2\sqrt{\mu}t}}$ is only proved for strongly convex and twice differentiable functions \cite[Theorem $9$]{polyak1964some}. Once the strong convexity is replaced by a weaker notion, such as quasi-strong convexity, or quadratic growth condition, or $\Lb(2)$ the associated proven linear rate is strictly slower than $\bigof{t^{-2\sqrt{\mu}t}}$, both in the convex case, see e.g. \cite[Theorem $1$]{aujol2022convergencequasi},  \cite[Corollary $1$]{aujol2023convergencePL} and the non-convex one, see \cite[Theorem $1$]{apidopoulos2022convergence}.
	This observation makes more clear quantitatively the advantages of the system \eqref{eq:DIN} with respect to the plain heavy ball with friction method \eqref{HBode}, in more complex settings where strong convexity does not hold true.
\end{remark}

\subsection{Convergence analysis under perturbations}\label{paragraph stability}

In this section we extend the convergence analysis of system \eqref{eq:DIN} in the presence of external perturbations. In particular we consider the following differential equation:
\begin{equation}\label{eq:DINperturbed}\tag{DIN-p}
	\begin{cases}
		&\ddot{x}(t) +\alpha\dot{x}(t) +\beta\nabla^{2}F(x(t))\dot{x}(t) +\nabla F(x(t)) +g(t) =0 \\ 
		& x(0) = x_{0} ~, \quad \dot{x}(0) =v_{0}
	\end{cases}
\end{equation} 
where $g:\hilbert\to\R$ is an artificial term incorporating any possible inexact information in \eqref{eq:DIN} (e.g. approximation of $\nabla F$ or computational errors). Assuming that $g\in L^{1}_{loc}(\hilbert,\R)$, together with Assumption \ref{assumption basic F}, guarantee the existence of a solution to \eqref{eq:DINperturbed}. 

In particular, in Theorems \ref{basicteo2inexact} and \ref{basicteo2inexact q} we provide estimates on $F(x(t))-F(\bar{x})$ depending on: $i)$ the local geometry of $F$ around its critical points, i.e. the \L ojasiewicz exponent $q$ and $ii)$ on the controlability of the perturbation error $g(t)$.


\begin{theorem}\label{basicteo2inexact}
	Let $F$ be a function satisfying Assumption \ref{assumption basic F} and Assumption \ref{definition PL} with $q=2$. Let also $\left(x(t)\right)_{t\geq 0}$ be the solution-trajectory of \eqref{eq:DINperturbed} satisfying the following:
	
	i) There exists some $t_{0}>0$, such that  $\left(x(t)\right)_{t\geq t_{0}} \in \Omega$,  where $\Omega$ is a neighborhood of $\bar{x}\in\crit{F}$, that $F$ satisfies $\Lb(2)$ with $\mu>0$.
	
	Then the following hold:
	\begin{equation}\label{estimate gap perturbed 2}
		\begin{aligned}
			F(x(t))- & F(\bar{x})  \leq \left(\sqrt{M_{0}} + \frac{\Jb(t)}{\sqrt{2C(\alpha,\beta)}}\right)^{2}e^{-\Rb(\alpha,\beta)t}
		\end{aligned}
	\end{equation}
with $\Rb(\alpha,\beta)$ and $C(\alpha,\beta)$ as given in \eqref{rate factor on alphabeta} and \eqref{constant on alphabeta} (respectively) of Theorem \ref{basicteorates}
and $\Jb(t) = \int_{t_{0}}^{t}e^{\frac{\Rb(\alpha,\beta)s}{2}}\norm{g(s)}ds$, 
$M_{0}=\left(F(x(t_{0}))-F(\bar{x}) +\frac{1}{2C(\alpha,\beta)}\norm{\beta\nabla F(x(t_{0}))+\dot{x}(t_{0})}^{2}\right)e^{\Rb(\alpha,\beta)t_{0}}$.

In addition, if $F(x(t))\geq F(\bar{x})$ for all $t\geq t_{0}$, then for any $\varepsilon>0$, it holds:
\begin{equation}\label{estimate grad basic teo first inexact}
	\int_{t_{0}}^{+\infty}e^{(\Rb(\alpha,\beta)-\varepsilon)t}\norm{\nabla F(x(t))}^2dt =\bigof{\left(\int_{t_{0}}^{+\infty}e^{\frac{(\Rb(\alpha,\beta)-\varepsilon)s}{2}}\norm{g(s)}ds\right)^{2}}
\end{equation} 
	
	In particular, if $\Jb_{\infty}:=\int_{t_{0}}^{+\infty}e^{\frac{\Rb(\alpha,\beta)s}{2}}\norm{g(s)}ds <+\infty$, then
	\begin{equation}\label{estimate gap perturbed 2 strong}
		F(x(t))- F(x_{\ast}) = \bigof{e^{-\Rb(\alpha,\beta)t}}
	\end{equation}
\end{theorem}

\begin{remark}
	Theorem \ref{basicteo2inexact} provides an estimate for $F(x(t))-F(\bar{x})$ similar to the one of Theorem  \ref{basicteorates}, with the presence of the perturbation term $g(t)$. Indeed, from \eqref{estimate gap perturbed 2}, we observe that the decay rate of the value gap, depends on the decay rate (integrability) of the error term $g(t)$. As also mentioned in Theorem \ref{basicteo2inexact} (see in particular \eqref{estimate gap perturbed 2 strong}), in order to achieve the same rate as the one in the exact setting $\bigof{e^{-\Rb(\alpha,\beta)}}$, the control condition on the perturbation term is  $e^{\frac{\Rb(\alpha,\beta)s}{2}}\norm{g(s)} \in L^{1}(t_{0},+\infty)$. Note however that the decay rate for the perturbation term $\norm{g(t)}$ needed, is strictly smaller than the convergence rate obtained by a factor of $2$.
\end{remark}
\begin{remark}
In the same spirit with Corollary \ref{cor optimal rates}, by setting $\alpha =\sqrt{\mu}-\frac{\varepsilon}{2}$ and $\beta\in\left(\frac{2\sqrt{\mu}-\varepsilon}{2\mu},\frac{2}{2\sqrt{\mu-\varepsilon}}\right)$, if we additionally assume that  $\bar{\Jb}_{\infty}:=\int_{t_{0}}^{+\infty}e^{\left(\sqrt{\mu}-\frac{\varepsilon}{2}\right)s}\norm{g(s)}ds <+\infty$, from \eqref{estimate gap perturbed 2} it follows that $	F(x(t))- F(x_{\ast}) = \bigof{e^{-(2\sqrt{\mu}-\varepsilon)t}}$. This rate is improved than the one found in \cite{attouch2021effect} for the system \ref{eq:DINperturbed}, which is of order  $\bigof{e^{\frac{\sqrt{\mu}}{2}t}}$.
\end{remark}

The next Theorem treats the case of an objective function satisfying the \L ojasiewicz condition \ref{PL} with $q\in(1,2)$.
\begin{theorem}\label{basicteo2inexact q}
Let $F$ be a function satisfying Assumptions \ref{assumption basic F} and \ref{definition PL} with $q\in(1,2)$ and $\left(x(t)\right)_{t\geq 0}$ be the solution-trajectory of \eqref{eq:DINperturbed}) Assume that there exists some $t_{0}>0$, such that for all $t\geq t_{0}$:\\
$i)$  $\left(x(t)\right)_{t\geq t_{0}} \in \Omega$, where $\Omega$ is a neighborhood of $\bar{x}\in \crit{F}$, that $F$ satisfies $\Lb(2)$ with $\mu>0$.\\
$ii)$ $\norm{\dot{x}(t)} \leq 1$ for all $t\geq t_{0}$.
	
Then the following hold:
\begin{equation}\label{estimate gap perturbed q}
	\begin{aligned}
		F(x(t))-  F(\bar{x})
		& \leq \frac{1}{\alpha\beta+1}\left(\sqrt{C_{4}} +\frac{\Ib(t)}{\sqrt{2}}\right)^{2}t^{-\frac{q}{2-q}}
	\end{aligned}
\end{equation}
	where 	$\Ib(t) = \int_{t_{0}}^{t}s^{\frac{q}{2(2-q)}}\norm{g(s)}ds$, 
	and\\
	$C_{4} = \max\left\{\left(\frac{q\max\left\{1,\left(\frac{\alpha\beta+1}{2\mu }+\beta^{2}\right)^{\frac{2}{q}}\right\}}{(2-q)\min\{\alpha,\beta\}}\right)^{\frac{q}{2-q}},t_{0}^{\frac{q}{2-q}}\left((\alpha\beta+1)\left(F(x(t_{0}))-F(\bar{x})\right) +\frac{\norm{\beta\nabla F(x(t_{0}))+\dot{x}(t_{0})}^{2}}{2}\right)\right\}$.
\end{theorem}

\begin{remark}
Theorem \ref{basicteo2inexact q} states that under sufficient integrability condition on $g$, the convergence rate  is sublinear and depends on the parameter $q$ appearing in condition \ref{PL}. Indeed, smaller $q$ (closer to $1$) indicates that the local geometry of $F$ around a limiting critical point is flatter, which demands a weaker control on $g$, but also leads to a slower asymptotic rate. In fact, even if not clearly stated, similar rates have been found in \cite[Theorem $13$]{castera2021inertial} (see also \cite{maulen2024inertial}) in the exact case when $g\equiv 0$, for the generalized version of system \ref{HBHDfirstorder1} (while the rates are not stated in the main Theorem, they can be deduced from its proof). In Theorem \ref{basicteo2inexact q} we extend these rates to the perturbed version (see also Theorem \ref{basicteorates firstorder q} for the first-order equivalent system).
\end{remark}

\begin{remark}
We stress out that the boundedness assumption on $x(t)\in \Omega$ and $\dot{x}(t)$ for all $t\geq t_{0}$, where $F$ satisfies \ref{PL} in $\Omega$, is necessary for the analysis in Theorems \ref{basicteo2inexact} and \ref{basicteo2inexact q}. In contrast with the exact setting (i.e. for the system \eqref{eq:DIN}), for the perturbed system \eqref{eq:DINperturbed}, there is no result showing convergence of the trajectory to a critical point of $F$. Our analysis suggests that under sufficient integrability and boundedness conditions on $g$ (depending on the exponent $q$), one could show that the assumption is actually superfluous. Nevertheless a detailed analysis of all the convergence properties of system \eqref{eq:DINperturbed} is beyond the scope of the current work and is let for future study. 
\end{remark}

\subsection{First-order in time/space equivalent system}\label{paragraph first order}
In this section we extend the previous results for the first-order coupled system \eqref{HBHDfirstorder1} associated to \eqref{eq:DIN}. In particular, by setting $y(t) = (1-\alpha\beta)x(t) - \beta \dot{x}(t) - \beta^{2}\nabla F(x(t))$ and using \eqref{eq:DIN} one can obtain (see \cite{alvarez2002second}):
\begin{equation}
	\begin{cases}{}\label{HBHDfirstorder}\tag{g-DIN}
		\dot{x}(t) +\beta\nabla F(x(t)) + \left(\alpha-\frac{1}{\beta}\right)x(t) + \frac{1}{\beta}y(t) = 0 \\
		\dot{y}(t) + \left(\alpha-\frac{1}{\beta}\right)x(t) + \frac{1}{\beta}y(t) = 0 
	\end{cases}
\end{equation}
More precisely, if $(x(t),y(t))_{t\geq 0}$ is a solution of \eqref{HBHDfirstorder} and $y\in \Cb^{2}([0,+\infty))$, then $x$ is a solution to \eqref{eq:DIN} (see \cite[Theorem $6.1$]{alvarez2002second}). Note however, that the system \eqref{HBHDfirstorder}, still makes sense for a larger class of functions $F$, since the twice differentiable character of $F$ is no longer needed.

In this paragraph we extend the previous results for the system \eqref{HBHDfirstorder} for functions that are not necessarily twice continuously differentiable. In fact, in order to avoid repeatability of the arguments, we directly present the results and their associated proofs for the perturbed version of \eqref{HBHDfirstorder1}:
\begin{equation}
	\begin{cases}{}\label{HBHDfirstorder inexact}\tag{g-DIN-p}
		\dot{x}(t) +\beta\nabla F(x(t)) + \left(\alpha-\frac{1}{\beta}\right)x(t) + \frac{1}{\beta}y(t) = 0 \\
		\dot{y}(t) + \left(\alpha-\frac{1}{\beta}\right)x(t) + \frac{1}{\beta}y(t) - \beta g(t) = 0 \\
		x(0)=x_{0} ~ , \quad y(0)= (1-\alpha\beta)x_{0}  -\beta v_{0}- \beta^{2}\nabla F(x_{0})
	\end{cases}
\end{equation}
where $g:\hilbert\to\R$, such that $g(\cdot)\in L^{1}_{loc}(\hilbert)$ encodes the perturbation error. 

The next two Theorems provide stability estimates for functions satisfying condition \ref{PL} with $q=2$ and $q\in(1,2)$ for the system \eqref{HBHDfirstorder inexact}, as analogues  of Theorems \ref{basicteo2inexact} and \ref{basicteo2inexact q} (respectively). 

%

\begin{theorem}\label{basicteorates firstorder}
Let $F:\hilbert\to\R$ be a function bounded from below and differentiable with locally Lipschitz gradient satisfying \hyperref[PL]{$\mathcal{L}(2)$} and $\left(x(t),y(t)\right)_{t\geq0}$ be the solution-trajectory of \eqref{HBHDfirstorder inexact} satisfying the following:\\
$i)$ There exists some $t_{0}>0$, such that  $\left(x(t)\right)_{t\geq t_{0}} \in \Omega$,  where $\Omega$ is a neighborhood of $\bar{x}\in \crit{F}$, that $F$ satisfies $\Lb(2)$ with $\mu>0$.

Then the following estimate holds:
\begin{equation}
	F(x(t))-F(\bar{x}) \leq \left(\sqrt{M_{0}}+\frac{\Jb(t)}{\sqrt{2C(\alpha,\beta)}}\right)^{2}e^{-\Rb(\alpha,\beta)(t)}
\end{equation} 
with $\Rb(\alpha,\beta)$ and $C(\alpha,\beta)$ as given in \eqref{rate factor on alphabeta} and \eqref{constant on alphabeta} (respectively) of Theorem \ref{basicteorates}
and $\Jb(t)=\int_{t_{0}}^{t}e^{\frac{\Rb(\alpha,\beta)s}{2}}\norm{g(s)}ds$, $M_{0}=\left(F(x(t_{0}))-F(\bar{x}) +\frac{1}{2C(\alpha,\beta)}\norm{\beta\nabla F(x(t_{0}))+\dot{x}(t_{0})}^{2}\right)e^{\Rb(\alpha,\beta)t_{0}}$.

In addition, if $F(x(t))\geq F(\bar{x})$,  for all $t\geq t_{0}$), then for any $\varepsilon>0$, it holds:
\begin{equation}\label{estimate grad basic teo first}
	\int_{t_{0}}^{+\infty}e^{(\Rb(\alpha,\beta)-\varepsilon)t}\norm{\nabla F(x(t))}^2dt = \bigof{\left(\int_{t_{0}}^{+\infty}e^{\frac{(\Rb(\alpha,\beta)-\varepsilon)s}{2}}\norm{g(s)}ds\right)^{2}}
\end{equation} 
\end{theorem}

\begin{theorem}\label{basicteorates firstorder q}
Let $F:\hilbert\to\R$ be a function bounded from below and differentiable with locally Lipschitz gradient satisfying \ref{PL} with $q\in(1,2)$ and $\left(x(t),y(t)\right)_{t\geq0}$ be the solution-trajectory of \eqref{HBHDfirstorder inexact} Assume that there exists some $t_{0}>0$, such that for all $t\geq t_{0}$:\\
$i)$  $\left(x(t)\right)_{t\geq t_{0}} \in \Omega$, where $\Omega$ is a neighborhood of $\bar{x}\in\crit{F}$, that $F$ satisfies $\Lb(2)$ with $\mu>0$.\\
$ii)$ $\norm{\dot{x}(t)} \leq 1$ , for all $t\geq t_{0}$.
	
Then the following estimate holds:
\begin{equation}\label{rate on alphabetaq first}
	F(x(t))-F(\bar{x}) \leq  \left(\sqrt{C_{4}} + \Ib(t)\right)^{2}t^{-\frac{q}{2-q}}
\end{equation}
with $\Ib(t)=\int_{t_{0}}^{t}s^{\frac{q}{2(2-q)}}\norm{g(s)}ds$ and\\ $C_{4} = \max\left\{\left(\frac{q\max\left\{1,\left(\frac{\alpha\beta+1}{2\mu}+\beta^{2}\right)^{\frac{2}{q}}\right\}}{(2-q)\min\{\alpha,\beta\}}\right)^{\frac{q}{2-q}} , t_{0}^{\frac{q}{2-q}}\left((\alpha\beta+1)\left(F(x(t_{0}))-F(\bar{x})\right) +\frac{\norm{\beta\nabla F(x(t_{0}))+\dot{x}(t_{0})}^{2}}{2}\right)\right\}$
\end{theorem}
\begin{remark}[Extension to the non-smooth setting \cite{castera2021inertial}] The results stated in Theorems \ref{basicteorates firstorder} and \ref{basicteorates firstorder q} can be also extended to a non-smooth setting, when $g\equiv0$. In particular, for a locally Lipschitz function $F$, one can consider the following differential inclusion for a.e. $t\geq0$:
		\begin{equation}
			\begin{cases}{}\label{HBHDinclusion}
				\dot{x}(t) +\beta\partial F(x(t)) + \left(\alpha-\frac{1}{\beta}\right)x(t) + \frac{1}{\beta}y(t) \ni 0 \\
				\dot{y}(t) + \left(\alpha-\frac{1}{\beta}\right)x(t) + \frac{1}{\beta}y(t) \ni 0 
			\end{cases}
		\end{equation}
		where $\partial F(x(t))$ denotes the Clarke subdifferential \cite[Definition $10.3$]{clarke2013functional}. In this case \eqref{HBHDinclusion} admits a (coupled) solution which is absolutely continuous (see \cite{aubin1999set} for a general theory and \cite{castera2021inertial} for the application to system \ref{HBHDinclusion}). If $F$ satisfies the \L ojasiewicz condition \ref{PL} (where the right-hand side in \ref{PL} is replaced with $\frac{1}{2\mu}\text{dist}\left(0 , \partial F(x)\right)^{q}$ ), then by using the generalized chain rule property (see e.g. \cite[Lemma $4$]{castera2021inertial}), we can still perform the same Lyapunov analysis described in Section \ref{sec:analysis} (equalities and inequalities will only be valid for a.e. $t\geq t_{0}$) and obtain the same conclusions of Theorems \ref{basicteorates firstorder} and \ref{basicteorates firstorder q} (with $g\equiv 0$) for the solution of \eqref{HBHDinclusion}. Nevertheless, a thorough study of the differential inclusion \eqref{HBHDinclusion} goes beyond the scope of the current work.
	\end{remark}

\PM{Put this as a future direction in the conclusions section.}

\subsection{Almost sure avoidance of strict saddle points} 
\PM{To be done: move to separate subsection at end; introduce required notions to state the theorem; avoid informal statements ("we can guarantee the convergence"... what does this mean, precisely?); simplify statements (since $\obj$ satisfies KL).} 

As we have seen in Proposition \ref{proposition basic}, under the \L ojasiewicz condition, the dynamics of \eqref{eq:DIN} are attracted to critical points.
This does not exclude attraction by saddle points, i.e. critical points that are not local minima, which makes sense, since there exist trajectories with such limit points. 
However, these trajectories are expected to represent exceptional, atypical instances in the sense that it is improbable that a ball balances on top of a hill, or even on a saddle. 
This concept of atypicality can be captured in the vanishing Lebesgue measure of initial conditions attracted by (strict) saddle points \cite{PdM82}.
Although such an investigation has previously been carried out with regard to \eqref{eq:DIN} or \eqref{HBHDfirstorder} in \cite{castera2023inertial, maulen2024inertial} and the relevant tools were developed for a discretization of DIN, there has been no presentation of the almost sure avoidance of strict saddle points in the presence of connected components of (strict) saddle points for the continuous-in-time system. 
For concreteness, we use the following definition
\begin{definition} 
    A strict saddle point is a point $x\in\mathbb{R}^d$ such that $\nabla F(x) = 0$ and $\nabla^2 F(x)$ has a negative eigenvalue. 
\end{definition}
 
The following theorem captures that, under the $C^{2}$ differentiability and the \L ojasiewicz condition, for almost all initial conditions, the solution of \eqref{eq:DIN} does not converge to strict saddle points. 
From this result and the guaranteed convergence, provided by Proposition \ref{proposition basic}, it follows that strict saddle points are collectively avoided by the dynamics of \eqref{eq:DIN}. 
Specifically, the following theorem does not require that the critical points are isolated in any way. 
\begin{theorem} 
\label{thrm:avoidance_DIN} 
    Suppose that $\mathcal{H}=\mathbb{R}^d$ and that $F$ satisfies Assumptions \ref{asm:obj} and \ref{asm:KL}. 
    Then, the set of all pairs $(x,v)$ with $x$ and $v\in\mathbb{R}^{d}$ that belong to a solution-trajectory of \eqref{eq:DIN} with a strict saddle point in their limit points is of Lebesgue measure $0$. 
\end{theorem} 
The proof is provided in Section \ref{subsec:avoidance_proof}.
\begin{remark}
    Actually, the proof of Theorem \ref{thrm:avoidance_DIN} reveals that the above result also holds for solution-trajectories of \eqref{HBHDfirstorder} under the same assumptions.
\end{remark}

\section{Convergence analysis}
\label{sec:analysis}

In this Section we provide the main tools needed for the proofs of the associated Theorems in Section \ref{sec:results}. All the proofs of the auxiliary lemmas are postponed in the Appendices \ref{app:aux} and \ref{app:rates}. 
For presentational and pedagogical reasons, we first give a detailed presentation of the analysis and proofs of Theorems concerning system \eqref{eq:DIN} and then provide the ones corresponding to its first-order variant \eqref{HBHDfirstorder1}, which follow the same principles.

As also mentioned in the introduction, the main analysis is based on the use of a well-chosen Lyapunov energy for the system \eqref{eq:DIN}. Let $a>0$ and $\left(x(t)\right)_{t\geq 0}$ the trajectory generated by the system \eqref{eq:DIN}. For all $t\geq 0 $, we define the following energy-function:
\begin{equation}\label{Lyapunov def}
    V(t) = 
    a(F(x(t))-F(\bar{x})) + \frac{1}{2}\norm{\beta\nabla F(x(t))+\dot{x}(t)}^{2}
\end{equation}
which plays a central role in the main analysis.

The energy $V$ as defined in \eqref{Lyapunov def} is similar to the one used in \cite{alvarez2002second} (see also \cite{castera2021inertial}), with a different a-posteriori choice for the  parameter $a$ as we shall see in the forthcoming analysis. 
Other choices of Lyapunov energies are also possible (see e.g. \cite{attouch2022first,aujol2023fast}) and lead to useful convergence properties depending on the optimization setting. In contrast with the choice made in \cite{attouch2022first}, in our case, it is important to point out  that the energy function $V$ does not explicitly depend on any critical point of $F$ (only on $F(\bar{x})$) and convexity is not crucial in the analysis.

\begin{lemma}\label{lemma general conditions}
 Let $F$ be a function satisfying Assumption \eqref{assumption basic F} and \hyperref[PL]{$\Lb(2)$}.  Let also $\left(x(t)\right)_{t\geq0}$ be the solution-trajectory of \eqref{eq:DIN}, with $\lim\limits_{t\to\infty}x(t)=\bar{x}\in \crit{F}$
 and $\left(x(t)\right)_{t\geq t_{0}} \in \Omega$, where $\Omega$ is a neighborhood of $\bar{x}$, that $F$ satisfies $\Lb(2)$ with $\mu>0$. Let also $V$ the function defined in \eqref{Lyapunov def}, with  $a\geq0$ satisfying the following condition:
 \begin{equation}\label{conditions on a/d}\tag{H$1$}
 \begin{cases}
 1-\alpha\beta \leq a \leq 1+\alpha\beta \\
 a^2 - \left((\alpha-\mu\beta)\beta +1\right)a + (1-\alpha\beta)\mu\beta^{2} \geq 0
 \end{cases}
 \end{equation} 
 and set $R=\frac{1}{\beta}\left(\alpha\beta+1-a\right)$.
 Then for all $t\geq t_{0}$, the following estimate holds
 \begin{equation}\label{decay of V}
 V(t) \leq V(t_{0})e^{-R(t-t_{0})}
 \end{equation}
 In particular 
 \begin{equation}
 F(x(t))- F(\bar{x}) \leq \left(F(x(t_{0})) - F(\bar{x}) + \frac{1}{2a}\norm{\beta\nabla F(x(t_{0}))+\dot{x}(t_{0})}^2\right)e^{-R(t-t_{0})}
 \end{equation}
 If in addition $\varepsilon(a,\alpha,\beta)=a^2 - \left((\alpha-\mu\beta)\beta +1\right)a + (1-\alpha\beta)\mu\beta^{2}>0$ (i.e. the second inequality in \eqref{conditions on a/d} holds true) and $F(x(t))\geq F(\bar{x})$ for all $t\geq t_{0}$, then 
 \begin{equation}\label{estimate grad lemma1}
 	\int_{t_{0}}^{t}e^{Rs}\norm{\nabla F(x(s))}^{2}ds \leq \frac{ 2\mu\beta\left(2a\left(F(x(t_{0})) - F(\bar{x})\right) + \norm{\beta\nabla F(x(t_{0}))+\dot{x}(t_{0})}^2\right)e^{Rt_{0}}}{2\varepsilon(a,\alpha,\beta)}
 \end{equation}
\end{lemma}

\begin{remark}
Lemma \ref{lemma general conditions} states that under the (implicit) condition \eqref{conditions on a/d}, the energy function $V$ has an exponential decay of order $R$. While it may not be clear that condition \eqref{conditions on a/d} is non-void at this point, the next Lemma (see Lemma \ref{lemma lyapunov2}) ensures the existence of a Lyapunov parameter $a \geq 0$, such that \eqref{conditions on a/d}. holds true.
\end{remark}


\begin{lemma}\label{lemma lyapunov2}
Let $V$ the function defined in \ref{Lyapunov def}, with $a$ satisfying the following conditions:
\begin{enumerate}
\item If $0<\alpha\leq 2(\sqrt{2}-1)\sqrt{\mu}$
\begin{itemize}
\item If $\beta \in [\beta_{1},\beta_{2}]$, then  
\begin{equation}\label{condition on y 1}
\max\{0,1-\alpha\beta\} \leq a \leq 1+\alpha\beta
\end{equation}
\item If $\beta\in (0,\beta_{1}]\cup[\beta_{2},+\infty)$,  then
\begin{equation}\label{condition on y 2}
\begin{cases}
&\max\{0,1-\alpha\beta\} \leq a \leq y_{-} \\
\text{OR} & \max\{0,1-\alpha\beta,y_{+}\} \leq a \leq 1+\alpha\beta
\end{cases}
\end{equation}
\end{itemize} 
\item If $\alpha>2(\sqrt{2}-1)\sqrt{\mu}$ , then 
\begin{equation}\label{condition on y 3}
\begin{cases}
&\max\{0,1-\alpha\beta\} \leq a \leq y_{-} \\
\text{OR} & \max\{0,1-\alpha\beta,y_{+}\} \leq a \leq 1+\alpha\beta
\end{cases}
\end{equation}
\end{enumerate}
\text{where: }
\begin{align}
 \beta_{1} & = \frac{1}{2\mu}\left(2\sqrt{2\mu}-\alpha - \sqrt{(a-2\sqrt{2\mu})^{2}-4\mu}\right) \label{definition beta1} \\ 
\beta_{2} &= \frac{1}{2\mu}\left(2\sqrt{2\mu}-\alpha + \sqrt{(a-2\sqrt{2\mu})^{2}-4\mu}\right) \label{definition beta2} \\
\text{and } \quad y_{\pm}& =\frac{1}{2}\left((\alpha-\mu\beta)\beta+1 \pm \sqrt{\left((\alpha-\mu\beta)\beta+1\right)^2 - 4(1-\alpha\beta)\mu\beta^2}\right) \label{definition ypm}
\end{align}
Then, condition \eqref{conditions on a/d} in Lemma \ref{lemma general conditions} is satisfied. More precisely, under the same assumptions of Lemma \ref{lemma general conditions}, if the parameters $(\alpha,\beta,a)\in \R_{+}^{3}$ satisfy one of the conditions \eqref{condition on y 1}, \eqref{condition on y 2} or \eqref{condition on y 3}, then the following estimate holds:
\begin{equation}\label{rate V lemma 2}
	V(t) \leq V(t_{0})e^{-\frac{1}{\beta}\left(1+\alpha\beta-a\right)(t-t_{0})}
\end{equation}
In particular:
\begin{equation}\label{objectifunctionvaluesgaplemma 2}
	F(x(t))- F(\bar{x}) \leq \left(F(x(t_{0})) - F(\bar{x}) + \frac{1}{2a}\norm{\beta\nabla F(x(t_{0}))+\dot{x}(t_{0})}^2\right)e^{-\frac{1}{\beta}\left(1+\alpha\beta-a\right)(t-t_{0})}
\end{equation}
Here the convention $[0,y_{-}]=\emptyset$, if $y_{-}<0$ is used.
\end{lemma}

In the following Lemma we provide an explicit formula for the maximal value of the geometric factor $\frac{1}{\beta}\left(1+\alpha\beta-a\right)$ in the estimate \eqref{objectifunctionvaluesgaplemma 2} depending on the parameters $\alpha$ and $\beta$.

\begin{lemma}\label{lemma lyapunov3}
	Let $\alpha$, $\beta$ two non-negative numbers and $a\geq 0$ satisfying the conditions in Lemma \ref{lemma lyapunov2}. Let also $R(a,\alpha,\beta) := \frac{1}{\beta}\left(1+\alpha\beta-a\right)$. Then:
	\begin{equation}
		\begin{aligned}
		\underset{a\geq 0}{\max} R(a,\alpha,\beta)= \begin{cases}
				2\alpha  \qquad \qquad  ~ \text{ if } ~   \left\{\alpha\leq \sqrt{\mu}\right\}\&  \left\{\beta\in[\frac{\alpha}{\mu},\frac{1}{\alpha}]\right\} &  \\
				\\
				\frac{1}{2\beta}\left((\alpha+\mu\beta)\beta+1 - \sqrt{\left((\alpha-\mu\beta)\beta+1\right)^{2}-4(1-\alpha\beta)\mu\beta^{2}}\right)   & \\ 
				\qquad \qquad \qquad  ~ \text{ if } ~   \left\{  \alpha > 0\right\} \& \left\{\beta\in \left(0, \frac{\alpha}{\mu}\right)\cup \left(\frac{1}{\alpha},+\infty\right) \right\} &
			\end{cases}
		\end{aligned}
	\end{equation}
with 		\begin{equation}
	\begin{aligned}
		\underset{a\geq 0}{\argmax} R(a,\alpha,\beta)= \begin{cases}
			1-\alpha\beta \qquad \qquad  ~ \text{ if } ~   \left\{\alpha\leq \sqrt{\mu}\right\}\&  \left\{\beta\in[\frac{\alpha}{\mu},\frac{1}{\alpha}]\right\} &  \\
			\\
\frac{1}{2}\left((\alpha-\mu\beta)\beta+1 + \sqrt{\left((\alpha-\mu\beta)\beta+1\right)^{2}-4(1-\alpha\beta)\mu\beta^{2}}\right)  & \\ 
			\qquad \qquad \qquad  ~ \text{ if } ~   \left\{  \alpha > 0\right\} \& \left\{\beta\in \left(0, \frac{\alpha}{\mu}\right)\cup \left(\frac{1}{\alpha},+\infty\right) \right\} &
		\end{cases}
	\end{aligned}
\end{equation}
\end{lemma}

We are now ready to give the full proof of Theorem \ref{basicteorates}

\begin{proof}[\textbf{Proof of Theorem \ref{basicteorates}}]
	The proof of Theorem \ref{basicteorates}, follows by combining Lemmas \ref{lemma general conditions}, \ref{lemma lyapunov2} and \ref{lemma lyapunov3}. More precisely, for all $a$ satisfying the conditions in Lemma \ref{lemma lyapunov2}, from \eqref{objectifunctionvaluesgaplemma 2} we have:
	\begin{equation}\label{bound on F in the proof of Thm}
		F(x(t))- F(\bar{x}) \leq \left(F(x(t_{0})) - F(\bar{x}) + \frac{1}{2a}\norm{\beta\nabla F(x(t_{0}))+\dot{x}(t_{0})}^2\right)e^{-\frac{1}{\beta}\left(1+\alpha\beta-a\right)(t-t_{0})},
	\end{equation}
	and in particular
		\begin{equation}\label{bound on F in the proof of Thm2}
		F(x(t))- F(\bar{x}) \leq \left(F(x(t_{0})) - F(\bar{x}) + \frac{1}{2a_{\ast}}\norm{\beta\nabla F(x(t_{0}))+\dot{x}(t_{0})}^2\right)e^{-\Rb(\alpha,\beta)(t-t_{0})},
	\end{equation}
with  $\Rb(\alpha,\beta) :=  \max\left\{\frac{1}{\beta}(\alpha\beta+1-a)~:~ a\sim \text{ Lemma \ref{lemma lyapunov2}}\right\}$ and $a_{\ast}=\inf\{a~:~a \sim \text{ Lemma \ref{lemma lyapunov2}}\}$. Moreover,  from Lemma \ref{lemma lyapunov3} and the definition of $\auxconst(\alpha,\beta)$ in \eqref{eq: q expression}, it follows that 
\begin{equation}
	\Rb(\alpha,\beta) =\begin{cases}
		2\alpha  \qquad \qquad \qquad \qquad  & \text{ if } ~   \left\{\alpha\leq \sqrt{\mu}\right\}\&  \left\{\beta\in[\frac{\alpha}{\mu},\frac{1}{\alpha}]\right\}   \\
		\frac{1}{2\beta}\left((1+\alpha\beta - \auxconst(\alpha,\beta) \right)   \quad  & \text{ if } ~   \left\{  \alpha > 0\right\} \& \left\{\beta\in \left(0, \frac{\alpha}{\mu}\right)\cup \left(\frac{1}{\alpha},+\infty\right) \right\} 
	\end{cases}
\end{equation}
\begin{equation}
\text{and } ~	C(\alpha,\beta) = \underset{a\geq 0}{\argmax} R(a) = \begin{cases}
			1-\alpha\beta   & \text{ if } ~   \left\{\alpha\leq \sqrt{\mu}\right\}\&  \left\{\beta\in[\frac{\alpha}{\mu},\frac{1}{\alpha}]\right\}  \\
			\frac{1+\alpha\beta + \auxconst(\alpha,\beta)}{2}   \quad   & \text{ if } ~   \left\{  \alpha > 0\right\} \& \left\{\beta\in \left(0, \frac{\alpha}{\mu}\right)\cup \left(\frac{1}{\alpha},+\infty\right) \right\} 
		\end{cases}
	\end{equation}

For the last part of the Theorem, concerning the estimate on $\norm{\nabla F(x(t))}$, one can observe, that if $\alpha \in (0,\sqrt{\mu})$ and $\beta\in \left(\frac{\alpha}{\mu},\frac{1}{\alpha}\right)$, then the binomial $P_{2}(a):= a^2 - \left((\alpha-\mu\beta)\beta +1\right)a + (1-\alpha\beta)\mu\beta^{2}$ evaluated at $a=1-\alpha\beta>0$ is strictly positive. Since $P_{2}(1-\alpha\beta)>0$, from \eqref{estimate grad lemma1} of Lemma \ref{lemma general conditions}, it follows that the function $t~\mapsto ~ e^{-2\alpha t}\norm{\nabla F(x(t))}^{2}$ is integrable in $[t_{0},+\infty)$. In the same way, for any $\varepsilon>0$ sufficiently small, for all $\alpha>0$ and $\beta\in\left(0,\frac{\alpha}{\mu}\right)\cup\left(\frac{1}{\alpha},+\infty\right)$, the binomial $P_{2}(a)$ evaluated at $a=\frac{1}{2}\left((\alpha-\mu\beta)\beta+1 + \sqrt{\left((\alpha-\mu\beta)\beta+1\right)^{2}-4(1-\alpha\beta)\mu\beta^{2}}\right)+\varepsilon$ is strictly positive (note that $y_{+}$ is the largest root of $P_{2}(a)$), thus from \eqref{estimate grad lemma1} of Lemma \ref{lemma general conditions}, it follows that the function $t~\mapsto ~ e^{-\frac{1}{\beta}(1+\alpha\beta-y_{+}+\varepsilon)t}\norm{\nabla F(x(t))}^{2}$ is integrable in $[t_{0},+\infty)$. By combining the two cases we conclude with the estimate \eqref{estimate grad basic teo}.
\end{proof}

\subsection{Convergence analysis for the perturbed system}
In this paragraph we provide the proofs of Theorems \ref{basicteo2inexact} and \ref{basicteo2inexact q} for the perturbed system \eqref{eq:DINperturbed}. 

The proof of Theorem \ref{basicteo2inexact} follows closely the one of Theorem \ref{basicteorates} and is based in the same Lyapunov function $V$ defined in \eqref{Lyapunov def}.

\begin{proof}[\textbf{Proof of Theorem \ref{basicteo2inexact}}]
	Let $(x(t))_{t\geq 0}$ be the trajectory generated by system \eqref{eq:DINperturbed}  and $V(t)$ the energy defined in \eqref{Lyapunov def}. By taking the time-derivative, for all $t\geq t_{0}$, we have:
	\begin{equation}
		\begin{aligned}
			\dot{V}(t) & = a\inner{\nabla F(x(t))}{\dot{x}(t)} + \beta\inner{\nabla^{2}F(x(t))\dot{x}(t)}{\dot{x}(t)} +\beta\inner{\nabla F(x(t))}{\ddot{x}(t)} \\ & \quad +\beta^{2} \inner{\nabla^{2}F(x(t))\dot{x}(t)}{\nabla F(x(t))} 
			+ \inner{\dot{x}(t)}{\ddot{x}(t)} \\
			& = (a-\alpha\beta -1)\inner{\nabla F(x(t))}{\dot{x}(t)} -\beta \norm{\nabla F(x(t))}^{2} -\alpha \norm{\dot{x}(t)}^{2}  -\inner{\beta\nabla F(x(t))+\dot{x}(t)}{g(t)}
		\end{aligned}
	\end{equation}
	
	By performing the exact same computations as the ones in the proof of Lemma \ref{lemma general conditions} and using condition \hyperref[PL]{$\Lb(2)$} and $\mu>0$ and the Cauchy-Schwarz inequality for the scalar product $\inner{\beta\nabla F(x(t)+\dot{x}(t))}{g(t)}$, it follows:
	\begin{equation}\label{perturbed 1}
		\begin{aligned}
			\dot{V}(t) & \leq -RV(t) +\frac{1}{2\mu}\left(aR - (a+1-\alpha\beta )\mu\beta\right)\norm{\nabla F(x(t))}^{2} +\frac{1}{2\beta}\left(1-\alpha\beta -a\right)\norm{\dot{x}(t)}^{2}\\ & \quad +\norm{\beta\nabla F(x(t))+\dot{x}(t)}\norm{g(t)}
		\end{aligned}
	\end{equation}
	where $R=\frac{1}{\beta}(1+\alpha\beta-a)$.
	
	Thus, if condition \eqref{conditions on a/d} as stated in Lemma \ref{lemma general conditions} holds true, from \eqref{perturbed 1}, it follows
	\begin{equation}\label{perturbed 2}
		\dot{V}(t)  \leq -RV(t) +\norm{\beta\nabla F(x(t)+\dot{x}(t))}\norm{g(t)}
	\end{equation}
	By using Gronw\"all's lemma (see Lemma \ref{lemmagronwallcontinuous}), from \eqref{perturbed 2}, it follows
	\begin{equation}\label{perturbed 3}
		\begin{aligned}
			V(t) & \leq V(t_{0})e^{-R(t-t_{0})} + e^{-Rt}\int_{t_{0}}^{t}e^{Rs}\norm{\beta\nabla F(x(s)+\dot{x}(s))}\norm{g(s)} ds
		\end{aligned}
	\end{equation}
		In this point, without loss of generality we may assume that $F(x(t)) \geq F(\bar{x})$, since in the case $F(x(t)) \leq F(\bar{x})$, the estimate \eqref{estimate gap perturbed 2} holds trivially true.
Thus, for all $t\in I:=\left\{t\geq t_{0}~:~F(x(t) \geq F(\bar{x}))\right\}$, from \eqref{perturbed 3}, we obtain
	\begin{equation}\label{perturbed 4}
		\begin{aligned}
			\frac{e^{Rt}}{2}\norm{\beta\nabla F(x(t)) + \dot{x}(t)}^{2} & \leq	V(t_{0})e^{Rt_{0}} + \int_{t_{0}}^{t}e^{Rs}\norm{g(s)}\norm{\beta\nabla F(x(s)+\dot{x}(s))} ds		
		\end{aligned}
	\end{equation}

	By using Lemma \ref{lemmaBihari} with $\psi(t)=\frac{e^{Rt}}{2}\norm{\beta\nabla F(x(t)) + \dot{x}(t)}^{2}$, $a=V(t_{0})e^{Rt_0}$, $k(t) = e^{\frac{Rt}{2}}\norm{g(t)}$ and $\omega(u) = \sqrt{2u}$, from \eqref{perturbed 4}, it follows:
	\begin{equation}\label{perturbed 5}
		e^{\frac{Rt}{2}}\norm{\beta\nabla F(x(t)) + \dot{x}(t)} \leq \sqrt{2V(t_{0})e^{Rt_{0}}} + \int_{t_{0}}^{t}e^{\frac{Rs}{2}}\norm{g(s)}ds
	\end{equation}
By setting $J(t)=\int_{t_{0}}^{t}e^{\frac{Rs}{2}}\norm{g(s)}ds$ and injecting the estimate \eqref{perturbed 5} into \eqref{perturbed 3}, it follows
	\begin{equation}
		\begin{aligned}
			V(t) &\leq V(t_{0})e^{-R(t-t_{0})} + e^{-Rt}\int_{t_{0}}^{t}e^{\frac{Rs}{2}}\norm{g(s)}\left(\sqrt{2V(t_{0})e^{Rt_{0}}} + J(s) \right)ds \\
			& =\left(V(t_{0})e^{Rt_{0}} + \sqrt{2V(t_{0})e^{Rt_{0}}}J(t) +\frac{1}{2}J^{2}(t) \right)e^{-Rt} = \left(\sqrt{V(t_{0})e^{Rt_{0}}} + \frac{J(t)}{\sqrt{2}}\right)^{2}e^{-Rt}
		\end{aligned}
	\end{equation} 
	which, by definition of  $V(t)$, yields
	\begin{equation}\label{perturbed 6}
		F(x(t)) - F(\bar{x}) \leq \left(\sqrt{\frac{V(t_{0})e^{Rt_{0}}}{a}} + \frac{J(t)}{\sqrt{2a}}\right)^{2}e^{-Rt}
	\end{equation}
	
	Since \eqref{perturbed 6} holds true with $R=\frac{1}{\beta}(1+\alpha\beta-a)$ under condition \eqref{conditions on a/d} of Lemma \ref{lemma general conditions}, by following the same lines of the proof of Theorem \ref{basicteorates}, and in particular using Lemmas \ref{lemma lyapunov2} and \ref{lemma lyapunov3}, it follows:
	\begin{equation}\label{perturbed 7}
		\begin{aligned}
			F(x(t)) & -  F(\bar{x}) \leq \\
			& \leq  \left( e^{\frac{\Rb(\alpha,\beta)t_{0}}{2}}\sqrt{F(x(t_{0})) - F(\bar{x}) +\frac{\norm{\beta\nabla F(x(t_{0}))+\dot{x}(t_{0})}^{2}}{2C(\alpha,\beta)}} + \frac{\Jb(t)}{\sqrt{2C(\alpha,\beta)}}\right)^{2}e^{-\Rb(\alpha,\beta) t}
		\end{aligned}
	\end{equation} 
	with $\Rb(\alpha,\beta)$ and $C(\alpha,\beta)$ as given in \eqref{rate factor on alphabeta} and \eqref{constant on alphabeta} (resp.) and $\Jb(t)= \int_{t_{0}}^{t}e^{\frac{\Rb(\alpha,\beta)s}{2}}\norm{g(s)}ds$.

	Finally, in the same way as in the proof of Theorem \ref{basicteo2inexact}, if the binomial $\varepsilon(a,\alpha,\beta)=a^2 - \left((\alpha-\mu\beta)\beta +1\right)a + (1-\alpha\beta)\mu\beta^{2}$, is strictly positive, from \eqref{perturbed 1}, if $F(x(t)) \geq F(\bar{x})$, for all $t\geq t_{0}$, it follows
	\begin{equation}\label{perturbed 2 final}
		\begin{aligned}
			\varepsilon(a,\alpha,\beta)\int_{t_{0}}^{t}e^{Rs}\norm{\nabla F(x(s)}^{2}ds & \leq e^{Rt_{0}}V(t_{0}) + \int_{t_{0}}^{t}e^{Rs}\norm{g(s)}\norm{\beta\nabla F(x(s))+\dot{x}(t)}ds	\\
			&\leq \left(\sqrt{V(t_{0})e^{Rt_{0}}} + \frac{J(t)}{\sqrt{2}}\right)^{2}
		\end{aligned}
	\end{equation}
	where in the last inequality we used the estimate \eqref{perturbed 5}.
	
	In addition, note that when  $\alpha \in (0,\sqrt{\mu})$ and $\beta\in \left(\frac{\alpha}{\mu},\frac{1}{\alpha}\right)$, the binomial $\varepsilon(a,\alpha,\beta) = a^2 - \left((\alpha-\mu\beta)\beta +1\right)a + (1-\alpha\beta)\mu\beta^{2}$ evaluated at $a=1-\alpha\beta>0$ is strictly positive. In the same way, for any $\eta>0$ sufficiently small, for all $\alpha>0$ and $\beta\in\left(0,\frac{\alpha}{\mu}\right)\cup\left(\frac{1}{\alpha},+\infty\right)$, the binomial $\varepsilon(a,\alpha,\beta)$ evaluated at $a=y_{+}=\frac{1}{2}\left((\alpha-\mu\beta)\beta+1 + \sqrt{\left((\alpha-\mu\beta)\beta+1\right)^{2}-4(1-\alpha\beta)\mu\beta^{2}}\right)+\eta$ is strictly positive (note that $y_{+}$ is the largest root of $\varepsilon(a,\alpha,\beta)$). By combining the two cases ($a=1-\alpha\beta$ and $a=y_{+}+\eta$) in \eqref{perturbed 2 final}, we conclude with the estimate \eqref{estimate grad basic teo first inexact}.
\end{proof}

Next, we provide the proof of Theorem \ref{basicteo2inexact q}, concerning functions satisfying condition \ref{PL} with $q\in(1,2)$ for the system \eqref{eq:DINperturbed}.

\begin{proof}[\textbf{Proof of Theorem \ref{basicteo2inexact q}}]
	

	First, note that without loss of generality we may assume that for all $t\geq t_{0}$, it holds $\norm{\nabla F(x(t))}\leq 1$ and $\norm{\dot{x}(t)}\leq 1$ (since $x(t)$ and $\dot{x}(t)$ are bounded by assumption and $F$ is $\Cb^{2}$).

	By taking the energy function $V$ as defined in \eqref{Lyapunov def}, associated to the trajectory generated by system \eqref{eq:DINperturbed} $(x(t))_{t\geq 0}$, with $a=\alpha\beta+1$, it follows:
	\begin{equation}
		\begin{aligned}
			V(t) & = \left(\alpha\beta+1\right)\left(F(x(t))-F(\bar{x})\right) +\beta\inner{\nabla F(x(t))}{\dot{x}(t)}+\frac{\beta^{2}}{2}\norm{\nabla F(x(t))}^{2}+\frac{1}{2}\norm{\dot{x}(t)}^{2} \\
			& \leq \left(\alpha\beta+1\right)\left(F(x(t))-F(\bar{x})\right) +\beta^{2}\norm{\nabla F(x(t))}^{2}+\norm{\dot{x}(t)}^{2} \\
			& \leq \frac{\alpha\beta+1}{2\mu}\norm{\nabla F(x(t))}^{q}  +\beta^{2}\norm{\nabla F(x(t))}^{2}+\norm{\dot{x}(t)}^{2}\\
			&\leq \left(\frac{\alpha\beta+1}{2\mu}+\beta^{2}\right)\norm{\nabla F(x(t))}^{q} +\norm{\dot{x}(t)}^{q}
		\end{aligned}
	\end{equation}
	where in the first inequality we used the Young's inequality for the product $\inner{\beta\nabla F(x(t))}{\dot{x}(t)}$, in the second one, the \L ojasiewicz condition \ref{PL} and in the third one the fact that $\norm{\nabla F(x(t))}\leq 1$ and $\norm{\dot{x}(t)}\leq 1$ and $q\in (1,2)$. 
	
	On the one hand, by using the convexity of the function $\phi(s)=s^{\frac{2}{q}}$ (since $\frac{2}{q}>1$), from the previous inequality, it follows
	\begin{equation}\label{proof q 1}
		V(t)^{\frac{2}{q}} \leq \max\left\{1,\left(\frac{\alpha\beta+1}{2\mu }+\beta^{2}\right)^{\frac{2}{q}}\right\}\left(\norm{\nabla F(x(t))}^{2} +\norm{\dot{x}(t)}^{2}\right)
	\end{equation}	
On the other hand, by taking the time-derivative of $V$ and using that $x(t)$ is a solution of \eqref{eq:DINperturbed}, it holds (see in particular relation \eqref{lyapunov basic} in the proof of Lemma \ref{lemma general conditions}, with $a=\alpha\beta+1$):
	\begin{equation}\label{proof q 2}
		\begin{aligned}
			\dot{V}(t) &= -\beta\norm{\nabla F(x(t))}^{2} -\alpha\norm{\dot{x}(t)}^{2} -\inner{\beta\nabla F(x(t)) + \dot{x}(t)}{g(t)} \\ &\leq -\min\{\alpha,\beta\}\left(\norm{\nabla F(x(t))}^{2} +\norm{\dot{x}(t)}^{2}\right)
			-\inner{\beta\nabla F(x(t))+\dot{x}(t)}{g(t)}
		\end{aligned}
	\end{equation} 
	
	From \eqref{proof q 1} and \eqref{proof q 2}, we deduce that	
	\begin{equation}\label{proof q 3}
		\begin{aligned}
			\dot{V}(t) & \leq -C_{3}V(t)^{\frac{2}{q}}  -\inner{\beta\nabla F(x(t))+\dot{x}(t)}{g(t)} \\
			& \leq -C_{3}V(t)^{\frac{2}{q}} + \norm{g(t)}\norm{\beta\nabla F(x(t))+\dot{x}(t)}
		\end{aligned}
	\end{equation}
	with $C_{3}=\frac{\min\{\alpha,\beta\}}{\max\left\{1,\left(\frac{\alpha\beta+1}{2\mu }+\beta^{2}\right)^{\frac{2}{q}}\right\}}$.	
	Applying Lemma \ref{lemma nonlinear gronwall power p} with $p=\frac{2}{q}>1$, we infer
	\begin{equation}\label{proof q 4}
		t^{\frac{q}{2-q}}V(t) \leq  C_{4}+ \int_{t_{0}}^{t}s^{\frac{q}{2-q}}\norm{g(s)}\norm{\beta\nabla F(x(s))+\dot{x}(s)}ds
	\end{equation}
	where $C_{4} = \max\left\{\left(\frac{q}{C_{3}(2-q)}\right)^{\frac{q}{2-q}} , t_{0}^{\frac{q}{2-q}}V(t_{0})\right\}$.\\
	Without loss of generality we may assume that $F(x(t)) \geq F(\bar{x})$ (since if $F(x(t)) \leq F(\bar{x})$, then \eqref{estimate gap perturbed q} holds trivially true), thus, for all $t\in \left\{t\geq t_{0}~:~F(x(t) \geq F(\bar{x}))\right\}$, from \eqref{proof q 4}, it holds
	\begin{equation}\label{proof q 5}
		\frac{t^{\frac{q}{2-q}}}{2}\norm{\beta\nabla F(x(t))+\dot{x}(t)}^{2} \leq  C_{4}+ \int_{t_{0}}^{t}s^{\frac{q}{2(2-q)}}\norm{g(s)}s^{\frac{q}{2(2-q)}}\norm{\beta\nabla F(x(s))+\dot{x}(s)}ds
	\end{equation}
	According to Lemma \ref{lemmaBihari} with $\omega(u)=\sqrt{2u}$, $k(t)=t^{\frac{q}{2(2-q)}}\norm{g(t)}$ and $\psi(t) = \frac{t^{\frac{q}{2-q}}}{2}\norm{\beta\nabla F(x(t))+\dot{x}(t)}^{2}$, from \eqref{proof q 5}, it follows:
	\begin{equation}\label{proof q 6}
		t^{\frac{q}{2(2-q)}}\norm{\beta\nabla F(x(t))+\dot{x}(t)} \leq \sqrt{2C_{4}} +\int_{t_{0}}^{t}s^{\frac{q}{2(2-q)}}\norm{g(s)}ds
	\end{equation}
	By injecting the estimate \eqref{proof q 6} in \eqref{proof q 4}, it follows
	\begin{equation}
		V(t) \leq  \left(\sqrt{C_{4}} +\frac{\Ib(t)}{\sqrt{2}}\right)^{2}t^{-\frac{q}{2-q}}
	\end{equation}
	which concludes the proof of Theorem \ref{basicteo2inexact q}, with $\Ib(t)=\int_{t_{0}}^{t}s^{\frac{q}{2(2-q)}}\norm{g(s)}ds$ and $C_{4} = \max\left\{\left(\frac{q}{C_{3}(2-q)}\right)^{\frac{q}{2-q}} , t_{0}^{\frac{q}{2-q}}V(t_{0})\right\}$.
\end{proof}

\subsection{First-order coupled system}
In this paragraph we provide the proofs of Theorems \ref{basicteorates firstorder} and \ref{basicteorates firstorder q} concerning the first-order coupled system \eqref{HBHDfirstorder inexact}. The proofs follow the same analysis as the one in Theorems \ref{basicteo2inexact} and \ref{basicteo2inexact q} and are presented in short for the sake of brevity.

\begin{proof}[\textbf{Proof of Theorem \ref{basicteorates firstorder}}]
Let $\left(x(t),y(t)\right)_{t\geq t_{0}}$ be  the coupled trajectory generated by the system \eqref{HBHDfirstorder inexact}, such that $(x(t))_{t\geq t_{0}}\in \Omega$, where $\Omega$ is a neighborhood of $\bar{x}\in \crit{F}$, such that $F$ satisfies \hyperref[PL]{$\Lb(2)$} with $\mu>0$ in $\Omega$. 
For any $a\geq0$, we define the following energy function:
\begin{equation}\label{def lyapunov U}
	U(t)=a\left(F(x(t))-F(\bar{x})\right) + \frac{1}{2}\norm{\left(\alpha-\frac{1}{\beta}\right)x(t) + \frac{1}{\beta}y(t)}^{2}
\end{equation}
By taking the derivative of $U$ with respect to time in \eqref{def lyapunov U} and using \eqref{HBHDfirstorder inexact}, it follows:
\begin{equation}\label{for the proof first order1}
	\begin{aligned}
		\dot{U}(t)
		& = - a\beta\norm{\nabla F(x(t)}^{2} -(a+\alpha\beta-1)\inner{\nabla F(x(t))}{\left(\alpha-\frac{1}{\beta}\right)x(t) + \frac{1}{\beta}y(t)} \\
		& \quad - \alpha\norm{\left(\alpha-\frac{1}{\beta}\right)x(t) + \frac{1}{\beta}y(t)}^{2} + \inner{g(t)}{\left(\alpha-\frac{1}{\beta}\right)x(t) + \frac{1}{\beta}y(t)} \\
		& = -\frac{\beta}{2}\left(a+1-\alpha\beta\right)\norm{\nabla F(x(t))}^{2} -\frac{\left(a+\alpha\beta-1\right)}{2\beta}\norm{\beta\nabla F(x(t)) + \left(\alpha-\frac{1}{\beta}\right)x(t) + \frac{1}{\beta}y(t)}^{2}  \\
		& \quad  - \frac{1}{2\beta}\left(\alpha\beta+1-a\right)\norm{\left(\alpha-\frac{1}{\beta}\right)x(t) + \frac{1}{\beta}y(t)}^{2} + \inner{g(t)}{\left(\alpha-\frac{1}{\beta}\right)x(t) + \frac{1}{\beta}y(t)}
	\end{aligned} 
\end{equation}
where in the last equality, we used the identity $\inner{u}{v} = \frac{1}{2}\norm{\sqrt{\beta}u+\frac{1}{\sqrt{\beta}}v}^{2} - \frac{\beta}{2}\norm{u}^{2}-\frac{1}{2\beta}\norm{v}^{2}$, with $u =\nabla F(x(t))$ and $v=\left(\alpha-\frac{1}{\beta}\right)x(t) + \frac{1}{\beta}y(t)$.
	
From the definition of $U$ in \eqref{def lyapunov U} and \eqref{for the proof first order1}, if $a\leq 1+\alpha\beta$, by using condition \hyperref[PL]{$\Lb(2)$}  and the Cauchy-Schwarz inequality for the scalar product $\inner{g(t)}{\left(\alpha-\frac{1}{\beta}\right)x(t) + \frac{1}{\beta}y(t)}$, it follows:
\begin{equation}\label{for the proof first order2}
	\begin{aligned}
		\dot{U}(t)
		& \leq -RU(t) -\frac{1}{2}\left(a+\alpha\beta-1\right)\norm{\sqrt{\beta}\nabla F(x(t))\frac{1}{\sqrt{\beta}}\left(\left(\alpha-\frac{1}{\beta}\right)x(t) + \frac{1}{\beta}y(t)\right)}^{2}  \\
		& \quad  -\frac{1}{2\mu}\left((a+1-\alpha\beta)\mu\beta-aR\right)\norm{\nabla F(x(t))}^{2} +\norm{g(t)}\norm{\left(\alpha-\frac{1}{\beta}\right)x(t) + \frac{1}{\beta}y(t)}
	\end{aligned}
\end{equation}
where we set $R=\frac{1}{\beta}(\alpha\beta+1-a)$.
	
From \eqref{for the proof first order2}, if we additionally impose that $a\geq 1-\alpha\beta$ and $aR\leq (a+1-\alpha\beta)\mu\beta$, it follows that $\dot{U}(t) \leq -RU(t) + \norm{g(t)}\norm{\left(\alpha-\frac{1}{\beta}\right)x(t) + \frac{1}{\beta}y(t)}$ and thus
\begin{equation}\label{U in proof 1}
	e^{Rt}U(t) \leq U(t_{0})e^{Rt_{0}} +\int_{t_{0}}^{t}e^{Rs}\norm{\beta\nabla F(x(s)+\dot{x}(s))}\norm{g(s)} ds
\end{equation}
In the same way as in the proof of Theorem \ref{basicteo2inexact}, by using Lemma \ref{lemmaBihari},  from \eqref{U in proof 1} it follows
\begin{equation}
	U(t) \leq \left(\sqrt{U(t_{0})e^{Rt_{0}}} + \frac{\int_{t_{0}}^{t}e^{\frac{Rs}{2}}\norm{g(s)}ds}{\sqrt{2}}\right)^{2}e^{-Rt}
\end{equation}
and consequently:
\begin{equation}\label{eq: geometric first in the proof}
	F(x(t))-F(\bar{x}) \leq \left(\sqrt{U(t_{0})e^{Rt_{0}}} + \frac{J(t)}{\sqrt{2}}\right)^{2}e^{-\frac{1}{\beta}(\alpha\beta+1-a)t}
\end{equation}
with $J(t)=\int_{t_{0}}^{t}e^{\frac{Rs}{2}}\norm{g(s)}ds$.

The overall imposed conditions on the parameters $a$, $\alpha$ and $\beta$ for \eqref{eq: geometric first in the proof} to hold true, are exactly the same as the ones in \eqref{conditions on a/d} of Lemma \ref{lemma general conditions}. Thus, in the same spirit as for the proof of Theorem \ref{basicteorates}, by using Lemmas \ref{lemma lyapunov2} and \ref{lemma lyapunov3}, from \eqref{eq: geometric first in the proof}, it follows that
\begin{equation}\label{bound on F in the proof of Thm3}
	F(x(t))- F(\bar{x}) \leq \frac{\left(\sqrt{U(t_{0})e^{\Rb(\alpha,\beta)t_{0}}} + \frac{\Jb(t)}{\sqrt{2}}\right)^{2}}{C(\alpha,\beta)}e^{-\Rb(\alpha,\beta)t}
\end{equation} 
with $\Rb(\alpha,\beta)$ and $C(\alpha,\beta)$ as given in \eqref{rate factor on alphabeta} and \eqref{constant on alphabeta} (resp.) and $\Jb(t)=\int_{t_{0}}^{t}e^{\frac{\Rb(\alpha,\beta)s}{2}}\norm{g(s)}ds$.

The last part of the Theorem, concerning the estimate on $\norm{\nabla F(x(t))}$ is identical with the one in the proof of Theorem \ref{basicteo2inexact} and is left to the reader.

\end{proof}

\begin{proof}[\textbf{Proof of Theorem \ref{basicteorates firstorder q}}] 
Since $x(t)$ is bounded, from the Lipschitz character of $\nabla F$, it follows that $\nabla F(x(t))$ is also bounded and without loss of generality we may assume that $\norm{\dot{x}(t)}\leq 1$ and $\norm{\nabla F(x(t))}\leq 1$ for all $t\geq t_{0}$.
	
By using \eqref{HBHDfirstorder inexact} in the definition of $U$ in \eqref{def lyapunov U}, with $a=\alpha\beta+1$, it follows:
\begin{equation}\label{eq: proof first q general}
	\begin{aligned}
		U(t)
		& \leq  \frac{\alpha\beta+1}{2\mu}\norm{\nabla F(x(t))}^{q} + \norm{\dot{x}(t)}^{2} +\beta^{2}\norm{\nabla F(x(t))}^{2}  \\
		& \leq  \left(\frac{\alpha\beta+1}{2\mu}+\beta^{2}\right)\norm{\nabla F(x(t))}^{q} + \norm{\dot{x}(t)}^{q} 
	\end{aligned}
\end{equation}
where in the first inequality we used condition \ref{PL} with $q\in(1,2)$ and the convexity of the function $\phi(s)=s^{2}$ and in the second one the fact that $\norm{\nabla F(x(t))}\leq 1$ and $\norm{\dot{x}(t)}\leq 1$ and $q\in (1,2)$. 
From \eqref{eq: proof first q general}, since $q\in(1,2)$, it follows that
\begin{equation}\label{proof q first 1}
	U(t)^{\frac{2}{q}} \leq C_{1}\left(\norm{\nabla F(x(t))}^{2} +\norm{\dot{x}(t)}^{2}\right)
\end{equation}	
with $C_{1}=\max\left\{1,\left(\frac{\alpha\beta+1}{2\mu}+\beta^{2}\right)^{\frac{2}{q}}\right\}$.
	
Moreover, by using \eqref{HBHDfirstorder inexact} in \eqref{for the proof first order1}, with $a=\alpha\beta+1$, it follows:
\begin{equation}\label{proof general q final}
	\begin{aligned}
		\dot{U}(t) &= -\beta\norm{\nabla F(x(t))}^{2} -\alpha\norm{\dot{x}(t)}^{2} -\inner{g(t)}{\dot{x}(t)+\beta\nabla F(x(t))} \\
		&\leq -C_{2}U(t)^{\frac{2}{q}}
			+\norm{g(t)}\norm{\dot{x}(t)+\beta\nabla F(x(t))}
	\end{aligned}
\end{equation}
with $C_{2}=\frac{\min\{\alpha,\beta\}}{C_{1}}$, where we used \eqref{proof q first 1} and the Cauchy-Schwarz inequality.
	
In the same way as in the proof of Theorem \ref{basicteo2inexact q}, by using Lemmas \ref{lemma nonlinear gronwall power p} and \ref{lemmaBihari}, from \eqref{proof general q final}, it follows
\begin{equation}
	t^{\frac{q}{2-q}}U(t) \leq \left(\sqrt{C_{4}}+\frac{\Ib(t)}{\sqrt{2}}\right)^{2} 
\end{equation}
which concludes the proof of \cref{basicteorates firstorder q}
with $\Ib(t)=\int_{t_{0}}^{t}s^{\frac{q}{2(2-q)}}\norm{g(s)}ds$ and\\ $C_{4} = \max\left\{\left(\frac{q\max\left\{1,\left(\frac{\alpha\beta+1}{2\mu}+\beta^{2}\right)^{\frac{2}{q}}\right\}}{(2-q)\min\{\alpha,\beta\}}\right)^{\frac{q}{2-q}} , t_{0}^{\frac{q}{2-q}}U(t_{0})\right\}$.	
\end{proof}

\subsection{Proof of Theorem \ref{thrm:avoidance_DIN}}
\label{subsec:avoidance_proof}

In this subsection we will prove Theorem \ref{thrm:avoidance_DIN}, which guarantees avoidance of strict saddle points by the dynamics of \eqref{eq:DIN}. 
Actually, we will infer this result from the avoidance of strict saddle points by \eqref{HBHDfirstorder}. 
The line of approach will be by applying lemma \ref{lemma:avoidstrong}, which provides conditions for the avoidance of strict saddle points, for general autonomous first-order ODEs. 
Before we proceed, we introduce some notation for the set of fixed points with strong escape directions, $\mathcal{S}$, and the set of initial conditions attracted by them, $\mathcal{M}_s$. 
Given a function $f\in\mathbb{R}^D\to\mathbb{R}$, consider a $D$-dimensional autonomous first-order ODE of general form 
\begin{equation} 
\label{eq:ODE}
    \dot{z} = f(z), \qquad z\in\mathbb{R}^D
\end{equation}
we can use the following definitions
\begin{definition}
    The set of fixed points of \eqref{eq:ODE} with strong escape directions is defined to be
    \begin{equation}
        \mathcal{S}(f)=\{\bar{z}\in\mathbb{R}^D:f(\bar{z})=0 \ \& \ \nabla f(\bar{z}) \ \text{has a strictly positive eigenvalue}\}
    \end{equation} 
    and the set of all initial conditions such that their solution-trajectories converge to $\mathcal{S}$ is defined through
    \begin{equation}
        \mathcal{M}_s(f) = \{z_0\in\mathbb{R}^D : \text{if} \ z(0)=z_0 \ \text{then under \eqref{eq:ODE}} \ z(t)\to\bar{z}\in\mathcal{S}\}.
    \end{equation} 
    We will often omit the explicit dependence on $f$. 
\end{definition}
Given these definitions we can state the following theorem
\begin{lemma}
\label{lemma:avoidstrong}
    Assuming $f \in C^1(\mathbb{R}^D)$ and that all solution-trajectories of \eqref{eq:ODE} converge, the set of all initial conditions that, under \eqref{eq:ODE}, are attracted by fixed points with strong escape directions, $\mathcal{M}_s(f)$, is of Lebesgue measure 0. 
\end{lemma} 
The proof of the above lemma is based on Carr's Theorem \cite{perko2013differential} and it is provided in Appendix \ref{app:avoidance}.

We aim to apply the above lemma to \eqref{HBHDfirstorder}, which means for $D=2d$ and $z=(x,y)$ with
\begin{equation}
\label{eq:gDIN_f}
    f(x,y) = 
	\begin{pmatrix}
	    - \beta\nabla F(x) - \left(\alpha-\frac{1}{\beta}\right)x - \frac{1}{\beta}y \\
		- \left(\alpha-\frac{1}{\beta}\right)x - \frac{1}{\beta}y
	\end{pmatrix}
\end{equation}
To this end, we need to connect the strict saddle points of $F$ with the set $\mathcal{M}_s$ whose avoidance is provided by lemma \ref{lemma:avoidstrong}. 
Since the fixed points of \eqref{eq:ODE}, with $f$ as defined in \eqref{eq:gDIN_f}, are the points $(\bar{x},(1-\alpha\beta)\bar{x})$ for all critical points $\bar{x}$ of $F$, we need to connect the eigenvalues of $\nabla^2 F(\bar{x})$ with the eigenvalues of $D=\nabla f(\bar{x},(1-\alpha\beta)\bar{x})$. 
Given the following representation
\begin{equation}
    D = 
    -
    \begin{pmatrix}
        \beta\nabla^2F(\bar{x}) & \mathbb{0}_d \\ 
        \mathbb{0}_d & \mathbb{0}_d
    \end{pmatrix}
    - \frac{1}{\beta}
    \begin{pmatrix}
        \left( \alpha\beta - 1 \right)\mathbb{I}_d & \mathbb{I}_d \\ 
        \left( \alpha\beta - 1 \right)\mathbb{I}_d & \mathbb{I}_d 
    \end{pmatrix},
\end{equation} 
we can derive the following lemma, which is also implicit in \cite{castera2023inertial}
\begin{lemma}[\cite{castera2023inertial}]
\label{lmm:eigenvalues}
    For any distinct eigenvalue $\lambda_0$ of $\nabla^2F(\bar{x})$, define $\lambda_+$ and $\lambda_-$ through
    \begin{equation}
        \lambda_{\pm} = \frac{1}{2} \left( -(a+\beta\lambda_0) \pm \sqrt{(a+\beta\lambda_0)^2 -4\lambda_0} \right)
    \end{equation}
    Then, $\lambda_+$ and $\lambda_-$ are two distinct eigenvalues of $D$. 
    Specifically, if $\nabla^2F_0(\bar{x})$ has a strictly negative eigenvalue, then $D$ has a strictly positive eigenvalue.
\end{lemma}

Finally, we need the following lemma which shows that the mapping between the systems \eqref{eq:DIN} and \eqref{HBHDfirstorder} is given by
\begin{equation}
\label{eq:mapping}
    h(x,v) = (x,(1-\alpha\beta)x - \beta^2\nabla F(x) - \beta v)
\end{equation}

\begin{lemma}{\cite[Theorem $6.1$]{alvarez2002second}}
\label{lemma:equivalence}
    Assume $F\in C^2(\mathbb{R}^d)$.
    Then $\{(x(t),v(t))\}_{t=0}^{+\infty}$ is a solution-trajectory for \eqref{eq:DIN} if and only if $\{h(x(t),v(t))\}_{t=0}^{+\infty}$ is a solution trajectory for \eqref{HBHDfirstorder}.
\end{lemma}

Now we have all the tools to proceed with a sort proof of Theorem \ref{thrm:avoidance_DIN}.

\begin{proof}[\textbf{Proof of Theorem \ref{thrm:avoidance_DIN}.}]
    This proof is an application of lemma \ref{lemma:avoidstrong} for the first order system \eqref{HBHDfirstorder} as described above.
    Since we assume $F\in C^2(\mathbb{R}^d)$, clearly we have $f\in C^1(\mathbb{R}^d)$.
    Additionally, from the \L ojasiewicz assumption and Proposition \ref{proposition basic}, all trajectories of \eqref{eq:DIN} converge to points of the form $(\bar{x},0)$. 
    By the equivalence of the two systerms \eqref{eq:DIN} and \eqref{HBHDfirstorder} for $F\in C^2(\mathbb{R}^d)$ (provided by lemma \ref{lemma:equivalence}) and by the continuity of $h$, convergence of solution-trajectories of $\eqref{eq:DIN}$ implies convergence of solution-trajectories of $\eqref{HBHDfirstorder}$.
    Specifically, if $\bar{x}$ is a strict saddle point of $F$ and a solution trajectory $\{(x(t),v(t))\}_{t=0}^{+\infty}$ of \eqref{eq:DIN} converges to $(\bar{x},0)$, then $\{h(x(t),v(t))\}_{t=0}^{+\infty}$ converges to $(\bar{x},(1-\alpha\beta)\bar{x} - \beta^2\nabla F(\bar{x}))$, thus, by lemma \ref{lmm:eigenvalues}, the trajectory $\{h(x(t),v(t))\}_{t=0}^{+\infty}$ belongs in $\mathcal{M}_s$. 
    It follows that the initial conditions that are attracted by strict saddle points under \eqref{eq:DIN} are included in the inverse image of $\mathcal{M}_s$ under $h$ defined in \eqref{eq:mapping}. 

    Since under our assumptions $f\in C^1(\mathbb{R}^d)$ and the solution-trajectories of \eqref{HBHDfirstorder} converge, the conditions of lemma \ref{lemma:avoidstrong} hold and as a result we know that the set $\mathcal{M}_s$ is of Lebesque measure 0. 
    Since $h$ is 1-1 and differentiable, its inverse image of a measure 0 set is also of measure 0 and this concludes the proof.  
\end{proof}

\section{Numerical experiments}
\label{sec:numerics}

In this section we illustrate the behavior of \eqref{eq:DIN} and the associated results in Theorem \ref{basicteorates}, on two simple practical examples. More precisely we compare the efficiency of \eqref{eq:DIN} in terms of objective function values for different choices of the parameters $\alpha$ and $\beta$ and the plain heavy ball with friction system \eqref{HBode}. Both experiments were performed in the Julia programming language, by using the ODE solver with the Runge–Kutta pairs of orders $5$ and $4$ method "Tsit$5$", with absolute tolerance $\thickapprox 10^{-12}$.

\subsection{Quadratic function}
In this first example, we consider the minimization problem of a quadratic function $F(x)=\inner{Ax}{x}$, where $A\in\R^{d\times d}$, $x\in \R^{d}$, $d=400$. $A$ is a random positive semi-definite matrix with non-empty kernel and some given numbers for the smallest positive (and largest) eigenvalue equal to $\mu$ (and $L$ respectively). As a consequence, the function is $L$-smooth and convex (but not \textit{strongly} convex), and satisfies condition \hyperref[PL]{$\mathcal{L}(2)$} and $\mu$. We fix $L=20$ and we make three choices for $\mu = 0.2$, $0.04$ and $0.02$, corresponding to a condition number $\kappa = \frac{L}{\mu}$ equal to $100$, $500$ and $1000$ (respectively). Here we compare four choices of pairs $(\alpha,\beta)$: $(2\sqrt{\mu},\frac{1}{2\sqrt{\mu}})$ as suggested in \cite{attouch2022first}, $(\sqrt{\frac{\mu}{2}},\sqrt{\frac{2}{\mu}})$, $(1,1)$ and $(\sqrt{\mu}-\eps,\frac{1}{\sqrt{\mu}})$ with $\eps=10^{-4}$, as suggested in Corollary \ref{cor optimal rates}. The results are reported in Figure \ref{Figure3}. In all cases, we observe that the choice $(\alpha,\beta)=(\sqrt{\mu}-\eps,\frac{1}{\sqrt{\mu}})$ has better performance among all the others and is also optimal according to the upper bound $\bigof{e^{-2\sqrt{\mu}t}}$ (black dashed line in Figure \ref{Figure3}). 

\begin{figure}[t]
\centering
\includegraphics[height=29ex]{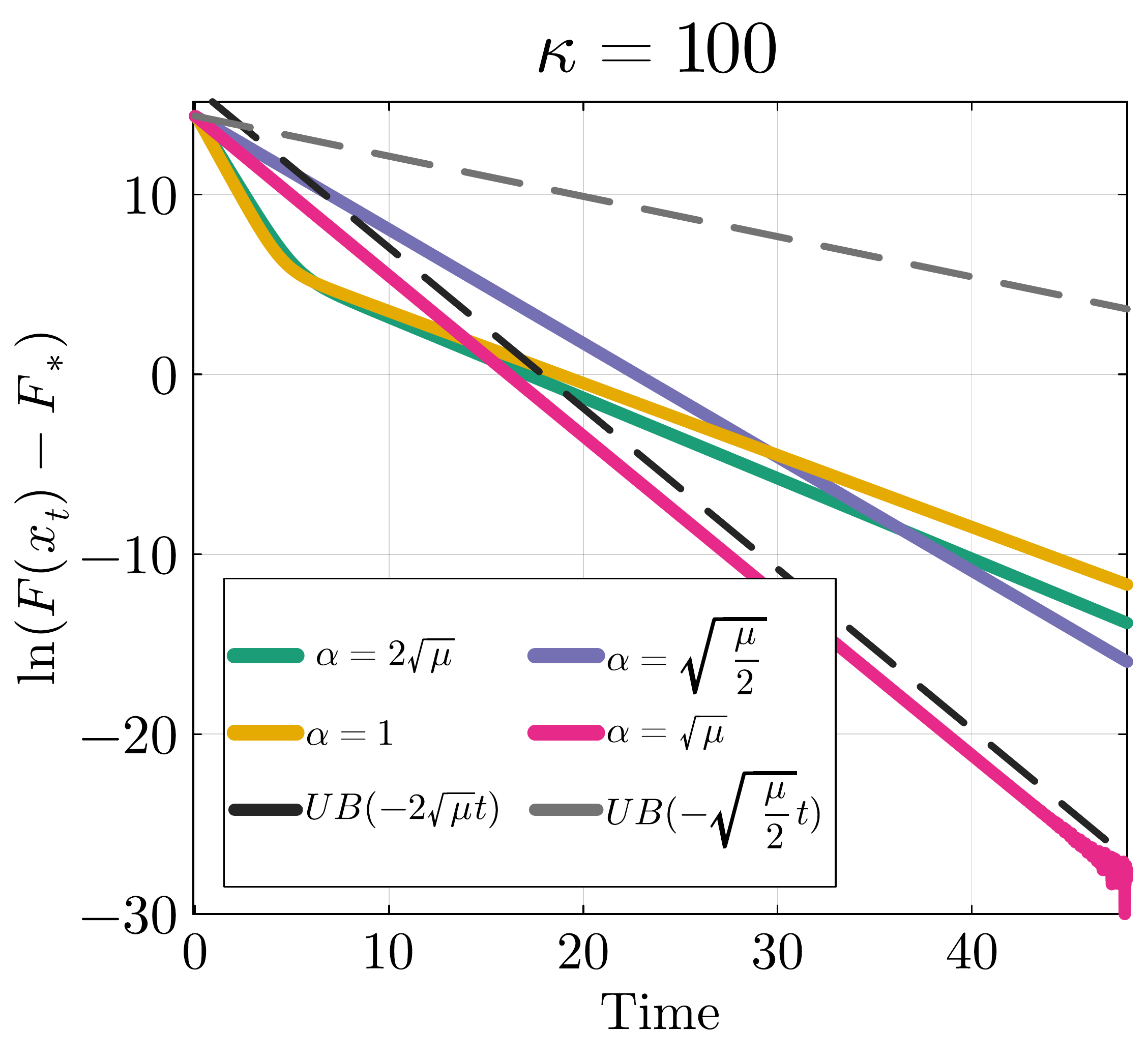}%
\includegraphics[height=29ex]{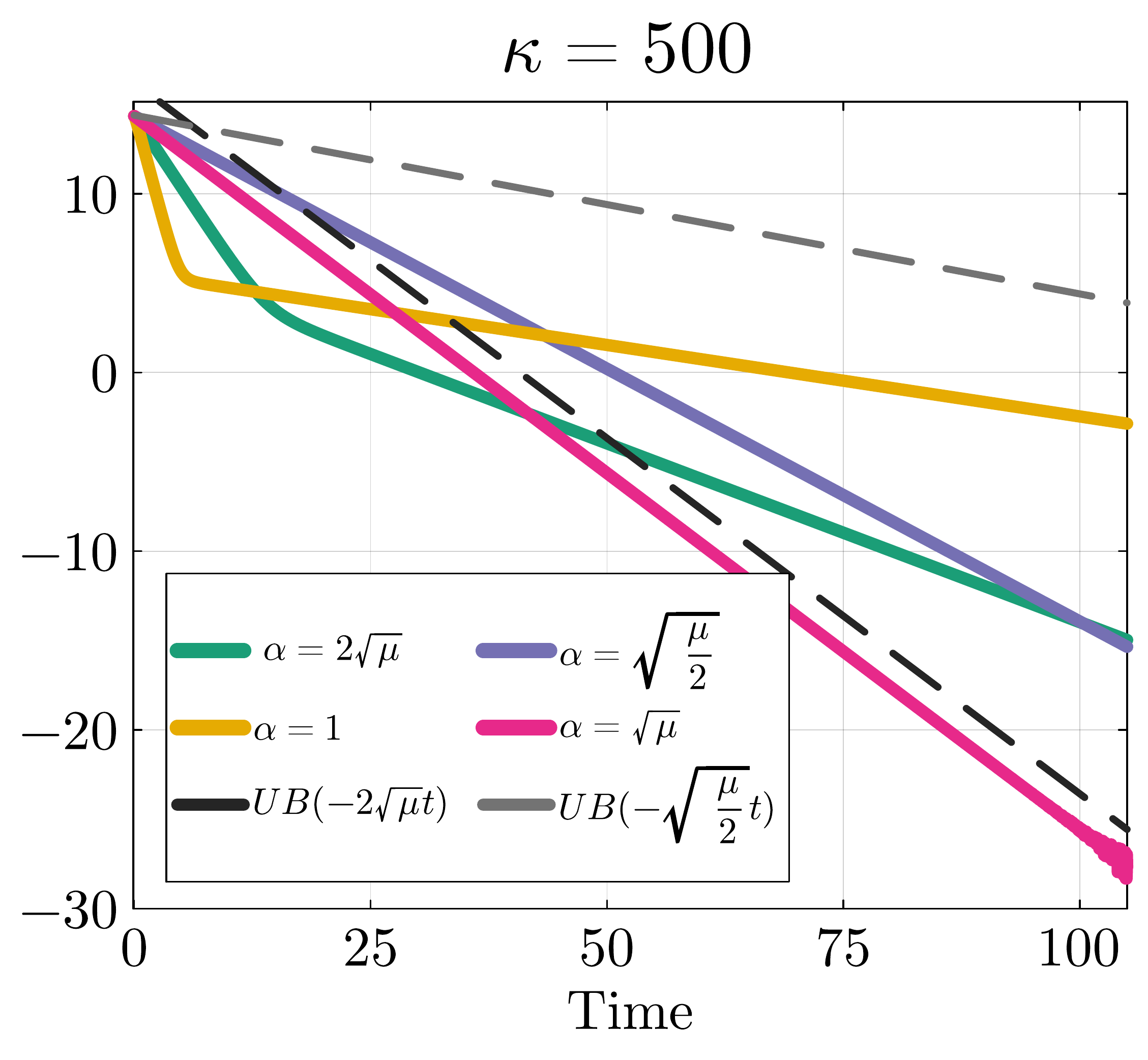}%
\includegraphics[height=29ex]{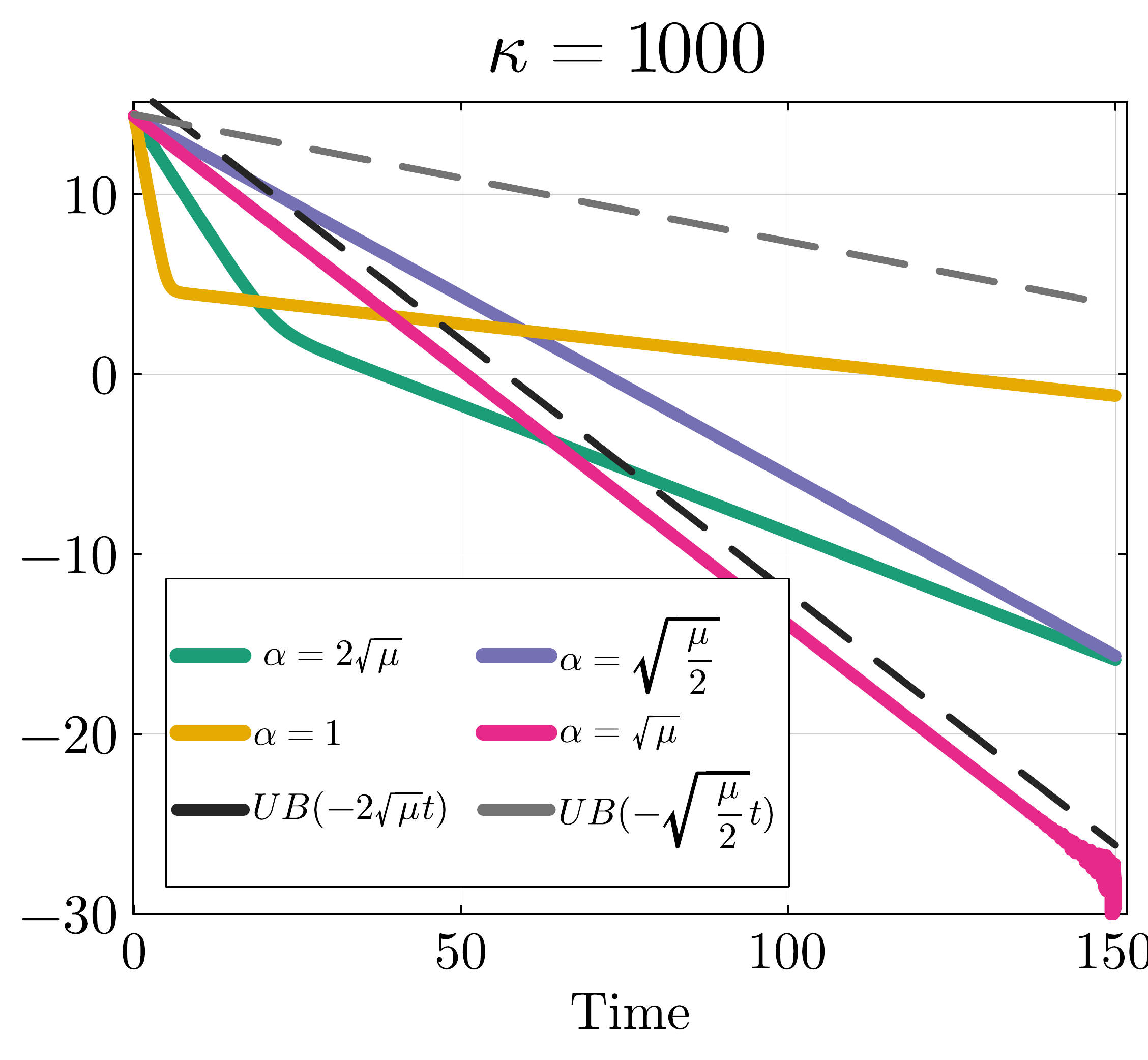}%
\caption{Performance of the trajectories generated by system \eqref{eq:DIN}, for minimizing a quadratic function $F(x)=\inner{Ax}{x}$, in terms of objective function values in logarithmic scale ($\ln\left(F(x(t))-F_{\ast}\right)$). Each column corresponds to the selection of a different condition number, starting from $\kappa=100$ (left) to $\kappa = 1000$ (right). We compare four different choices of friction parameters $\alpha>0$ for the system \eqref{eq:DIN}: $\alpha=2\sqrt{\mu}$ (green), $\alpha=1$ (yellow), $\alpha = \sqrt{\frac{\mu}{2}}$ (blue) and $\alpha =\sqrt{\mu}-\eps$ (magenta). The parameter $\beta$ is set to $\frac{1}{\alpha}$ in each case. In light gray and black dashed line the worst-case upper bounds for $F(x(t))-F_{\ast}$ corresponding to $\bigof{e^{-\sqrt{\frac{\mu}{2}}t}}$ and $\bigof{e^{-2\sqrt{\mu}t}}$ as found in \cite[Theorem $7$]{attouch2022first} and Corollary \ref{cor optimal rates}  (respectively).}\label{Figure3}
\end{figure}

\subsection{Rosenbrock function}
In this example we test the \eqref{eq:DIN} dynamics for the minimization problem of the Rosenbrock function
\begin{equation}
	\underset{(x,y)\in\R^{2}}{\min}F(x,y) = (1-x)^2 +100(y-x^{2})^{2}
\end{equation} 
The function is not convex, has a unique critical point which is the global minimum at $(x_{\ast},y_{\ast})=(1,1)$, has locally Lipschitz gradient $L\thicksim 501$ and satisfies condition \hyperref[PL]{$\Lb(2)$} and $\mu\thicksim 0.4$ (this can be seen by a Taylor approximation of $F$ or $\nabla F$ near $(x_{\ast},y_{\ast})$). Here we test system \eqref{eq:DIN} with the choices $(\alpha,\beta)=(2\sqrt{\mu},\frac{1}{2\sqrt{\mu}})$ as in \cite{attouch2022first} and $(\alpha,\beta)=(\sqrt{\mu}-\eps,\frac{1}{\sqrt{\mu}})$ with $\eps=10^{-4}$, as in Corollary \ref{cor optimal rates} and the plain heavy ball method \eqref{HBode} with $\alpha=2\left(\sqrt{\frac{L}{\mu}}-\sqrt{\frac{L}{\mu}-1}\right)\sqrt{\mu}$ as suggested in \cite{apidopoulos2022convergence} and $\alpha=2\sqrt{\mu}$ (see e.g. \cite{polyak1964some}, here we point out that the choice $\alpha=2\sqrt{\mu}$ is proven to be optimal only for strongly convex functions, however in this example it is still witnessed to be a good choice). In Figure \ref{Figure4} we can observe that system \eqref{eq:DIN} with $(\alpha,\beta)=(\sqrt{\mu}-\eps,\frac{1}{\sqrt{\mu}})$ performs better, without overshooting the minimizer $(x_{\ast},y_{\ast})$, in contrast to the heavy ball system with $\alpha=2\sqrt{\mu}$.

\begin{figure}[t]
\centering
\includegraphics[height=32ex]{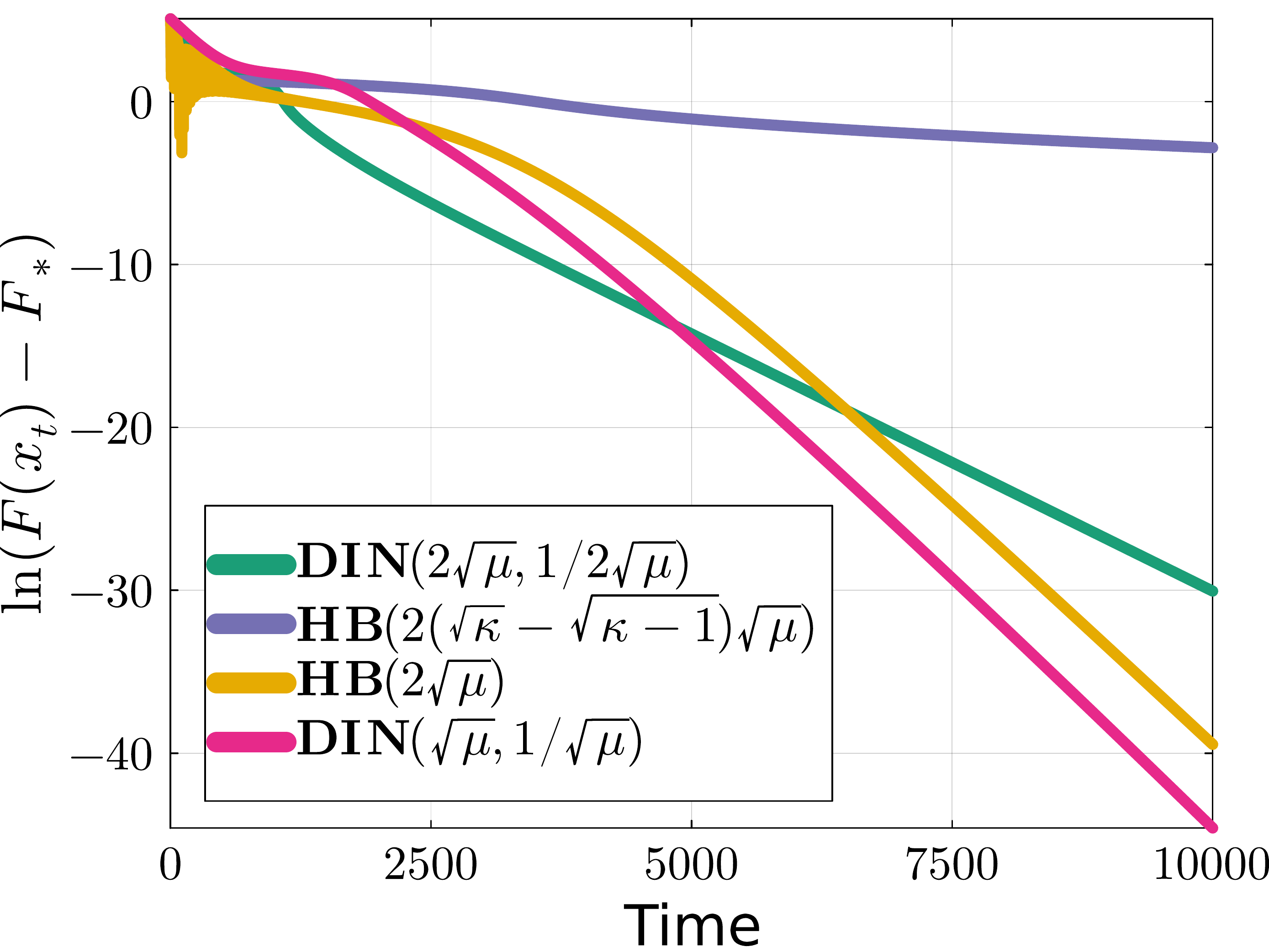}
\hfill
\includegraphics[height=32ex]{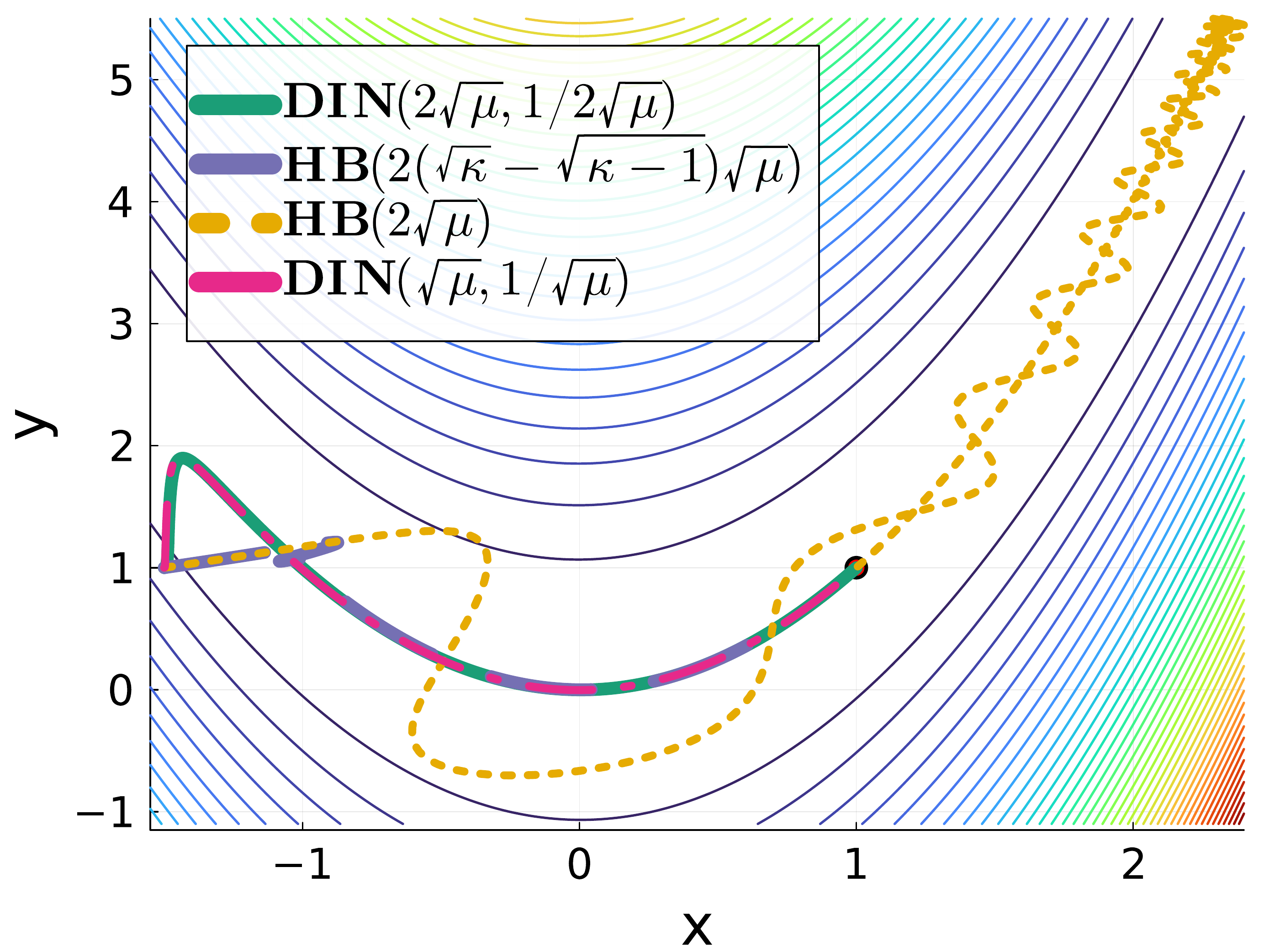}
\caption{Comparison of \eqref{eq:DIN} and the plain heavy ball with friction \eqref{HBode} on the Rosenbrock's function. The first column illustrates the values of the objective function (in logarithmic scale: $\ln\left(F(x(t))-F_{\ast}\right)$), while the second one the generated trajectories $(x(t),y(t))$ in $\R^{2}$. In green and magenta the \eqref{eq:DIN} system with $(\alpha,\beta)=\left(2\sqrt{\mu},\frac{1}{2\sqrt{\mu}}\right)$ and $(\alpha,\beta)=\left(\sqrt{\mu} - \eps,\frac{1}{\sqrt{\mu}}\right)$ (respectively). In blue and yellow lines the \eqref{HBode} system with $\alpha = 2\left(\sqrt{\kappa}-\sqrt{\kappa-1}\right)\sqrt{\mu}$ and $\alpha=2\sqrt{\mu}-\eps$. }\label{Figure4}
\end{figure}

\section{Concluding remarks}
\label{sec:conclusion}

In this work, we studied system \eqref{eq:DIN} in a smooth, non-convex setting, under the \L ojasiewicz condition (Assumption \ref{definition PL}). In particular we derived exponential decay estimates for the gap $F(x(t)) -F(\bar{x})$, that are optimal under appropriate tuning for the friction coefficients $\alpha$ and $\beta$, which are improved with respect to previous related works. Moreover we showed that these estimates are robust with respect to external perturbation errors. Finally, in the same setting, we provided an almost sure convergence result of the solution to \eqref{eq:DIN} (equivalently \eqref{HBHDfirstorder}) to local minima of the objective function $F$. 

An interesting future direction consists in extending all the above results to the algorithms associated to \eqref{eq:DIN} (or \eqref{HBHDfirstorder}), which can be seen as their finite difference scheme and thus may inherit their convergence properties.

\paragraph{\textbf{Acknowledgements}} This work has been partially supported by project MIS 5154714 of the National Recovery and Resilience Plan Greece 2.0 funded by the European Union under the NextGenerationEU Program.

\appendix
\crefalias{section}{appendix}
\crefalias{subsection}{appendix}
\numberwithin{lemma}{section}	
\numberwithin{corollary}{section}	
\numberwithin{proposition}{section}	
\numberwithin{equation}{section}	

\section{Auxiliary lemmas}
\label{app:aux}


\begin{lemma}[Gr\"onwall's Lemma]\label{lemmagronwallcontinuous}
	Let $I\subset \R_{+}=[0,+\infty)$ be an inteval and $u$,$g$,$h$ $:I\longrightarrow \R$, with $h\in L^{1}_{loc}(I)$, $g\in \mathcal{C}(I)$, and $u\in\mathcal{C}^{1}(I)$, such that for all $t\in I$: \begin{equation}\label{gronwalhypothesis}\dot{u}(t)\leq g(t)u(t)+h(t)
	\end{equation}
	Then for all $s<t\in I$, it holds: \begin{equation}\label{gronwalconclusion}
		u(t)\leq u(s)e^{G(t)}+\int_{s}^{t}e^{G(t)-G(r)}h(r)dr
	\end{equation}
	where $G(t)=\int_{s}^{t}g(r)dr$.
\end{lemma}

\begin{lemma}\label{lemma nonlinear gronwall power p}
Let $p>1$, $C>0$ and $t_{0}>0$ be three constants, $I=[t_{0},+\infty)$ and $u$,$G$ $:I\longrightarrow \R$, with $G\in L^{1}_{loc}(I;\R_{+})$ and $u\in\mathcal{C}^{1}(I)$, such that for all $t\in I$:
	\begin{equation}\label{relation nonlinear gronwall appendix}
		\dot{u}(t) \leq G(t) - Cu^{p}(t)
	\end{equation}
	Then, for all $t\geq t_{0}$, it holds:
\begin{equation}\label{bound nonlinear gronwall appendix}
	u(t) \leq \left(\max\left\{ \left(\frac{1}{C(p-1)}\right)^{\frac{1}{p-1}}, t_{0}^{\frac{1}{p-1}}u(t_{0}) \right\} + \int_{t_{0}}^{t}r^{\frac{1}{p-1}}G(r)dr \right)t^{-\frac{1}{p-1}}
\end{equation}
\end{lemma}
\begin{proof} 
We start by defining the following sets: 
$I_{1}:=\left\{t\in I ~:~ u^{p-1}(t) \leq \frac{1}{C(p-1)t} \right\}\cup\{t_{0}\}$ and $I_{2}:= \left\{t\in I ~:~ u^{p-1}(t) > \frac{1}{C(p-1)t} \right\}$, so that $I=I_{1}\cup I_{2}$.

Firstly, note that for all $t\in I_{1}$, \eqref{bound nonlinear gronwall appendix} holds immediately true.
On the other hand, for all $t\in I_{2}$, it holds  $Cu^{p}(t) \geq \frac{1}{(p-1)t}u(t)$. Therefore, by multiplying \eqref{relation nonlinear gronwall appendix} with $t^{\frac{1}{p-1}}$, it follows
\begin{equation}\label{eq proof nonlinear growall}
	t^{\frac{1}{p-1}}\dot{u}(t) +\frac{1}{p-1}t^{\frac{1}{p-1}-1}u(t) \leq t^{\frac{1}{p-1}}G(t)
\end{equation}

In addition, since $I_{2}$ is and open set in $I$ for all $t\in I_{2}$, there exists $\bar{t} \in I_{2}$ and $0<\varepsilon<t-t_{0}$, such that $t\in (\bar{t}-\varepsilon,\bar{t}+\varepsilon)\subset I_{2}$ (which without loss of generality is considered to be maximal, i.e. $\bar{t}-\varepsilon \in I_{1}$). 
By integrating \eqref{eq proof nonlinear growall} on $(\bar{t}-\varepsilon,t)\subset I_{2}$, it follows:
\begin{equation}
	\begin{aligned}
	t^{\frac{1}{p-1}}u(t) & \leq (\bar{t}-\varepsilon)^{\frac{1}{p-1}}u(\bar{t}-\varepsilon) + \int_{\bar{t}-\varepsilon}^{t}r^{\frac{1}{p-1}}G(s)ds\\
	& \leq 	\max\left\{t_{0}^{\frac{1}{p-1}}u(t_{0}), \left(C(p-1\right)^{-\frac{1}{p-1}}\right\} +\int_{t_{0}}^{t}s^{\frac{1}{p-1}}G(s) ds
	\end{aligned}
\end{equation}	
where the last inequality follows from the fact that $\bar{t} - \varepsilon \in I_{1}$ and that $G(t)\geq 0$.
\end{proof}

The next Lemma is a generalization of Gr\"onwall-Bellman Lemma
(see \cite{bihari1956generalization}).
\begin{lemma}\label{lemmaBihari}[Bihari]
	Let $\psi(t)$, $k(t)$ be positive and continous functions in $c < t < d$. Let $a$ and $b$ be non-negative constants; further let $\omega(u)$ be a positive non-decreasing function for $u > 0$. Then the inequality
	\begin{equation}\label{Biharihyp}
	\psi(t) \leq a+b\int_{c}^{t} k(s)\omega(\psi(s)) \ ds \ \ \ \ \text{for every} \ c\leq t \leq d
	\end{equation}
	implies
	\begin{equation}\label{bihariconclusion}
	\psi(t) \leq \Omega^{-1}\left[\Omega(a)+b\int_{c}^{t} k(s) \ ds\right] \ \ \ \ \text{for every} \ c\leq t \leq d'\leq d
	\end{equation}
	where
	$$\Omega(u):=\int_{\varepsilon}^{u} \frac{ds}{\omega(s)},  \ \ \ \ \text{for} \ \varepsilon > 0 \ \text{and} \ u>0$$
	and $d'$ characterizes the interval for which, $\Omega^{-1}$ is well defined.
\end{lemma}

\begin{lemma}\label{lemma binomial}
Let $\alpha$, $\beta$ and $\mu$ be some positive constants and  for all $y\geq 0$, consider the quadratic function 
\begin{equation}\label{P2y quadratic def}
	P_{2}(y)=y^{2} - \left((\alpha-\mu\beta)\beta+1\right)y + (1-\alpha\beta)\mu\beta^{2}
\end{equation}
 Then $P_{2}(y) \geq 0$, if-f the following condition holds true
\begin{equation}
\begin{cases}
\left(\alpha \leq 2(\sqrt{2}-1)\sqrt{\mu}\right)\&\left(\beta\in [\beta_{-},\beta_{+}]\right) \\
\text{ OR } \\
\left(\alpha \leq 2(\sqrt{2}-1)\sqrt{\mu}\right)\&\left(\beta\in (0,\beta_{-}]\cup[\beta_{+},+\infty)\right)\&\left(y\in (0,y_{-}]\cup[\max\{0,y_{+}\},+\infty)\right) \\
\text{ OR } \\
\left(\alpha > 2(\sqrt{2}-1)\sqrt{\mu}\right)\&\left(y\in (0,y_{-}]\cup[\max\{0,y_{+}\},+\infty)\right)
\end{cases}
\end{equation}
\begin{equation}
\text{where } \qquad \qquad \qquad \qquad \beta_{\pm}=\frac{1}{2\mu}\left(2\sqrt{2\mu}-\alpha \pm \sqrt{\left(2\sqrt{2\mu}-\alpha\right)^{2} - 4\mu}\right) \qquad \qquad \qquad 
\end{equation}
\begin{equation}
\text{and } \qquad \qquad	y_{\pm}=\frac{1}{2}\left((\alpha-\mu\beta)\beta+1 \pm \sqrt{\left((\alpha-\mu\beta)\beta+1\right)^{2}-4(1-\alpha\beta)\mu\beta^{2}}\right). \qquad
\end{equation}
Here we use the convention $(0,y_{-}]=\emptyset$, if $y_{-}<0$.
\end{lemma}

\begin{proof}
First, note that the discriminant of the quadratic function $P_2(y)$, can be written as a $4$th-order polynomial of $\beta$:
\begin{equation}\label{polynomial beta}
	P_{4}(\beta)=\mu^{2}\beta^{4}+2\mu\alpha\beta^{3} +(\alpha^{2}-6\mu)\beta^{2}+2\alpha\beta +1
\end{equation}
For all $\alpha> 0$, the polynomial $P_{4}(\beta)$ has always two negative roots $\hat{\beta}_{\pm}$, where
\begin{equation}
	\hat{\beta}_{\pm} = \frac{1}{2\mu}\left(-\alpha-2\sqrt{2\mu}\pm\sqrt{(\alpha+2\sqrt{2\mu})^{2}-4\mu}\right)
\end{equation}

In addition, if $\alpha \leq 2(\sqrt{2}-1)\sqrt{\mu}$, then the polynomial $P_{4}(\beta)$ has also two positive roots $\beta_{\pm}$ (that may collapse to a single one, when $\alpha = 2(\sqrt{2}-1)\sqrt{\mu}$), where 
\begin{equation}
	\beta_{\pm} = \frac{1}{2\mu}\left(2\sqrt{2\mu}-\alpha\pm\sqrt{(\alpha-2\sqrt{2\mu})^{2}-4\mu}\right)
\end{equation}

When $\alpha>2(\sqrt{2}-1)\sqrt{\mu}$, the polynomial $P_{4}(\beta)$ has only negative roots (and their cardinality is either $2$, if $2(\sqrt{2}-1)\sqrt{\mu}<\alpha<2(\sqrt{2}+1)\sqrt{\mu}$, or $3$ if $\alpha=2(\sqrt{2}+1)\sqrt{\mu}$, or $4$ if $\alpha>2(\sqrt{2}+1)\sqrt{\mu}$).

Therefore, if $\alpha\leq 2(\sqrt{2}-1)\sqrt{\mu}$, we have that $P_{4}(\beta)>0 $ when $\beta\in (0,\beta_{-})\cup(\beta_{+},+\infty)$ (this follows from the continuity of $P_{4}(\beta)$).
In the same way, when $\alpha > 2(\sqrt{2}-1)\sqrt{\mu}$, it follows that $P_{4}(\beta) \geq 0$, for all $\beta\geq 0$ (indeed, since all the roots are negative and $P_{4}(0)>0$, the polynomial $P_{4}(\beta)$ keeps constant sign in $[0,+\infty)$). Overall, we have the following characterization for $P_{4}(\beta)$:
\begin{equation}\label{P4bpositive}
	P_{4}(\beta) \geq 0 \Leftrightarrow \begin{cases}
		& \{\alpha \leq 2(\sqrt{2}-1)\sqrt{\mu}\} \& \{\beta\in (0,\beta_{-})\cup(\beta_{+},+\infty)\} \\
		\text{ OR } ~ &\{\alpha>2(\sqrt{2}-1)\sqrt{\mu}\} \& \{\beta>0\}\}
	\end{cases}
\end{equation}

Finally for the quadratic function $P_{2}(y)$ defined in \eqref{P2y quadratic def}, it follows that $P_{2}(y) \geq 0 $, if-f $P_{4}(\beta)\leq 0$, or $P_{4}(\beta)\geq 0$ and $y\in (0,y_{-}]\cup[\max\{0,y_{+}\},+\infty]$. This concludes the proof of Lemma \ref{lemma binomial}.
\end{proof}

%
%

\section{Convergence analysis}
\label{app:rates}

\begin{proof}[\textbf{Proof of Lemma \ref{lemma general conditions}}]
	Let $(x(t))_{t\geq 0}$ the trajectory generated by \eqref{eq:DIN}, converging to $\bar{x}$, with $(x(t))_{t\geq t_{0}} \in \Omega$, where $\Omega$ is a neighborhood of $\bar{x}$, such that  $F$ satisfies the \L ojasiewicz condition \hyperref[PL]{$\Lb(2)$} with $\mu>0$ in $\Omega$. Taking the derivative of $V$ with respect to time in \eqref{Lyapunov def} and characterizing $\ddot{x}(t)$ via \eqref{eq:DIN}, for all $t\geq t_{0}$ we obtain
	\begin{equation}\label{lyapunov basic}
		\begin{aligned}
        	\dot{V}(t) &= a\inner{\nabla F(x(t))}{\dot{x}(t)} + \beta\inner{\nabla^{2}F(x(t))\dot{x}(t)}{\dot{x}(t)+\beta\nabla F(x(t)} +\inner{\beta\nabla F(x(t)) +\dot{x}(t)}{\ddot{x}(t)} \\ 
			 &= (a-\alpha \beta -1)\inner{\nabla F(x(t))}{\dot{x}(t)} -\beta\norm{\nabla F(x(t))}^{2} -\alpha \norm{\dot{x}(t)}^{2} 
		\end{aligned}
	\end{equation}

	From the definition of $V(t)$ in \eqref{Lyapunov def}, the term $\inner{\nabla F(x(t))}{\dot{x}(t)}$, can be expressed as follows:
	\begin{equation}\label{expression of Wdot}
		\inner{\nabla F(x(t))}{\dot{x}(t)} =\frac{1}{\beta }V(t) -\frac{a}{\beta }(F(x(t))-F(\bar{x})) -\frac{\beta}{2}  \norm{\nabla F(x(t))}^{2} -\frac{1}{2\beta}\norm{\dot{x}(t)}^{2} ,
	\end{equation}
	thus by injecting \eqref{expression of Wdot} into \eqref{lyapunov basic}, we find
\begin{flalign}
\label{eq: before using PL}
\dot{V}(t)
	&= \frac{a-\alpha\beta -1}{\beta }V(t)
		- \frac{(a-\alpha\beta -1)a}{\beta }(F(x(t))-F(\bar{x}))
	\notag\\
	&\qquad - \left(\frac{(a-\alpha\beta -1)\beta}{2} +\beta \right)\norm{\nabla F(x(t))}^{2}
		-\left(\frac{a-\alpha\beta -1}{2\beta}
		+ \alpha \right)\norm{\dot{x}(t)}^{2}
	\notag\\
			&= -RV(t) +aR(F(x(t))-F(\bar{x}))-\frac{1}{2}\left(a+1-\alpha\beta \right)\beta\norm{\nabla F(x(t))}^{2}
			\notag\\
			&\qquad +\frac{1}{2\beta}\left(1-\alpha\beta -a\right)\norm{\dot{x}(t)}^{2}
	\end{flalign}
	where in the second equality we set $R = \frac{1}{\beta}\left(\alpha\beta+1-a\right)$.
	
	Since the function $F$ satisfies condition \hyperref[PL]{$\Lb(2)$}, if $R = \frac{1}{\beta}\left(\alpha\beta+1-a\right)\geq 0$, from \eqref{eq: before using PL}, it follows
	\begin{equation}\label{before conditions for exp V}
		\begin{aligned}
			\dot{V}(t) & \leq -R V(t) +\frac{1}{2\mu}\left(aR - (a+1-\alpha\beta )\mu\beta\right)\norm{\nabla F(x(t))}^{2} +\frac{1}{2\beta}\left(1-\alpha\beta -a\right)\norm{\dot{x}(t)}^{2} 
		\end{aligned}
	\end{equation}
	
	From relation \eqref{before conditions for exp V}, we note that if the algebraic expressions multiplying $\norm{\nabla F(x(t))}^{2}$ and $\norm{\dot{x}(t)}^{2}$ are non-positive, it follows that $\dot{V}(t)\leq -RV(t)$ with $R=\frac{1}{\beta}\left(\alpha\beta+1-a\right)$, which (by Gronw\"all lemma, see Lemma \ref{lemmagronwallcontinuous} in Appendix \ref{app:aux}), leads to an exponential decay for the energy $V$, with a geometric rate of order $R$, i.e. $V(t) \leq V(t_{0})e^{-R(t-t_{0})}$.
	
	The overall imposing conditions are the following:
	\begin{equation}\label{arxiko condition}
		\begin{cases}
			R = \frac{1}{\beta}\left(\alpha\beta +1 -a\right) \geq0 \\
			aR \leq \left(a+1-\alpha\beta \right)\mu\beta  \\
			1-\alpha\beta\leq a
		\end{cases}
	\end{equation}  
	A straightforward computation, shows that relation \eqref{arxiko condition} is equivalent to the following:
	\begin{equation}\label{finalconditionsforlemma1}
		\begin{cases}
			\left(\alpha\beta +1 -a - \mu\beta^{2}\right)a \leq (1-\alpha\beta)\mu\beta^{2} \\
			1-\alpha\beta \leq a \leq 1+\alpha\beta
		\end{cases}
	\end{equation}
	Therefore, if condition \eqref{finalconditionsforlemma1} holds true, then \eqref{decay of V} follows. 
	Lastly, since $b=\beta$ and $c=\beta^{2}$, from the definition of $V$ in \eqref{Lyapunov def}, it holds:
	\begin{equation}
		a(F(x(t))-F(\bar{x})) \leq V(t)=a(F(x(t))-F(\bar{x})) +\frac{1}{2}\norm{\beta\nabla F(x(t))+\dot{x}(t)}^{2} \leq V(t_{0})e^{-R(t-t_{0})}
	\end{equation}
	which allows to conclude the first part of the proof of Lemma \ref{lemma general conditions}.
	
	Finally, if we additionally assume that the second inequality in \eqref{finalconditionsforlemma1} is strict, (i.e. if $\varepsilon(a,\alpha,\beta)=a^2 - \left((\alpha-\mu\beta)\beta +1\right)a + (1-\alpha\beta)\mu\beta^{2}>0$), then from \eqref{finalconditionsforlemma1}, it follows
	\begin{equation}
		\dot{V}(t) \leq -R V(t) -\frac{\varepsilon(a,\alpha,\beta)}{2\mu\beta}\norm{\nabla F(x(t))}^{2}
	\end{equation}
	By using again Gronw\"all's lemma  (Lemma \ref{lemmagronwallcontinuous} in Appendix \ref{app:aux}), for all $t\geq t_{0}$, it follows that
	\begin{equation}
		\frac{\varepsilon(a,\alpha,\beta)}{2\mu\beta}\int_{t_{0}}^{t}e^{Rs}\norm{\nabla F(x(s))}^{2}ds \leq e^{Rt_{0}}V(t_{0})-e^{Rt}V(t)
	\end{equation}
	which concludes the proof of Lemma \ref{lemma general conditions}.
\end{proof}

\begin{proof}[\textbf{Proof of Lemma \ref{lemma lyapunov2}}]
	From the statement of Lemma \ref{lemma lyapunov2}, a direct observation shows that the  first condition of \eqref{conditions on a/d} in Lemma \ref{lemma general conditions} (i.e. $\max\{0,1-\alpha\beta\}\leq a\leq 1+\alpha\beta$) is satisfied.
	
	Moreover, the second condition in \eqref{conditions on a/d} is equivalent to the following inequality constraint:
	\begin{equation}\label{P2aconditioninproof}
		P_{2}(a):=a^{2} - \left((\alpha-\mu\beta)\beta+1\right)a +(1-\alpha\beta)\mu\beta^{2}\geq 0
	\end{equation}
	
	From Lemma \ref{lemma binomial} in Appendix \ref{app:aux}, it follows that  $P_{2}(a) \geq 0$  if and only if the following holds:
	\begin{equation}\label{conditionsquadraticintheproof}
		\begin{cases}
			\left\{\alpha \leq 2(\sqrt{2}-1)\sqrt{\mu}\right\}\&\left\{\beta\in [\beta_{-},\beta_{+}]\right\} \\
			\text{ OR } \\
			\left\{\alpha \leq 2(\sqrt{2}-1)\sqrt{\mu}\right\}\&\left\{\beta\in (0,\beta_{-}]\cup[\beta_{+},+\infty)\right\}\&\left\{a\in (0,y_{-}]\cup[\max\{0,y_{+}\},+\infty)\right\} \\
			\text{ OR } \\
			\left\{\alpha > 2(\sqrt{2}-1)\sqrt{\mu}\right\}\&\left\{a\in (0,y_{-}]\cup[\max\{0,y_{+}\},+\infty)\right\}
		\end{cases}
	\end{equation}
	where \begin{equation}
		\beta_{\pm}=\frac{1}{2\mu}\left(2\sqrt{2\mu}-\alpha \pm \sqrt{\left(2\sqrt{2\mu}-\alpha\right)^{2} - 4\mu}\right)
	\end{equation} 
	\begin{equation}
	\text{and } \qquad \qquad	y_{\pm}=\frac{1}{2}\left((\alpha-\mu\beta)\beta+1 \pm \sqrt{\left((\alpha-\mu\beta)\beta+1\right)^{2}-4(1-\alpha\beta)\mu\beta^{2}}\right).
	\end{equation}
	Therefore the condition \eqref{conditionsquadraticintheproof} together with $\max\{0,1-\alpha\beta\}\leq a\leq 1+\alpha\beta$, show that $a$ satisfy \eqref{conditions on a/d} in Lemma \ref{lemma general conditions} and allows to conclude the proof.
\end{proof}

\begin{proof}[\textbf{Proof of Lemma \ref{lemma lyapunov3}}]
	
	In order to prove the statement of Lemma \ref{lemma lyapunov3}, we start by an \textit{ad hoc} examination of all the possible cases in Lemma \ref{lemma lyapunov2} for the friction parameters $\alpha$ and $\beta$ and the Lyapunov parameter $a$. From the expression of $R(a,\alpha,\beta) = \frac{1}{\beta}\left(1+\alpha\beta-a\right)$, the value of $a$, that maximizes $R(a,\alpha,\beta)$, is equivalent of choosing the minimal feasible $a$, i.e. satisfying the conditions of Lemma \ref{lemma lyapunov2}.
	
	Thus, without loss of generality, we divide the set of all feasible couples of parameters $\Fb = \{(\alpha,\beta)\in \R_{+}^{2} ~ : ~ a \text{ satisfies Lemma } \ref{lemma lyapunov2}  \}$, into the following sets:
	\begin{align}
		\Fb_{-} &= \{(\alpha,\beta)\in\Fb  ~ : ~ \underset{}{\inf} a = 1-\alpha\beta\}  \label{def Fbmin} \\
		\Fb_{+} &= \{(\alpha,\beta)\in\Fb  ~ : ~ \underset{}{\inf} a = y_{+}\} \label{def Fbplus}
	\end{align}

	We now examine the various cases in Lemma \ref{lemma lyapunov2}:
	\begin{itemize}[wide=0pt]
		\item If $0<\alpha\leq 2(\sqrt{2}-1)\sqrt{\mu}$, then:	
		\begin{itemize}[wide=5pt]
			\item[\textbf{-}] If $\beta \in [\beta_{1},\beta_{2}]$, then, since $\frac{1}{\alpha} \geq \beta_{2} \geq \beta_{1}$, it follows that condition \eqref{condition on y 1} holds true for all $0\leq 1-\alpha\beta\leq a\leq 1+\alpha\beta$.

%
%
%
			
			\item[\textbf{-}] If $\beta\in (0,\beta_{1}]\cup[\beta_{2},+\infty)$,  then the condition in the first line of \eqref{condition on y 2} (i.e. $\max\{0,1-\alpha\beta\} \leq a \leq y_{-}$), is compatible (non-void), if-f the following set 
			\begin{equation}
				C^{1}_{-} := \{(\alpha,\beta)\in \R_{+}^{2}~:~\alpha\leq 2(\sqrt{2}-1)\sqrt\mu~,~ \beta\in (0,\beta_{1}]\cup[\beta_{2},+\infty) ~,~ \max\{0,1-\alpha\beta\} \leq y_{-} \}		
			\end{equation} is not empty. Next, we give a precise characterization of $C^{1}_{-}$:
			
			On the one hand, we have
			\begin{equation}\label{yminpositive1}
				\begin{aligned}
					y_{-}\geq 0 &\Leftrightarrow		\sqrt{\left((\alpha-\mu\beta)\beta+1\right)^2 - 4(1-\alpha\beta)\mu\beta^2} \leq -\mu\beta^{2}+\alpha\beta+1	\\
					& \Leftrightarrow 
					\beta \in \bigg(0,\min\left\{\beta_{0},\frac{1}{\alpha}\right\}\bigg]
				\end{aligned}
			\end{equation}
			where $\beta_{0}$ is the sole positive root of the quadratic function $\mu\beta^{2}-\alpha\beta-1$, i.e.  $\beta_{0}=\frac{1}{2\mu}\left(\alpha+\sqrt{\alpha^{2}+4\mu}\right)$.
			
			On the other hand, the following equivalence holds:
			\begin{equation}\label{ymincomp2}
				\begin{aligned}
					y_{-} \geq 1-\alpha\beta &\Leftrightarrow		\sqrt{\left((\alpha-\mu\beta)\beta+1\right)^2 - 4(1-\alpha\beta)\mu\beta^2} \leq -\mu\beta^{2}+3\alpha\beta-1	\\
					& \Leftrightarrow 
					\left\{ 8\beta(1-\alpha\beta)(\mu\beta-\alpha)\geq 0  \right\} \& \{\mu\beta^{2}-3\alpha\beta+1\leq 0\} \\
					& \Leftrightarrow \begin{cases}
						& \left\{\alpha\beta \leq 1\right\}\&\left\{\beta\geq    \frac{\alpha}{\mu}\right\}\&\left\{\alpha\geq  \frac{2}{3}\sqrt{\mu}\right\}\&\left\{\beta\in  [\beta_{3},\beta_{4}]\right\} \\
						\text{OR } &\left\{\alpha\beta \geq 1\right\}\&\left\{\beta\leq \frac{\alpha}{\mu}\right\}\&\left\{\alpha\geq \frac{2}{3}\sqrt{\mu}\right\}\&\left\{\beta\in [\beta_{3},\beta_{4}]\right\} 
					\end{cases}
				\end{aligned}
			\end{equation}
			where $\beta_{3}$, $\beta_{4}$ are the two non-negative roots of the quadratic function $\mu\beta^{2}-3\alpha\beta+1$ (when they exist, i.e. $\alpha>\frac{2}{3}\sqrt{\mu}$), i.e.:
			\begin{equation}\label{beta34proof}
				\beta_{3}=\frac{1}{2\mu}\left(3\alpha-\sqrt{9\alpha^{2}-4\mu}\right) ~ \text{ and } ~ \beta_{4}=\frac{1}{2\mu}\left(3\alpha+\sqrt{9\alpha^{2}-4\mu}\right)
			\end{equation}	
			
			From \eqref{yminpositive1} and \eqref{ymincomp2} it follows that
			\begin{equation}\label{compatibility general1}
				\begin{aligned}
					y_{-}\geq \max\{0,1-\alpha\beta\} & \Leftrightarrow \left\{ \frac{2}{3}\sqrt{\mu}\leq \alpha\right\}\&\left\{\beta\in\left[\max\left\{\frac{\alpha}{\mu},\beta_{3}\right\},\min\left\{\frac{1}{\alpha},\beta_{0},\beta_{4}\right\}\right]\right\} \\
					& \Leftrightarrow \left\{\frac{2}{3}\sqrt{\mu}\leq \alpha\leq \sqrt{\mu}\right\}\&\left\{\beta\in\left[\max\left\{\frac{\alpha}{\mu},\beta_{3}\right\},\min\left\{\frac{1}{\alpha},\beta_{4}\right\}\right]\right\}
				\end{aligned}
			\end{equation}
			Since $\alpha\leq 2(\sqrt{2}-1)\sqrt{\mu}$ and $\beta\in (0,\beta_{1}]\cup[\beta_{2},+\infty)$, from \eqref{compatibility general1}, it follows that
			\begin{equation}\label{compatibility case1}
				\begin{aligned}
					C_{-}^{1} &= \begin{cases}
						& \left\{\frac{2}{3}\sqrt{\mu}\leq \alpha\leq 2(\sqrt{2}-1)\sqrt{\mu}\right\}\&\left\{\beta\in \left[\max\left\{\frac{\alpha}{\mu},\beta_{3}\right\},\min\left\{\frac{1}{\alpha},\beta_{1},\beta_{4}\right\}\right]\right\}\\
						\text{OR} & \left\{\frac{2}{3}\sqrt{\mu}\leq \alpha \leq 2(\sqrt{2}-1)\sqrt{\mu}\right\} \&\left\{\left[\max\left\{\frac{\alpha}{\mu},\beta_{2},\beta_{3}\right\},\min\left\{\frac{1}{\alpha},\beta_{4}\right\}\right]\right\}			
					\end{cases} \\ 
					&=\left\{\sqrt{\frac{\mu}{2}} \leq \alpha \leq 2(\sqrt{2}-1)\sqrt{\mu}\right\} \& \left\{\beta\in \left[ \frac{\alpha}{\mu},\beta_{1}\right]\cup\left[\beta_{2},\frac{1}{\alpha}\right]\right\}		
				\end{aligned}
			\end{equation}
			Thus it follows, that if $(\alpha,\beta) \in C_{-}^{1}$, condition \eqref{condition on y 2} in Lemma \ref{lemma lyapunov2} holds true for all $0\leq 1-\alpha\beta \leq a \leq y_{-}$.
			
			Concerning the compatibility of the condition in the second line of \eqref{condition on y 2} in Lemma \ref{lemma lyapunov2} (i.e. $\max\{0,1-\alpha\beta,y_{+}\} \leq a \leq 1+\alpha$), a direct comparison shows that $\max\{0,1-\alpha\beta,y_{+}\}<1+\alpha\beta$, for all $(\alpha,\beta)\in \R^{2}_{+}$. 
			
			Next, we study the quantity $\max\{0,1-\alpha\beta,y_{+}\}$ appearing in \eqref{condition on y 2} and \eqref{condition on y 3} in Lemma \ref{lemma lyapunov2}. We start by defining the following conditions (sets) on the set of all possible couples $(\alpha,\beta)\in \R_{+}^{2}$.
			\begin{equation}\label{H1plus}
				\begin{aligned}
					H_{+}^{1}  =\left\{(\alpha,\beta)\in \R_{+}^{2}~:~ \alpha\leq 2(\sqrt{2}-1)\sqrt{\mu} ~,~ \beta\in (0,\beta_{1})\cup(\beta_{2},+\infty) ~,~ y_{+}\geq \max\{0,1-\alpha\beta\}\right\}	
				\end{aligned}
			\end{equation}
			\begin{equation}\label{H1min}
				H_{-}^{1} =  \left\{(\alpha,\beta)\in \R_{+}^{2}~:~ \alpha\leq 2(\sqrt{2}-1)\sqrt{\mu} ~,~ \beta\in (0,\beta_{1})\cup(\beta_{2},+\infty) ~,~ 1-\alpha\beta\geq \max\{0,y_{+}\}\right\}
			\end{equation}
			
			Regarding condition $H_{+}^{1}$, on the one hand it holds:
			\begin{equation}\label{ypositive1}
				\begin{aligned}
					y_{+}\geq 0 &\Leftrightarrow		\sqrt{\left((\alpha-\mu\beta)\beta+1\right)^2 - 4(1-\alpha\beta)\mu\beta^2} \geq \mu\beta^{2}-\alpha\beta-1	\\
					& \Leftrightarrow 
					\beta \in (0,\beta_{0}]\cup \left(\max\left\{\beta_{0},\frac{1}{\alpha}\right\},+\infty\right) 
				\end{aligned}
			\end{equation}
			where $\beta_{0}=\frac{1}{2\mu}\left(\alpha+\sqrt{\alpha^{2}+4\mu}\right)$. On the other hand, a straightforward computation yields:
			\begin{equation}\label{ygeq1minab}
				\begin{aligned}
					y_{+}  \geq 1-\alpha\beta & \Leftrightarrow		\sqrt{\left((\alpha-\mu\beta)\beta+1\right)^2 - 4(1-\alpha\beta)\mu\beta^2} \geq \mu\beta^{2}-3\alpha\beta+1	\\
					& \Leftrightarrow \begin{cases}
						& \left\{(\alpha,\beta)\in\R_{+}^{2} ~:~  8\beta(1-\alpha\beta)(\mu\beta-\alpha)\leq 0\right\} ~ , ~ \text{ if } \mu\beta^{2}-3\alpha\beta +1 \geq 0 \\
						\text{OR }	 &(\alpha,\beta)\in \R_{+}^{2} ~ ~ , \hspace{4.3cm}  \text{ if } \mu\beta^{2}-3\alpha\beta +1 < 0
					\end{cases}\\
					& \Leftrightarrow\begin{cases}
						& \left\{\alpha\leq \frac{2}{3}\sqrt{\mu}\right\}\&\left\{\beta\leq \min\left\{\frac{1}{\alpha}, \frac{\alpha}{\mu}\right\}\right\} \\
						\text{OR } &
						\{\alpha > \frac{2}{3}\sqrt{\mu}\}\&\left\{\beta\leq \min\left\{\frac{1}{\alpha}, \frac{\alpha}{\mu}\right\}\right\} \& \left\{\beta\in (0,\beta_{3}]\cup[\beta_{4},+\infty)\right\}   \\
						\text{OR } &
						\{\alpha \leq \frac{2}{3}\sqrt{\mu}\}\&\left\{\beta\geq \max\left\{\frac{1}{\alpha}, \frac{\alpha}{\mu}\right\}\right\}   \\
						\text{OR } &
						\{\alpha > \frac{2}{3}\sqrt{\mu}\}\&\left\{\beta\geq \max\left\{\frac{1}{\alpha}, \frac{\alpha}{\mu}\right\}\right\} \& \left\{\beta\in (0,\beta_{3}]\cup[\beta_{4},+\infty)\right\}   \\
						\text{OR } &
						\{\alpha > \frac{2}{3}\sqrt{\mu}\}\&\left\{\beta\in [\beta_{3},\beta_{4}]\right\} 
					\end{cases}
				\end{aligned}
			\end{equation}
			where $\beta_{3}$, $\beta_{4}$ are the two roots of the quadratic function $\mu\beta^{2}-3\alpha\beta+1$ (when they exist, i.e. $\alpha>\frac{2}{3}\sqrt{\mu}$):
			\begin{equation}
				\beta_{3}=\frac{1}{2\mu}\left(3\alpha-\sqrt{9\alpha^{2}-4\mu}\right) ~ \text{ and } ~ \beta_{4}=\frac{1}{2\mu}\left(3\alpha+\sqrt{9\alpha^{2}-4\mu}\right)
			\end{equation}	
			
			By relations \eqref{ypositive1} and $\eqref{ygeq1minab}$, it follows that $y_{+}\geq \max\{0,1-\alpha\beta\}$, if and only if:
			\begin{equation}\label{ygeq1minab2}
				\begin{aligned}
					\begin{cases}
						& \left\{\alpha\leq \frac{2}{3}\sqrt{\mu}\right\}\&\left\{\beta \in (0, \min\left\{\frac{1}{\alpha}, \frac{\alpha}{\mu},\beta_{0}\right\}\big]\right\} \\
						\text{OR } &
						\{\alpha > \frac{2}{3}\sqrt{\mu}\} \& \left\{\beta\in (0,\min\left\{\frac{1}{\alpha}, \frac{\alpha}{\mu},\beta_{0},\beta_{3}\right\}\big]\cup\left[\beta_{4},\min\left\{\frac{1}{\alpha}, \frac{\alpha}{\mu},\beta_{0}\right\}\right]\right\}   \\
						\text{OR } &
						\{\alpha \leq \frac{2}{3}\sqrt{\mu}\}\&\left\{\beta \in  \big(\max\left\{\frac{1}{\alpha},\frac{\alpha}{\mu}\right\},\beta_{0})\cup \big[\max\left\{\frac{1}{\alpha},\frac{\alpha}{\mu},\beta_{0}\right\},+\infty) \right\}   \\
						\text{OR } &
						\{\alpha > \frac{2}{3}\sqrt{\mu}\} \& \left\{\beta\in \left[\max\left\{\frac{1}{\alpha}, \frac{\alpha}{\mu}\right\},\min\{\beta_{0},\beta_{3}\}\right]\cup[\max\left\{\frac{1}{\alpha}, \frac{\alpha}{\mu},\beta_{4}\right\},\beta_{0})\right\}   \\
						\text{OR } &
						\{\alpha > \frac{2}{3}\sqrt{\mu}\} \& \left\{\beta\in \left[\max\left\{\frac{1}{\alpha}, \frac{\alpha}{\mu},\beta_{0}\right\},\beta_{3}\right]\cup[\max\left\{\frac{1}{\alpha}, \frac{\alpha}{\mu},\beta_{0},\beta_{4}\right\},+\infty)\right\}   \\
						\text{OR } &
						\{\alpha > \frac{2}{3}\sqrt{\mu}\}\&\left\{\beta\in [\beta_{3},\min\{\beta_{0},\beta_{4}\}]\right\}   \\
						\text{OR } &
						\{\alpha > \frac{2}{3}\sqrt{\mu}\}\&\left\{\beta\in [\max\left\{\frac{1}{\alpha},\beta_{0},\beta_{3}\right\},\beta_{4}]\right\} 
					\end{cases}
				\end{aligned}
			\end{equation}
			An \textit{ad hoc} examination of all the possible cases in \eqref{ygeq1minab2}, yields:
			\begin{equation}\label{ygeq1minab3}
				\begin{aligned}
					y_{+}\geq \max\{0,1-\alpha\beta\} &\Leftrightarrow\begin{cases}
						& \left\{\alpha\leq \frac{2}{3}\sqrt{\mu}\right\}\&\left\{\beta \in (0,\frac{\alpha}{\mu}\big]\right\} \\
						\text{OR } &
						\{\alpha > \frac{2}{3}\sqrt{\mu}\} \& \left\{\beta\in (0,\min\left\{ \frac{\alpha}{\mu},\beta_{3}\right\}\big]\right\}   \\
						\text{OR } &
						\{\alpha \leq \frac{2}{3}\sqrt{\mu}\}\&\left\{\beta \in  \big[\frac{1}{\alpha},+\infty) \right\}     \\
						\text{OR } &
						\{\alpha > \frac{2}{3}\sqrt{\mu}\} \& \left\{\beta\in \big[\max\left\{\frac{1}{\alpha},\beta_{4}\right\},+\infty)\right\}   \\
						\text{OR } &
						\{\alpha > \frac{2}{3}\sqrt{\mu}\}\&\left\{\beta\in [\beta_{3},\min\{\beta_{0},\beta_{4}\}]\right\}   \\
						\text{OR } &
						\{\alpha \geq \sqrt{\frac{\mu}{2}}\}\&\left\{\beta\in [\beta_{0},\beta_{4}]\right\} 
					\end{cases} \\
					& \Leftrightarrow\begin{cases}
						& \left\{\alpha\leq \frac{2}{3}\sqrt{\mu}\right\}\&\left\{\beta \in (0,\frac{\alpha}{\mu}\big]\cup\big[\frac{1}{\alpha},+\infty)\right\} \\
						\text{OR } &
						\left\{ \frac{2}{3}\sqrt{\mu} < \alpha <\sqrt{\frac{\mu}{2}}\right\} \& \left\{\beta\in (0, \frac{\alpha}{\mu}\big]\cup\big[\beta_{3},\beta_{4}\big]\cup\big[\frac{1}{\alpha},+\infty) \right\} \\
						\text{OR } &
						\{\alpha \geq \sqrt{\frac{\mu}{2}}\}\&\left\{\beta>0\right\} 
					\end{cases}
				\end{aligned}
			\end{equation}
			
			Since $\left\{\alpha\leq 2(\sqrt{2}-1)\sqrt{\mu}\right\}\&\left\{\beta\in (0,\beta_{1})\cup(\beta_{2},+\infty)\right\}$, from \eqref{ygeq1minab3} and the definition of $H_{+}^{1}$ in \eqref{H1plus}, it follows:
			\begin{equation}\label{H1plusfinal}
				\begin{aligned}
					H_{+}^{1} &=\begin{cases}
						& \left\{\alpha\leq \frac{2}{3}\sqrt{\mu}\right\}\&\left\{\beta \in (0,\min\left\{\frac{\alpha}{\mu},\beta_{1}\right\}\big]\cup\left[\frac{1}{\alpha},\beta_{1}\right]\cup\left[\beta_{2},\frac{\alpha}{\mu}\right]\cup\big[\max\left\{\beta_{2},\frac{1}{\alpha}\right\},+\infty)\right\} \\
						\text{OR } &
						\left\{ \frac{2}{3}\sqrt{\mu} < \alpha <\sqrt{\frac{\mu}{2}}\right\} \& \left\{\beta\in (0, \min\left\{\frac{\alpha}{\mu},\beta_{1}\right\}\big]\cup\big[\beta_{3},\min\left\{\beta_{1},\beta_{4}\right\}\big]\cup\big[\frac{1}{\alpha},\beta_{1}\big] \right\} \\
						\text{OR } &   \left\{ \frac{2}{3}\sqrt{\mu} < \alpha <\sqrt{\frac{\mu}{2}}\right\} \& \left\{\beta\in (\beta_{2}, \frac{\alpha}{\mu}\big]\cup\big[\max\{\beta_{2},\beta_{3}\},\beta_{4}\big]\cup\big[\max\left\{\beta_{2},\frac{1}{\alpha}\right\},+\infty) \right\}                                       \\
						\text{OR } & \left\{\sqrt{\frac{\mu}{2}}\leq \alpha \leq  2(\sqrt{2}-1)\sqrt{\mu}\right\}\&\left\{\beta\in(0,\beta_{1})\cup(\beta_{2}+\infty)\right\} 
					\end{cases} \\
					&= \begin{cases}
						& \left\{  \alpha <\sqrt{\frac{\mu}{2}}\right\} \& \left\{\beta\in (0, \frac{\alpha}{\mu}\big]\cup \big[\frac{1}{\alpha},+\infty) \right\} \\
						\text{OR } & \left\{\sqrt{\frac{\mu}{2}}\leq \alpha \leq  2(\sqrt{2}-1)\sqrt{\mu}\right\}\&\left\{\beta\in(0,\beta_{1})\cup(\beta_{2}+\infty)\right\} 
					\end{cases}
				\end{aligned}
			\end{equation}
			Therefore, if $(\alpha,\beta)\in H_{+}^{1}$, then condition \eqref{condition on y 2} in Lemma \ref{lemma lyapunov2} holds true for all $y_{+}\leq a\leq 1+\alpha\beta$.

			Analogously, concerning the set $H_{-}^{1}$ and the condition $1-\alpha\beta \geq \max\{0,y_{+}\}$, we have
			\begin{equation}\label{minabgeqy1}
				\begin{aligned}
					1-\alpha\beta \geq \max\{0,y_{+}\} &= \begin{cases}
						& \left\{\alpha\leq \frac{2}{3}\sqrt{\mu}\right\}\& \left\{\beta\in \left[\frac{\alpha}{\mu},\frac{1}{\alpha}\right]\right\} \\ 
						\text{OR } & \left\{\frac{2}{3}\sqrt{\mu}<\alpha \leq \sqrt{\mu} \right\}\& \left\{\beta\in \left[\frac{\alpha}{\mu},\min\left\{\frac{1}{\alpha},\beta_{3}\right\}\right]\cup \left[\max\left\{\frac{\alpha}{\mu},\beta_{4}\right\},\frac{1}{\alpha}\right]\right\} \\
					\end{cases} \\
					&=
					\begin{cases}
						& \left\{\alpha\leq \frac{2}{3}\sqrt{\mu}\right\}\& \left\{\beta\in \left[\frac{\alpha}{\mu},\frac{1}{\alpha}\right]\right\} \\ 
						\text{OR } & \left\{\frac{2}{3}\sqrt{\mu}<\alpha \leq \sqrt{\frac{\mu}{2}} \right\}\& \left\{\beta\in \left[\frac{\alpha}{\mu},\beta_{3}\right]\cup \left[\beta_{4},\frac{1}{\alpha}\right]\right\} \\
					\end{cases}
				\end{aligned}
			\end{equation}	
			Thus, by using \eqref{minabgeqy1}, $H_{-}^{1}$ can be expressed as follows:
			\begin{flalign}\label{H1minfinal}
					H_{-}^{1} &= 
					\begin{cases}
						& \left\{\alpha\leq \frac{2}{3}\sqrt{\mu}\right\}\& \left\{\beta\in \left[\frac{\alpha}{\mu},\min\left\{\frac{1}{\alpha},\beta_{1}\right\}\right]\cup\left[\max\left\{\frac{\alpha}{\mu},\beta_{2}\right\},\frac{1}{\alpha}\right]\right\} \\ 
						\text{OR } & \left\{\frac{2}{3}\sqrt{\mu}<\alpha \leq \sqrt{\frac{\mu}{2}} \right\}\& \left\{\beta\in \left[\frac{\alpha}{\mu},\min\{\beta_{1},\beta_{3}\}\right]\cup \left[\beta_{4},\min\{\beta_{1},\frac{1}{\alpha}\}\right]\right\} \\
						\text{OR } & \left\{\frac{2}{3}\sqrt{\mu}<\alpha \leq \sqrt{\frac{\mu}{2}} \right\}\& \left\{\beta\in \left[\max\left\{\frac{\alpha}{\mu},\beta_{2}\right\},\beta_{3}\right]\cup \left[\max\{\beta_{2},\beta_{4}\},\frac{1}{\alpha}\right]\right\}
					\end{cases}
					\notag\\
					&= 
					\left\{\alpha \leq \sqrt{\frac{\mu}{2}} \right\}\& \left\{\beta\in \left[\frac{\alpha}{\mu},\beta_{1}\right]\cup\left[\beta_{2},\frac{1}{\alpha}\right]\right\} 
			\end{flalign}
			which shows that  if $(\alpha,\beta)\in H_{-}^{1}$, then condition \eqref{condition on y 2} in Lemma \ref{lemma lyapunov2} holds true for all $0\leq 1-\alpha\beta\leq a\leq 1+\alpha\beta$.
		\end{itemize}
		
		\item If $\alpha>2(\sqrt{2}-1)\sqrt{\mu}$, then:
		
		Following the previous case,  the condition in the first line of \eqref{condition on y 3} (i.e. $\max\{0,1-\alpha\beta\} \leq a \leq y_{-}$), is compatible, if and only if the following set 
		\begin{equation}
			C^{2}_{-} := \{(\alpha,\beta)\in \R_{+}^{2}~:~\alpha> 2(\sqrt{2}-1)\sqrt\mu ~,~ \max\{0,1-\alpha\beta\} \leq y_{-} \}		
		\end{equation} is not empty.
		From \eqref{compatibility general1} and the definition of $C^{2}_{-}$, it follows that:
		\begin{equation}\label{compatibility case 2}
			\begin{aligned}
				C^{2}_{-}&=\left\{2(\sqrt{2}-1)\sqrt{\mu}<\alpha\leq \sqrt{\mu}\right\} \& \left\{\beta\in\left[\max\left\{\frac{\alpha}{\mu},\beta_{3}\right\},\min\left\{\frac{1}{\alpha},\beta_{4}\right\}\right]\right\} \\ &=  \left\{2(\sqrt{2}-1)\sqrt{\mu}<\alpha\leq \sqrt{\mu}\right\} \& \left\{\beta\in \left[ \frac{\alpha}{\mu}, \frac{1}{\alpha} \right]\right\}
			\end{aligned}
		\end{equation}
		Thus, if $(\alpha,\beta) \in C_{-}^{2}$, condition \eqref{condition on y 3} in Lemma \ref{lemma lyapunov2} holds true for all $0\leq 1-\alpha\beta \leq a \leq y_{-}$.
		
		In addition, regarding the quantity $\max\{0,1-\alpha\beta,y_{+}\}$ appearing in \eqref{condition on y 2} and \eqref{condition on y 3} in Lemma \ref{lemma lyapunov2}, in the same way as before, we define the sets:
		\begin{equation}\label{H2plus}
			H_{+}^{2} =  \left\{(\alpha,\beta)\in \R_{+}^{2}~:~ \alpha> 2(\sqrt{2}-1)\sqrt{\mu} ~,~ y_{+}\geq \max\{0,1-\alpha\beta\}\right\}
		\end{equation}
		\begin{equation}\label{H2min}
			H_{-}^{2} =   \left\{(\alpha,\beta)\in \R_{+}^{2}~:~ \alpha> 2(\sqrt{2}-1)\sqrt{\mu} ~,~ 1-\alpha\beta\geq \max\{0,y_{+}\}\right\}
		\end{equation}
		
		Following the previous case, relation  \eqref{ygeq1minab3}, together with $\alpha> 2(\sqrt{2}-1)\sqrt{\mu}$, yields:
		\begin{equation}\label{Hplus2 final}
			H_{+}^{2} = \{\alpha > 2(\sqrt{2}-1)\sqrt{\mu}\}\&\left\{\beta>0\right\} 
		\end{equation}
		while from $\eqref{minabgeqy1}$ and $\alpha> 2(\sqrt{2}-1)\sqrt{\mu}$, it follows that $H_{-}^{2}=\emptyset$.
		Therefore, condition \eqref{condition on y 3} in Lemma \ref{lemma lyapunov2} holds true for all $y_{+}\leq a\leq 1+\alpha\beta$.
	\end{itemize}
    
	Next, we give a precise characterization of the sets $\Fb_{-}$ and $\Fb_{+}$, as defined in \eqref{def Fbmin} and \eqref{def Fbplus}, by grouping all the possible conditions for the couple $(\alpha,\beta)\in \Fb$, as found in the cases above.

	From the characterization of $H_{+}^{1}$ and $H_{+}^{2}$ in \eqref{H1plusfinal} and \eqref{Hplus2 final} (respectively), it follows that:
	\begin{equation}
		\Fb_{+}=H_{+}^{1}\cup H_{+}^{2} = 
		\begin{cases}
			& \left\{  \alpha <\sqrt{\frac{\mu}{2}}\right\} \& \left\{\beta\in (0, \frac{\alpha}{\mu}\big]\cup \big[\frac{1}{\alpha},+\infty) \right\} \\
			\text{OR } & \left\{\sqrt{\frac{\mu}{2}}\leq \alpha \leq  2(\sqrt{2}-1)\sqrt{\mu}\right\}\&\left\{\beta\in(0,\beta_{1})\cup(\beta_{2}+\infty)\right\}  \\
			\text{OR } & \{\alpha > 2(\sqrt{2}-1)\sqrt{\mu}\}\&\left\{\beta>0\right\} 
		\end{cases}
	\end{equation}
	
For sake of simplicity and readability of the main results, instead of the condition $\Fb_{+}$, we will consider the slightly stronger condition $\Fb_{++}$:
\begin{equation}\label{final Fb++}
		\Fb_{++} = \left\{\alpha>0\right\}\&\left\{\beta\in \left(0, \frac{\alpha}{\mu}\right)\cup \left(\frac{1}{\alpha},+\infty\right) \right\}
\end{equation} 
Indeed, since $\frac{\alpha}{\mu} \leq \beta_{1}$ and $\beta_{2}\leq \frac{1}{\alpha}$, for all  $\alpha\in\left[\sqrt{\frac{\mu}{2}},2(\sqrt{2}-1)\sqrt{\mu}\right]$, it follows that $\Fb_{++} \subset \Fb_{+}$.

Analogously, from the expressions of the compatibility conditions $C_{-}^{1}$ and $C_{-}^{2}$ in \eqref{compatibility case1} \eqref{compatibility case 2} (respectively) and the characterization of $H_{-}^{1}$, as also the case \eqref{condition on y 1} in \eqref{H1minfinal},  we obtain:
\begin{equation}\label{final Fbmin}
	\begin{aligned}
		\Fb_{-} &= H_{-}^{1}\cup C_{-}^{1}\cup C_{-}^{2} \cup \left\{\left\{\alpha\leq 2(\sqrt{2}-1)\sqrt\mu\right\}\&\left\{\beta\in[\beta_{1},\beta_{2}]\right\}\right\} \\
		&=\begin{cases}
			& \left\{\alpha \leq \sqrt{\frac{\mu}{2}} \right\}\& \left\{\beta\in \left[\frac{\alpha}{\mu},\beta_{1}\right]\cup\left[\beta_{2},\frac{1}{\alpha}\right]\right\} \\
			\text{OR }	&\left\{\sqrt{\frac{\mu}{2}} \leq \alpha \leq 2(\sqrt{2}-1)\sqrt{\mu}\right\} \& \left\{\beta\in \left[  \frac{\alpha}{\mu},\beta_{1}\right]\cup\left[\beta_{2},\frac{1}{\alpha}\right]\right\} \\
			\text{OR } & \left\{2(\sqrt{2}-1)\sqrt{\mu}< \alpha\leq \sqrt{\mu}\right\} \& \left\{\beta\in \left[ \frac{\alpha}{\mu}, \frac{1}{\alpha} \right]\right\} \\
			\text{OR }	& \left\{\alpha\leq 2(\sqrt{2}-1)\sqrt\mu\right\}\&\left\{\beta\in[\beta_{1},\beta_{2}]\right\}
		\end{cases}		
		\\
		&= \left\{ \alpha\leq \sqrt{\mu}\right\} \& \left\{\beta\in \left[ \frac{\alpha}{\mu}, \frac{1}{\alpha} \right]\right\}
	\end{aligned}
\end{equation}	Finally, from the expressions of $\Fb_{++}$ and $\Fb_{-}$ in  \eqref{final Fb++} and \eqref{final Fbmin} (respectively),  it follows:
\begin{equation}
	\begin{aligned}
		\underset{a\geq0}{\max}R(a,\alpha,\beta) &=\begin{cases}
			2\alpha &~ \text{ if } ~  (\alpha,\beta) \in \Fb_{-}\\
			\frac{1}{2\beta}\left((\alpha+\mu\beta)\beta+1 - \sqrt{\left((\alpha-\mu\beta)\beta+1\right)^{2}-4(1-\alpha\beta)\mu\beta^{2}}\right) &~ \text{ if } ~  (\alpha,\beta) \in \Fb_{++}
		\end{cases}
	\end{aligned}
\end{equation}
\begin{equation}
\text{and } \qquad	\begin{aligned}
		\underset{a\geq 0}{\argmax} R(a,\alpha,\beta)= \begin{cases}
			1-\alpha\beta \qquad \qquad  ~ \text{ if } ~   (\alpha,\beta) \in \Fb_{-} \\
			\frac{1}{2}\left((\alpha-\mu\beta)\beta+1 + \sqrt{\left((\alpha-\mu\beta)\beta+1\right)^{2}-4(1-\alpha\beta)\mu\beta^{2}}\right)  & \\ 
			\qquad \qquad \qquad \quad   ~ \text{ if } ~  (\alpha,\beta) \in \Fb_{++} &
		\end{cases}
	\end{aligned}
\end{equation}
which concludes the proof of Lemma \ref{lemma lyapunov3}.
\end{proof}

\section{Avoidance analysis: proof of lemma \ref{lemma:avoidstrong}}
\label{app:avoidance}


In this appendix we aim to prove a lemma for the avoidance of fixed points with strong escape directions for autonomous ODEs.
Specifically, we aim to provide guarantees for the avoidance of such fixed points under the assumption of convergence, which is implied by the \textit{\L ojasiewicz} for the system \eqref{eq:DIN} by proposition \ref{proposition basic}. 
Since the heart of the argument of proofs of avoidance of saddle points is as far as we know singular \cite{PdM82,castera2023inertial,maulen2024inertial}, we aim to automate that argument, through corollary \ref{cor:avoidstrong}, which we then use to obtain the result in our case. 
Corollary \ref{cor:avoidstrong} is formulated in terms of general $d$-dimensional autonomous ODEs. 
This abstraction significantly simplifies the presentation. 
To this end, consider, as in \eqref{eq:ODE},
\begin{align}
\label{eq:ode}
	\dot{z} &= f(z) 
\end{align}
and let $\phi_t\in\mathbb{R}^{d}\to\mathbb{R}^{d}$ be its time evolution over time $t$.
Under suitable conditions the Center Manifold Theorem describes the dynamics of this system near an equilibrium point, i.e. a point $\bar{z}\in\mathbb{R}^d$ such that $f(\bar{z})=0$. Specifically it maps its time evolution to the time evolution evolution of its linearized version
\begin{align}
\dot{z} &= \nabla f (\bar{z}) (z-\bar{z})
\end{align}
A particularly useful version of the CMT is Carr's Theorem \cite{perko2013differential}, presented below for the reader's convenience. 

\begin{theorem}[Carr's Theorem (Theorem 2 in Section 2.12 of \cite{perko2013differential})]
\label{thrm:Carr}
Suppose $f \in C^{1}(E)$ where $E$ is an open subset of $R^{d}$ containing an equilibrium point $\bar{z}$.
Consider $\nabla f(\bar{z})$ in block diagonal form
\begin{align}
		\nabla f(\bar{z}) \equiv 
		 \begin{pmatrix}
			 C & & \\
			 & S & \\
			 & & U
		 \end{pmatrix}
	\end{align}
	with $C\in\mathbb{R}^{m\times m}$ having $m$ eigenvalues with zero real part, $U\in\mathbb{R}^{n\times n}$ having $n$ eigenvalues with positive real part and $S\in\mathbb{R}^{k\times k}$ having $k$ eigenvalues with negative real part, such that $m+n+k=d$. By a linear transformation taking $\bar{z}$ to $0$, the system can be written in the form 
	\begin{align}
	\dot{z}_c &= Cz_c + F_c(z_c,z_u,z_s)
	\\
	\dot{z}_s &= Sz_s + F_s(z_c,z_u,z_s)
	\\
	\dot{z}_u &= Uz_u + F_u(z_c,z_u,z_s)
	\end{align}
	with each $F$ continuous from $\mathbb{R}^{d}$ to its image and such that $F(0)=0$ and $\nabla F(0)=0$. 
    Notice that any of the integers $m$, $n$, $k$ can be zero and then the respective part above will not be necessary. 
	Then there exist $C^1$ functions $h_1$ and $h_2$ and a neighborhood of $(\bar{x},0)$ where the system (1) is topologically conjugate to 
	\begin{align}
    	\dot{z}_c &= Cz_c + F_c(z_c,h_1(z_c),h_2(z_c))
        \label{eq:linearized_c}
    	\\
    	\dot{z}_s &= Sz_s
        \label{eq:linearized_s}
    	\\
    	\dot{z}_u &= Uz_u 
        \label{eq:linearized_u}
	\end{align}
	near $(0,0,0)$.
\end{theorem}

From \ref{thrm:Carr} we can get the following Corollary, commonly used (implicitly) in the literature to prove avoidance of saddle points under the assumption of (strong) convergence \cite{castera2023inertial,maulen2024inertial}
\begin{corollary}
\label{cor:avoidstrong}
	Suppose $f \in C^{1}(E)$ where $E$ is an open subset of $R^{d}$ containing an equilibrium point $\bar{z}$ and $\nabla f(\bar{z})$ has a negative eigenvalue. Then there exists a neighborhood of $\bar{z}$, $E_{\bar{z}}$, such that the manifold of points that under time evolution forever remain in $E_{\bar{z}}$ is included in a $(d-1)$-dimensional (topological) submanifold of $\mathbb{R}^{d}$, say $\mathcal{V}_{\bar{z}}$.
\end{corollary}
\begin{proof}
    Let $\bar{\phi}_t \in \mathbb{R}^d\to\mathbb{R}^d$ be the time evolution jointly under \eqref{eq:linearized_c}, \eqref{eq:linearized_s} and \eqref{eq:linearized_u}.
    By \ref{thrm:Carr} there is a (bounded) neighborhood of $\bar{z}$, $E_{\bar{z}}$, and a homeomorphism $q$ from $E_{\bar{z}}$ to $q(E_{\bar{z}})$ such that $q(\bar{z})=0$ and as long as the trajectory $\phi_t(z_0)$ remains in $E$, we have $q(\phi_t(z_0)) = \bar{\phi}_t(q(z_0))$. 
    Let $q_u(z)$ be the projection of $q(z)$ to the last $n$ coordinates and let $\lambda$ be the smallest eigenvalue of $U$. 
    Notice that $U$ exists, since we assumed that $\nabla f (\bar{z})$ has a positive eigenvalue, and by the definition of $U$, $\lambda$ must be positive. 
    Then, using $q_u$ we can see that
    \begin{align}
        \| q(\phi_t(z_0)) \| &\geq \| q_u(\phi_t(z_0)) \| 
        = \| e^{Ut} q_u(z_0) \| 
        \geq e^{\lambda t} \| q_u(z_0) \| 
    \end{align}
    which becomes arbitrarily large for large $t$, as long as $q_u(z_0) \neq 0$. 
    It is then clear that the set of initial conditions $z_0 \in E_{\bar{z}}$ for which the trajectory $\phi_t(z_0)$ remains in $E_{\bar{z}}$ for all $t$ is a subset of $\mathcal{V}_{\bar{z}} = \{z_0 \in E_{\bar{z}} : q_u(z_0) = 0\}$, which is a topological manifold of dimension $d-n \leq d-1$. 
\end{proof}

Define the set of strict saddle points through 
\begin{equation} 
	\mathcal{S} = \{\bar{z}\in\mathbb{R}^{d}: f(\bar{z})=0 \ \& \ \nabla^2 F(\bar{z}) \ \text{has a negative eigenvalue} \}	
\end{equation} 
The set of points that belong to trajectories attracted by strict saddle points is defined as $\mathcal{M}_s=\{z\in\mathbb{R}^d:\lim_t \phi_t(z) \in \mathcal{S}\}$. 

Corollary \ref{cor:avoidstrong} can be used in conjunction with (strong) convergence to provide almost sure avoidance of saddle points

\begin{proof}[\textbf{Proof of Lemma \ref{lemma:avoidstrong}}]
    For each $z \in \mathcal{S}$, Corollary \ref{cor:avoidstrong} provides a neighborhood with desirable properties. These neighborhoods clearly form an open covering of $\mathcal{S}$. By second countability of $\mathbb{R}^d$, we have that there exists a countable sequence of them that is enough to provide a covering of $\mathcal{S}$, say $\{E_n\}_{n=1}^\infty$ with invariant submanifolds $\{\mathcal{V}_n\}_{n=1}^\infty$. Since any initial condition $z_0$ that converges to a point in $\mathcal{S}$ must eventually end up inside $E_n$ for some $n$, it must then eventually end up in $\mathcal{V}_n$, i.e. $z_0$ belongs to $\phi_{-t}\left( \mathcal{V}_n \right)$ for all $t$ large enough. This last set is of measure zero by the local Lipschitz continuity of $\phi_{t}$ (see \cite[Theorem $2.8$]{teschl2012ordinary}). 
	Thus we have
	\begin{align}
	\mathcal{M}_s \subset \cup_{t\in\mathbb{N}} \cup_{n\in\mathbb{N}} \phi_{-t}\left( \mathcal{V}_n \right)
	\end{align}
	and as a subset of the countable union of measure zero sets, $\mathcal{M}_s$ must have measure zero. 
\end{proof}

\normalsize
\bibliographystyle{siam}
\bibliography{bibtex/Bibliography-PM.bib,bibtex/reference2.bib}

\end{document}